\documentclass[preprint,
3p,
times loads txfonts
]{elsarticle}

\usepackage{amsmath}
\usepackage{amsthm}
\usepackage{cases}
\usepackage{enumitem}              
\usepackage{mathrsfs}
\usepackage{mathtools}
\usepackage[OMLmathsfit]{isomath}  
\usepackage{graphicx}
\usepackage{grffile}               
\usepackage[T1]{fontenc}
\usepackage{lmodern}

\usepackage{xcolor}

\usepackage[cmintegrals]{newpxmath}  

\usepackage{bm}                      

\usepackage{dsfont}                                              
  \newcommand{\IQ}{\ensuremath\mathds{Q}}                        
  \newcommand{\IR}{\ensuremath\mathds{R}}                        
  \newcommand{\IP}{\ensuremath\mathds{P}}                        

\newcommand*{\setE}{\ensuremath{\mathcal{T}}}                    
\newcommand*{\Gammah}{\Gamma_h}                                  
\newcommand*{\Nloc}{{N_\mathrm{loc}}}                            
\renewcommand*{\vec}[1]{{\boldsymbol{#1}}}                       
\DeclareMathAlphabet{\mathbfsf}{\encodingdefault}{\sfdefault}{bx}{n}
\newcommand*{\vecc}[1]{\mathbfsf{#1}}                            
\newcommand*{\transpose}[1]{{#1}^\mathrm{T}}                     
\newcommand*{\normal}{\vec{n}}                                   
\newcommand*{\dd}{\mathrm{d}}                                    
\newcommand*{\grad}{\vec{\nabla}}                                
\renewcommand*{\div}{\vec{\nabla}\cdot}                          
\newcommand*{\laplace}{\upDelta}                                 
\newcommand*{\strain}{{\boldsymbol{\varepsilon}}}                
\newcommand*{\llbrace}{\lbrace\hspace*{-0.18em}\vert}
\newcommand*{\rrbrace}{\vert\hspace*{-0.18em}\rbrace}
\newcommand*{\avg}[1]{\llbrace{#1}\rrbrace}                      
\newcommand*{\jump}[1]{\left\llbracket{#1}\right\rrbracket}      
\newcommand*{\abs}[1]{\ensuremath{|#1|}}                         
\newcommand*{\norm}[2]{\|#1\|_{#2}}                              
\newcommand*{\on}[2]{\left.#1\right\vert_{#2}}                   
\newcommand\tthash{\text{\ttfamily\#}}                           

\newcommand*{\Rey}{\mathrm{Re}}                                  
    
\usepackage{multirow}
\usepackage{tabularx}                                            
  \newcolumntype{R}{>{\raggedleft\arraybackslash}X}
  \newcolumntype{L}{>{\raggedright\arraybackslash}X}
  \newcolumntype{C}{>{\centering\arraybackslash}X}
\usepackage{booktabs}                                            
\usepackage{rotating}                                            
          
\newtheorem{lemma}{Lemma}
\newtheorem{theorem}{Theorem}
\newtheorem{remark}{Remark}

\journal{ }

\begin{document}

\begin{frontmatter}
\title{An optimization based limiter for enforcing positivity in a semi-implicit discontinuous Galerkin scheme for compressible Navier--Stokes equations
}
\author[label1]{Chen Liu}\ead{liu3373@purdue.edu}
\author[label1]{Gregery T. Buzzard}\ead{buzzard@purdue.edu}
\author[label1]{Xiangxiong Zhang\corref{cor1}}\ead{zhan1966@purdue.edu}
\address[label1]{Department of Mathematics, Purdue University, 150 North University Street, West Lafayette, Indiana 47907.} 

\begin{abstract}
We consider an optimization-based limiter for enforcing positivity of internal energy in a semi-implicit scheme for solving gas dynamics equations.
With Strang splitting, the  compressible Navier--Stokes system is split into the compressible Euler equations, which are solved by the positivity-preserving Runge--Kutta discontinuous Galerkin (DG) method, and the parabolic subproblem, which is solved by Crank--Nicolson in time with interior penalty DG method. 
Such a scheme is at most second order accurate in time,  high order accurate in space, conservative, and preserves positivity of density. 
To further enforce the positivity of internal energy, 
we impose an optimization-based limiter for the total energy variable to post-process DG polynomial cell averages.  
The optimization-based limiter can be efficiently implemented by the popular first order convex optimization algorithms such as the Douglas--Rachford splitting method by using nearly optimal algorithm parameters. 
Numerical tests suggest that the DG method with $\IQ^k$ basis and the optimization-based limiter is robust  for demanding   low-pressure problems such as high-speed flows.
\end{abstract}

\begin{keyword}
compressible Navier--Stokes \sep semi-implicit \sep discontinuous Galerkin \sep high order accuracy  \sep positivity-preserving \sep Douglas--Rachford splitting \sep optimization based limiter

\vspace{.5\baselineskip}
\MSC   65M12 \sep 65M60 \sep 65N30 \sep 90C25
\end{keyword}
\end{frontmatter}


\section{Introduction}

\subsection{Motivation and objective}
For studying viscous gas dynamics, the dimensionless  compressible Navier--Stokes (NS) equations  
without external forces in conservative form on a bounded spatial domain $\Omega\subset\IR^d$ over time interval $[0,T]$ are  
\begin{align}\label{eq:CNS:model_conv_form}
\partial_t{\vec{U}} + \div{\vec{F}^\mathrm{a}} = \div{\vec{F}^\mathrm{d}}, \quad \vec{F}^\mathrm{a} = \begin{pmatrix} 
\rho\vec{u} \\
\rho\vec{u}\otimes\vec{u} + p\vecc{I} \\
(E+p)\vec{u}
\end{pmatrix} 
\quad~\text{and}~\quad
\vec{F}^\mathrm{d} = \frac{1}{\Rey}\begin{pmatrix} 
\vec{0} \\
\vec{\tau} \\
\vec{u}\cdot\vec{\tau} - \vec{q}
\end{pmatrix},
\end{align}
where the conservative variables are density $\rho$, momentum $\vec{m}$, and total energy $E$, 
 $\Rey$ denotes the Reynolds number and $\vecc{I}\in\IR^{d\times d}$ denotes an identity matrix, $\vec{u}=\frac{\vec{m}}{\rho}$ is velocity and $p$ is pressure. 
With the Stokes hypothesis, the shear stress tensor is given by $\vec{\tau}(\vec{u}) = 2\strain{(\vec{u})} - \frac{2}{3}(\div{\vec{u}})\vecc{I}$, where $\strain{(\vec{u})}=\frac{1}{2}(\grad{\vec{u}}+\transpose{(\grad{\vec{u}})})$. 
The total energy can be expressed as $E = \rho e + \frac{\norm{\vec{m}}{}^2}{2\rho}$, where $e$ denotes the internal energy and  $\norm{\cdot}{}$ is the vector 2-norm. With Fourier's heat conduction law, the heat diffusion flux $\vec{q} = -\lambda\grad{e}$ with parameters $\lambda = \frac{\gamma}{\Pr}>0$, where the positive constant $\gamma$ is the ratio of specific heats and $\Pr$ denotes the Prandtl number. For air, we have $\gamma = 1.4$ and $\Pr = 0.72$.
For simplicity,   we only consider the ideal gas equation of state 
\begin{equation}
    \label{eos}p = (\gamma-1)\rho e.
\end{equation}
The system \eqref{eq:CNS:model_conv_form} can be written as
\begin{subequations}\label{eq:CNS:model}
\begin{align}
\partial_t{\rho} + \div{(\rho\vec{u})} = 0 && \text{in}~[0,T]\times\Omega,\label{eq:CNS:model_1}\\
\partial_t{(\rho\vec{u})} + \div{(\rho\vec{u}\otimes\vec{u})} + \grad{p} - {\textstyle\frac{1}{\Rey}}\div{\vec{\tau}(\vec{u})} = \vec{0} && \text{in}~[0,T]\times\Omega,\\
\partial_t{E} + \div{((E+p) \vec{u})} - {\textstyle\frac{\lambda}{\Rey}}\laplace{e} - {\textstyle\frac{1}{\Rey}}\div{(\vec{\tau}(\vec{u})\vec{u})} = 0 && \text{in}~[0,T]\times\Omega.\label{eq:CNS:model_3}
\end{align}
\end{subequations}

 When vacuums occur, the solutions of compressible NS equations may lose continuous dependency with respect to the initial data, see   \cite[Theorem~2]{hoff1991failure} and  \cite[Remark~3.3]{guermond2021second}. On the other hand, the density and internal energy of a physically meaningful solution in most applications should both be positive. For problems without any vaccum,
define the set of admissible states as
\begin{align*}
G = \{\vec{U} = \transpose{[\rho, \vec{m}, E]}\!:~ \rho>0,~ \rho e(\vec{U}) = E - \frac{\norm{\vec{m}}{}^2}{2\rho} > 0\}.
\end{align*}
The function $\rho e(\vec{U}) = E - \frac{\norm{\vec{m}}{}^2}{2\rho}$ is a concave function of $\vec{U}$, which implies the set $G$ is convex \cite{zhang2010positivity}.
For an initial condition $\vec{U}^0=\transpose{[\rho^0,\vec{m}^0,E^0]}\in G$,  a numerical solution preserving the positivity is preferred for the sake of not only physical meaningfulness but also numerical robustness. 
For the equation of state \eqref{eos}, negative internal energy means negative pressure, with which 
the linearized compressible Euler equation loses hyperbolicity and its initial value problem is ill-posed \cite{zhang2010positivity}. On the other hand, a conservative and positivity-preserving scheme in the sense of preserving the invariant domain $G$ is numerically robust \cite{grapsas2016unconditionally, zhang2017positivity,guermond2021second,fan2022positivity,liu2023positivity}.
\par

For solving a convection-diffusion system \eqref{eq:CNS:model}, fully explicit time stepping results in a time step constraint $\Delta{t} = \mathcal{O}(\Rey\Delta{x}^2)$, thus is suitable only for high Reynolds number flows in practice. In order to achieve larger time steps such as a hyperbolic CFL $\Delta{t} = \mathcal{O}(\Delta{x})$, a semi-implicit scheme can be used \cite{guermond2021second, liu2023positivity}.
\par

The objective of this paper is to construct a high order accurate in space, conservative, and positivity-preserving scheme for solving the compressible NS equations \eqref{eq:CNS:model}. In particular, we will use the Strang splitting approach in \cite{guermond2021second, liu2023positivity} with arbitrarily high order discontinuous Galerkin (DG)  method for spatial discretization, which gives a scheme of at most second order accuracy in time. 
In general, a scheme that is high order in both time and space is preferred. On the other hand,
for many fluid problems including gas dynamics problems, the solutions are often smoother   with respect to the time variable, thus the spatial resolution of a numerical scheme is often more crucial for capturing fine structures in solutions than its temporal  accuracy. Higher order spatial discretizations often produce better numerical solutions even if the time accuracy is only first order for various convection-diffusion problems \cite{shen2021discrete,hu2023positivity,liu2024structure,liu2023positivity}.
 
\subsection{Existing positivity-preserving schemes for compressible NS equations}

In the literature, there are many positivity-preserving schemes for compressible Euler equations, which have been well studied since 1990s. For compressible Navier--Stokes equations, most of the practical positivity-preserving schemes  were developed only in the past decade.  

Grapas et al. in \cite{grapsas2016unconditionally} constructed a fully implicit  pressure correction scheme on staggered grids, which is at most second order in space, conservative, and unconditionally positivity-preserving. Nonlinear systems must be solved for time marching. As a fully implicit scheme on a staggered grid, it seems difficult to extend it to a higher order accurate scheme.

Zhang in \cite{zhang2017positivity} proposed  a simple nonlinear diffusion numerical flux, with which arbitrarily high order Runge--Kutta DG schemes solving \eqref{eq:CNS:model} can be rendered positivity-preserving without losing conservation and accuracy by a simple positivity-preserving limiter in \cite{zhang2010positivity}. 
The advantages of such a fully explicit approach include easy extensions to general shear stress models and heat fluxes, and possible extensions to other types of schemes, such as high order finite volume schemes \cite{fan2021positivity} and the high order finite difference WENO (weighted essentially nonoscillatory) schemes \cite{fan2022positivity}.
However, like many fully explicit schemes for convection-diffusion systems \cite{ZLS2012maximum,chen2016third,srinivasan2018positivity,sun2018discontinuous}, the time step constraint is $\Delta t = \mathcal{O}(\Rey\,\Delta x^2)$.

Guermond et al. in \cite{guermond2021second} introduced a semi-implicit continuous finite element scheme via Strang splitting, which preserves positivity under standard hyperbolic CFL condition $\Delta t = \mathcal{O}(\Delta x)$.
By the same operator splitting approach, in \cite{liu2023positivity} we constructed a semi-implicit conservative DG scheme, with the continuous finite element method for  solving \eqref{eq:CNS:model}, and  the scheme with $\IQ^k$ ($k=1,2,3$) basis can be proven positivity-preserving with $\Delta t = \mathcal{O}(\Delta x)$. 

\par
The early pioneering work on DG methods for solving compressible NS equations was conducted by Bassi and Rebay \cite{bassi1997high,bassi2002numerical} as well as Baumann and Oden \cite{baumann1999discontinuous}.
Advantages of DG methods include high order accuracy, flexibility in handling complex meshes and hp-adaptivity, and highly parallelizable characteristics. See \cite{cockburn2012discontinuous,shu2014discontinuous,arnold2002unified} for an overview of DG methods. 
In this paper, we focus on constructing DG schemes within the Strang splitting approach, by which the compressible NS system \eqref{eq:CNS:model} is splitted into a hyperbolic subproblem ($\mathrm{H}$) and a parabolic subproblem ($\mathrm{P}$),  representing two asymptotic regimes:  the vanishing viscosity limit (the compressible Euler equations) and the dominance of diffusive terms: 
\begin{align}
(\mathrm{H})~\!
\begin{cases}
\partial_t{\rho} + \div{(\rho\vec{u})} = 0, \\
\partial_t(\rho\vec{\vec{u}}) + \div{(\rho\vec{u}\otimes\vec{u} + p\vecc{I})} = \vec{0}, \\
\partial_t{E} + \div{((E+p)\vec{u})} = 0,
\end{cases}
&&
(\mathrm{P})~\!
\begin{cases}
\partial_t{\rho} = 0, \\
\partial_t(\rho\vec{\vec{u}}) - {\textstyle\frac{1}{\Rey}}\div{\vec{\tau}(\vec{u})} = \vec{0}, \\
\partial_t{E} -{\textstyle\frac{\lambda}{\Rey}}\laplace{e} - {\textstyle\frac{1}{\Rey}}\div{(\vec{\tau}(\vec{u})\vec{u})} = 0.
\end{cases}
\label{strang-splitting}
\end{align}
The equation $\partial_t{\rho} = 0$ in the parabolic subproblem implies the variable $\rho$ in ($\mathrm{P}$) is time independent. Multiplying the second equation in $(\mathrm{P})$ by $\vec{u}$ and using the identity 
$\div{(\vec{\tau}(\vec{u})\vec{u})} = (\div{\vec{\tau}(\vec{u})})\cdot\vec{u} + \vec{\tau}(\vec{u}):\grad{\vec{u}}$, 
we obtain the following equivalent system in non-conservative form: 
\begin{subnumcases}{(\mathrm{P})~\label{eqn-P-nonconservative}}
\partial_t{\rho} = 0, \label{eqn-internal-rho}\\
\rho\partial_t\vec{\vec{u}} - {\textstyle\frac{1}{\Rey}}\div{\vec{\tau}(\vec{u})} = \vec{0}, \label{eqn-internal-u}\\
\rho\partial_t{e} -{\textstyle\frac{\lambda}{\Rey}}\laplace{e} = {\textstyle\frac{1}{\Rey}}\vec{\tau}(\vec{u}):\grad{\vec{u}}. \label{eqn-internal-e}
\end{subnumcases}
We use the positivity-preserving Runge--Kutta DG method \cite{zhang2010positivity} for subproblem ($\mathrm{H}$), i.e., 
the Zhang--Shu method for constructing positivity-preserving schemes \cite{zhang2010maximum,zhang2010positivity,zhang2011positivity,zhang2012maximum, zhang2012minimum} applied to solving compressible Euler equations, which is arbitrarily high order accurate, conservative, and positivity-preserving.
For the parabolic subproblem, many different types of DG methods have been developed for solving diffusion equations in literature, which include interior penalty DG \cite{girault2005discontinuous,liu2019interior,masri2022discontinuous,masri2023improved}, local DG \cite{cockburn1998local,castillo2000priori}, direct DG \cite{liu2010direct,zhang2012fourier,liu2015optimal}, hybridizable DG \cite{cockburn2009hybridizable,peraire2010hybridizable,nguyen2011implicit}, compact DG \cite{peraire2008compact,uranga2009implicit}, and so on.
In this paper, we utilize the interior penalty DG method to discretize subproblem ($\mathrm{P}$).
The first challenge of using DG methods for subproblem ($\mathrm{P}$) is how to ensure conservation of conserved variables. In \cite{liu2023positivity}, we have proven that 
conservation can be preserved via choosing appropriate interior penalty DG forms of $\div{\vec{\tau}(\vec{u})}$ and $\vec{\tau}(\vec{u}):\grad{\vec{u}}$. The next major challenge is how to ensure positivity when discretizing \eqref{eqn-internal-e}.  
It is very difficult to prove any positivity-preserving  result for arbitrarily high order schemes  solving \eqref{eqn-internal-e} for implicit time stepping, even if the temporal accuracy is only first order. 

Consider a heat equation $\partial_t{e} - \laplace{e} = 0$ as a simplification of \eqref{eqn-internal-e}.
When using backward Euler time discretization, a systematic approach to obtaining a sufficient condition for the discrete maximum principle or positivity is to show the monotonicity of the linear system matrix. A matrix is called {\it monotone} if all entries of its inverse are nonnegative. 
The monotonicity of $\IQ^1$ interior penalty DG on multi-dimensional structured meshes has been established in \cite{liu2023positivity}, also see \cite{horvath2013discrete, li2023monotone} for related results; and the monotonicity of continuous finite element method with $\IQ^2$ and $\IQ^3$ elements has been proven in \cite{li2020monotonicity,cross2023monotonicityQ2,cross2023monotonicityQ3}.
However, for arbitrary high order schemes on unstructured meshes, the monotonicity does not hold \cite{hohn1981some}.
Furthermore, for higher order implicit time marching strategy, such as the Crank--Nicolson method, the monotonicity of the linear system matrix is not enough to ensure positivity using a  time step like $\mathcal{O}(\Delta x)$.
The Crank--Nicolson method with a monotone spatial discretization preserves positivity only if 
 the time step is as small as $\mathcal{O}(\Delta x^2)$, see  \cite[Appendix B]{MR4710829} and \cite[Section 5.3]{guermond2021second}. 

\subsection{A constraint optimization approach for enforcing positivity and global conservation}

To preserve positivity of internal energy, we will introduce a constraint optimization  postprocessing approach. 
For enforcing bounds or positivity in numerical schemes solving PDEs, various optimization based approaches have been considered in the literature. We list a few such methods.
Guba et al. in \cite{guba2014optimization} introduced a bound-preserving limiter for spectral element method, implemented by standard quadratic programming solvers.
van der Vegt et al. in \cite{van2019positivity} considered a positivity-preserving limiter for DG scheme with implicit time integration and formulated the positivity constraints in the KKT system, implemented by an active set semismooth Newton method.
Cheng and Shen in \cite{cheng2022new} introduced a Lagrange multiplier approach to preserve bounds for semilinear and quasi-linear parabolic equations, which provides a new interpretation for the cut-off method and achieves the preservation of mass by solving a nonlinear algebraic equation for the additional space independent Lagrange multiplier.
Ruppenthal and Kuzmin in \cite{ruppenthal2023optimal} utilized optimization-based flux correction to ensure the positivity of finite element discretization of conservation laws. The primal-dual Newton method was employed to calculate the optimal flux potentials.  
\par

Next, we describe the main idea of our approach. 
Let $\overline{\vec{U}_i^\mathrm{P}} = \transpose{[\overline{\rho_i^\mathrm{P}},\overline{\vec{m}_i^\mathrm{P}}, \overline{E_i^\mathrm{P}}]}$ be a vector denoting the cell average of the DG polynomial ${\vec{U}_h^\mathrm{P}}(\vec{x})=\transpose{[\rho^\mathrm{P}_h(\vec{x}), \vec{m}^\mathrm{P}_h(\vec{x}),E^\mathrm{P}_h(\vec{x})]}$ on the $i$-th cell $K_i$ after solving subproblem ($\mathrm{P}$). The density cell averages are positive, which can be ensured if using a positivity-preserving scheme for subproblem ($\mathrm{H}$). The main challenge here is that  in general $\overline{\vec{U}_i^\mathrm{P}}$ may not be in the convex invariant domain set $G$.
We emphasize that 
the Zhang--Shu limiter \cite{zhang2010positivity} can be used only if $\overline{\vec{U}_i^\mathrm{P}}\in G$, which can be proven for one time step or time stage  for fully explicit finite volume and DG schemes with a positivity-preserving flux \cite{zhang2010positivity, zhang2017positivity}, or very special semi-implicit schemes like \cite{liu2023positivity}, thus these schemes can be rendered positivity-preserving by using the  Zhang--Shu limiter \cite{zhang2010positivity} in each time step or time stage.

With a prescribed small positive number $\epsilon$, which serves as the desired lower bound for density and internal energy, the numerical admissible state set $G^\epsilon$ is defined as follows.
\begin{align*}
G^\epsilon = \{\vec{U}=\transpose{[\rho,\vec{m},E]}\!:~\rho\geq\epsilon,~ \rho e(\vec{U}) = E - \frac{\|\vec{m}\|^2}{2\rho} \geq \epsilon\}.
\end{align*}

Define $\overline{E_h^\mathrm{P}}=\transpose{[\overline{E_1^\mathrm{P}}, \overline{E_2^\mathrm{P}}, \cdots, \overline{E_N^\mathrm{P}}]}$ as the vector of all cell averages for the total energy. We propose to modify the total energy only.  And we would like to modify it to another vector 
$\overline E_h=\transpose{[\overline E_1, \overline E_2, \cdots, \overline E_N]}$ such that it minimizes the $\ell^2$ distance to $\overline{E_h^\mathrm{P}}$, subject to the constraints of preserving global conservation and positivity. Specifically, given $\overline{\vec{U}_h^\mathrm{P}}=\transpose{[\overline{\vec{U}_1^\mathrm{P}}, \cdots, \overline{\vec{U}_N^\mathrm{P}}]}$ with positive density $\overline{\rho_i^\mathrm{P}}\geq \epsilon$, find the minimizer for 
\begin{subequations}
    \label{postprocessing}
\begin{align}
\label{total-energy-opt}
\min_{\overline E_h\in\IR^N}
\left\|\overline E_h-\overline{E_h^\mathrm{P}}\right\|^2 
\quad\text{subjects to}\quad 
\sum_{i=1}^N \overline E_i |K_i| = \sum_{i=1}^N \overline{E_i^\mathrm{P}} |K_i|
\quad\text{and}\quad
\transpose{[\overline{\rho_i^\mathrm{P}},\overline{\vec{m}_i^\mathrm{P}}, \overline E_i]} \in G^\epsilon,\quad \forall i,
\end{align}
where $|K_i|$ is the area or volume of each cell $K_i$.
Let $\overline E_h^{\,\ast}=\transpose{[\overline E_1^{\,\ast}, \cdots, \overline E_N^{\,\ast}]}$ be the minimizer. Then we correct the DG polynomial cell averages for the total energy variable. Namely, let $E_i^\mathrm{P}(\vec{x})$ be the DG polynomial in each cell $K_i$, and we correct it by a constant
\begin{equation}
\label{update-average-limiter}
    E_i(\vec{x})=E_i^\mathrm{P}(\vec{x})-\overline{E^\mathrm{P}_i}+\overline E_i^{\,\ast}.
\end{equation} 
 \end{subequations}
The updated or postprocessed DG polynomials  ${\vec{U}_h^\mathrm{P}}(\vec{x})=\transpose{[\rho^\mathrm{P}_h(\vec{x}), \vec{m}^\mathrm{P}_h(\vec{x}),E_h(\vec{x})]}$ now have cell averages in the numerical admissible state set $G^\epsilon$, and the simple Zhang--Shu positivity-preserving limiter in \cite{zhang2010positivity, zhang2011positivity} can be used to further ensure the full scheme is positivity-preserving. 

Since $\ell^2$ distance is minimized, the accuracy of \eqref{total-energy-opt} can also be justified under suitable assumptions, which will be discussed in Section \ref{sec:accuracy-pp}.

\subsection{Efficient implementation of the constraint optimization defined postprocessing}
\label{sec:intro-postprocessing}
 The simple postprocessing approach \eqref{postprocessing} was considered in \cite{liu2023simple} for preserving bounds of a scalar variable in complex phase field equations. Thanks to the constraints in 
 \eqref{total-energy-opt}, global conservation and positivity of the internal energy are easily achieved, and the accuracy is also easy to justify for scalar variables   \cite{liu2023simple}, which are 
the advantages of such a simple approach. On the other hand, in any optimization based approach, it is often quite straightforward to have these desired properties such as positivity, conservation, and high order accuracy. From this perspective, the critical issue in all optimization based approaches is computational efficiency, especially for a time-dependent, demanding nonlinear system like \eqref{eq:CNS:model}.

In large-scale high-resolution fluid dynamic simulations, degree of freedoms to be processed at each time step can be quite large. Thus in general it is preferred to solve \eqref{total-energy-opt} by first order optimization methods since they scale well with problem size, i.e., the complexity is $\mathcal{O}(N)$ for each iteration, with $N$ being the total number of cells.

In \cite{liu2023simple}, it is demonstrated that
the minimizer to a constraint minimization like \eqref{total-energy-opt} can be efficiently computed by using the Douglas--Rachford splitting method \cite{lions1979splitting} if using the nearly optimal algorithm parameters obtained from a sharp asymptotic convergence rate analysis.
The Douglas--Rachford splitting method is a very popular first order splitting method, because it is equivalent to ADMM \cite{fortin2000augmented} and dual split Bregman method \cite{goldstein2009split} with special parameters, see also \cite{demanet2016eventual} and references therein for the equivalence. For special convex optimization problems, it is also  equivalent to PDHG \cite{chambolle2016introduction}. 
\par 

There are  other efficient alternative methods to solve  the  minimization  \eqref{total-energy-opt}, such as the breakpoint searching algorithms \cite{kiwiel2008breakpoint} with an $\mathcal O(N)$ computational complexity. For the $\ell^2$-norm minimization  \eqref{total-energy-opt}, the Douglas--Rachford splitting with the optimal parameters also has a provable computational complexity $\mathcal O(N)$  as shown in \cite{liu2023simple}, but with more flexibilities and advantages. First, the Douglas--Rachford splitting method is simple to describe and easy to implement since only three steps are needed in each iteration, which allows easy implementation, especially for efficient parallel computing.
Second, it is straightforward to extend the Douglas--Rachford splitting method to other postprocessing models such as the  $\ell^1$-norm minimization and directly enforcing invariant domain $G^\epsilon$,  see Remark \ref{rmk-3} and Remark \ref{rmk-4} in Section \ref{sec: DR-solver}.  Though the Douglas--Rachford splitting method may no longer have a provable $\mathcal O(N)$ computational complexity for $\ell^1$-norm minimization, it is nontrivial or impossible to generalize other  alternative methods for \eqref{total-energy-opt} to $\ell^1$-norm minimization. 
In  \ref{sec:appendix}, we show a comparison to one  simple and efficient alternative solving \eqref{total-energy-opt} by the method of Lagrange multiplier, to demonstrate the practical efficiency of the Douglas--Rachford splitting for large problems.

Given the DG polynomial after solving the subproblem ($\mathrm{P}$), we define the $i$-th cell as a bad cell if its cell average has negative internal energy, i.e.,  $\overline{\vec{U}_i^\mathrm{P}} = \transpose{[\overline{\rho_i^\mathrm{P}},\overline{\vec{m}_i^\mathrm{P}}, \overline{E_i^\mathrm{P}}]}\notin G^\epsilon$. 
Let $r$ be the number of bad cells, then $r/N$ is the bad cell ratio. It is proven in \cite{liu2023simple} that the sharp asymptotic linear convergence rate of the Douglas--Rachford splitting with the nearly optimal parameters is approximately $\frac{1-2\frac{r}{N}}{3-2\frac{r}{N}}\approx\frac{1}{3}$ when $r\ll N$. In other words, such a minimization solver is provably extremely efficient when the bad cell ratio is small, which is usually the case for a good scheme solving \eqref{eq:CNS:model} such as Strang splitting with DG methods \cite{liu2023positivity}.

\subsection{The main result and organization of this paper}

Our full scheme in this paper
is a very high order accurate in space, conservative, and positivity-preserving semi-implicit DG scheme to solve the compressible NS equations \eqref{eq:CNS:model}, with a standard hyperbolic CFL $\Delta t = \mathcal{O}{(\Delta x)}$. For the implicit part, the scheme is fully decoupled with two linear systems to solve sequentially for each time step. We emphasize that the spatial discretization in this paper is done by only DG methods, which is not exactly the same as the spatial scheme in \cite{liu2023positivity}, where the internal energy equation is discretized by continuous finite element method.
The main novelties of this paper include the optimization-based postprocessing approach \eqref{postprocessing} to preserve conservation and positivity for solving the parabolic subproblem as well as  a proper semi-implicit DG scheme with high order basis, which is carefully designed so that the DG scheme combined with the optimization-based positivity-preserving limiter can produce stable and solid results for challenging benchmark gas dynamics problems.  The minimizer to \eqref{total-energy-opt} can be efficiently computed by using the generalized Douglas--Rachford splitting method with nearly optimal parameters.

The postprocessing step \eqref{total-energy-opt} only preserves the global conservation and does not preserve any local conservation property. We remark that the local conservation in the Strang splitting approach for solving \eqref{eq:CNS:model} is already lost since the non-conservative variables are computed in \eqref{eqn-P-nonconservative}. Nonetheless, the global conservation can be ensured \cite{liu2023positivity}.
Thus from this perspective, the postprocessing step \eqref{total-energy-opt} is acceptable whenever the non-conservative form \eqref{eqn-P-nonconservative} is solved.

One can also consider a more general version of \eqref{total-energy-opt} by also modifying the density and momentum variables to enforce the positivity of the internal energy $\overline{\vec{U}_i^\mathrm{P}} = \transpose{[\overline{\rho_i^\mathrm{P}},\overline{\vec{m}_i^\mathrm{P}}, \overline{E_i^\mathrm{P}}]}\in G^\epsilon$. Such a more complicated limiter is certainly more difficult to implement efficiently. On the other hand, for the Strang splitting approach in \cite{guermond2021second, liu2023positivity}, the momentum variable is robustly computed, which allows us to consider a simpler limiter like \eqref{total-energy-opt}. Most importantly,  numerical tests suggest that the simple postprocessing \eqref{postprocessing} is sufficient to enforce the positivity thus the robustness for the subproblem ($\mathrm{P}$) in the Strang splitting with very high order DG methods.

\par
We emphasize that the postprocessing \eqref{postprocessing} is too simple to make a bad scheme more useful, e.g., it does not eliminate any oscillations. It is most useful for a good scheme that is stable for most testing cases yet might lose positivity thus robustness for solving challenging low pressure problems, e.g., the Strang splitting method in \cite{guermond2021second, liu2023positivity}. For instance, as will be shown by numerical tests in this paper, for computing the Mach 2000 astrophysical jet problem, Strang splitting with very high order DG methods produces blow-up due to loss of positivity, but will be stable when combined with the postprocessing \eqref{postprocessing}, i.e., an optimization based positivity-preserving limiter. On the other hand, there are many different kinds of DG methods for diffusion operators. With  a proper choice of the interior penalty DG method,  we demonstrate that  global conservation can be ensured when solving the diffusion subproblem implicitly in the Strang splitting of compressible Navier-Stokes system, and only two linear systems need to be solved in the Crank--Nicolson  time discretization of the diffusion subproblem. Moreover, the numerical tests suggest that such a high order DG scheme is a practical scheme producing solid results for some chanllenging benchmark problems.

\par
The rest of this paper is organized as follows. In Section~\ref{sec:numercal_scheme}, we introduce the fully discrete numerical scheme. In Section~\ref{sec:limiter}, we discuss a high order accurate constraint optimization based postprecessing procedure, which preserves the conservation and positivity. Numerical tests are shown in Section~\ref{sec:numercal_experiment}. Concluding remarks are given in Section~\ref{sec-remark}.

\section{Numerical scheme}\label{sec:numercal_scheme}
In this section, we describe the fully discretized numerical scheme for solving the compressible NS equations \eqref{eq:CNS:model}. Our scheme incorporates the DG spatial discretization within the Strang splitting framework. 

\subsection{Time discretization}
Given the conserved variables $\vec{U}^n$ at time $t^n$ ($n\geq0$) and the step size $\Delta t$, the Strang splitting for evolving to time $t^{n+1}=t^n + \Delta t$ for the system \eqref{eq:CNS:model} is to solve subproblems $(\mathrm{H})$ and $(\mathrm{P})$ separately \cite{guermond2021second, liu2023positivity}. A schematic flowchart for time marching is as follows:
\begin{align}\label{algorithm-splitting}
\vec{U}^n
\xrightarrow[\text{step size}~\frac{\Delta t}{2}]{\text{solve}~(\mathrm{H})} \vec{U}^\mathrm{H}
\xrightarrow[\text{step size}~\Delta t]{\text{solve}~(\mathrm{P})} \vec{U}^\mathrm{P}
\xrightarrow[\text{step size}~\frac{\Delta t}{2}]{\text{solve}~(\mathrm{H})} \vec{U}^{n+1}.
\end{align}
We utilize the strong stability preserving (SSP) Runge--Kutta method to solve $(\mathrm{H})$ and the $\theta$-method with a parameter $\theta\in(0,1]$ to solve $(\mathrm{P})$. For any $n\geq 0$, the time discretization in one time step consists of the following steps.
\begin{itemize}[leftmargin=0.5cm]
\item[] Step~1. Given $\vec{U}^n=\transpose{[\rho^n,\vec{m}^n,E^n]}$, we use the third order SSP Runge--Kutta method with step size $\frac{1}{2}\Delta t$ to compute $\vec{U}^{\mathrm{H}}=\transpose{[\rho^{\mathrm{H}}, \vec{m}^{\mathrm{H}}, E^{\mathrm{H}}]}$:
\begin{subequations}\label{eq:rk3}
\begin{align}
\vec{U}^{\mathrm{(1)}} &= \vec{U}^n - \frac{\Delta t}{2}\div{\vec{F}^\mathrm{a}(\vec{U}^n)},\\
\vec{U}^{\mathrm{(2)}} &= \frac{3}{4} \vec{U}^{n}+\frac{1}{4}\Big[\vec{U}^{(1)} - \frac{\Delta t}{2}\div{\vec{F}^\mathrm{a}(\vec{U}^{(1)})}\Big],\\
\vec{U}^{\mathrm{H}} &= \frac{1}{3} \vec{U}^{n}+\frac{2}{3}\Big[\vec{U}^{(2)} - \frac{\Delta t}{2}\div{\vec{F}^\mathrm{a}(\vec{U}^{(2)})}\Big].
\end{align}
\end{subequations}
 
\item[] Step~2. Given $\vec{U}^{\mathrm{H}} = \transpose{[\rho^{\mathrm{H}}, \vec{m}^{\mathrm{H}}, E^{\mathrm{H}}]}$, compute $(\vec{u}^\mathrm{H}, e^\mathrm{H})$ by solving
\begin{align*}
\vec{m}^\mathrm{H} = \rho^\mathrm{H} \vec{u}^\mathrm{H} \quad\text{and}\quad
E^\mathrm{H} = \rho^\mathrm{H}e^\mathrm{H} + \frac{\norm{\vec{m}^{\mathrm{H}}}{}^2}{2\rho^\mathrm{H}}.
\end{align*}
\item[] Step~3.
Given $(\vec{u}^\mathrm{H}, e^\mathrm{H})$, set $\rho^{\mathrm{P}} = \rho^{\mathrm{H}}$ due to \eqref{eqn-internal-rho}. We employ the Crank--Nicolson method to discretize \eqref{eqn-internal-u} and apply the $\theta$-method, where $\theta\in(0,1]$, to discretize \eqref{eqn-internal-e}. For the second step in Strang splitting \eqref{algorithm-splitting}, we have
\begin{align*}
&\vec{u}^\ast=\frac{1}{2} \vec{u}^\mathrm{P}+\frac12 \vec{u}^\mathrm{H}
\quad\text{and}\quad 
e^\ast=\theta e^\mathrm{P} + (1-\theta)e^\mathrm{H},\\
&\rho^{\mathrm{P}}\frac{\vec{u}^\mathrm{P}-\vec{u}^\mathrm{H}}{\Delta t} - \frac{1}{\Rey}\div{\vec{\tau}(\vec{u}^\ast)} = \vec{0}, \\
&\rho^{\mathrm{P}} \frac{e^\mathrm{P}-e^\mathrm{H}}{\Delta t} - \frac{\lambda}{\Rey}\laplace{e^\ast} = \frac{1}{\Rey}\vec{\tau}(\vec{u}^\ast):\grad{\vec{u}^\ast}. 
\end{align*}
The scheme above can be implemented as first to compute $(\vec{u}^{\ast}, e^{\ast})$ by sequentially solving two decoupled linear systems
\begin{subequations}\label{eq:time_discretization}
\begin{align}
\rho^{\mathrm{P}}\vec{u}^\ast - \frac{\Delta t}{2\Rey}\div{\vec{\tau}(\vec{u}^\ast)} &= \rho^{\mathrm{H}}\vec{u}^{\mathrm{H}},\label{eq:time:step3_1}\\
\rho^{\mathrm{P}} e^\ast - \frac{\theta{\Delta t}\,\lambda}{\Rey}\laplace{e^\ast} &= \rho^{\mathrm{H}}e^{\mathrm{H}} + \frac{\theta\Delta t}{\Rey}\vec{\tau}(\vec{u}^\ast):\grad{\vec{u}^\ast},\label{eq:time:step3_2}
\end{align}
\end{subequations}
then set $\vec{u}^\mathrm{P} = 2\vec{u}^\ast - \vec{u}^\mathrm{H}$ and $e^\mathrm{P} = \frac{1}{\theta} e^\ast + (1-\frac{1}{\theta})e^\mathrm{H}$.
\item[] Step~4. Given $(\rho^{\mathrm{P}}, \vec{u}^{\mathrm{P}}, e^{\mathrm{P}})$, compute $(\vec{m}^\mathrm{P}, E^\mathrm{P})$ by
\begin{align*}
\vec{m}^\mathrm{P} = \rho^\mathrm{P}\vec{\vec{u}}^\mathrm{P}
\quad\text{and}\quad
E^\mathrm{P} = \rho^\mathrm{P}e^\mathrm{P} + \frac{\norm{\vec{m}^{\mathrm{P}}}{}^2}{2\rho^\mathrm{P}}.
\end{align*}
\item[] Step~5. Given $\vec{U}^{\mathrm{P}}=\transpose{[\rho^{\mathrm{P}}, \vec{\vec{m}}^\mathrm{P}, E^{\mathrm{P}}]}$,  
to obtain $\vec{U}^{n+1}=\transpose{[\rho^{n+1},\vec{m}^{n+1},E^{n+1}]}$ in the third step in Strang splitting \eqref{algorithm-splitting}, solve $(\mathrm{H})$ for another $\frac{1}{2}\Delta t$ by the third order SSP Runge--Kutta method.  
\end{itemize}
 We have the first order backward Euler scheme with $\theta = 1$,  for which $e^\mathrm{P} = e^\ast $ and it is possible to design positivity-preserving schemes if the discrete Laplacian is monotone, e.g., $\IQ^2$ and $\IQ^3$ spectral element methods on uniform meshes, as shown in \cite{liu2023positivity}.
Unfortunately, for any $\theta<1$, 
 $e^\mathrm{P} = \frac{1}{\theta} e^\ast + (1-\frac{1}{\theta})e^\mathrm{H}$ is not a convex combination thus
 it is difficult to have $e^P >0$ even if $e^\ast >0$ can be ensured by a monotone discrete Laplacian. 
 For $\theta=\frac{1}{2}$, we have the second order Crank--Nicolson scheme. 
It is important to note that in each time step, only two decoupled linear systems need to be sequentially solved in \eqref{eq:time_discretization}.
\subsection{Preliminary aspects of space discretization}
We use the Runge--Kutta DG method to discretize subproblem $(\mathrm{H})$ and the interior penalty DG method to discretize subproblem $(\mathrm{P})$. For completeness, we briefly review these methods without delving into their derivation.  See \cite{zhang2010positivity,zhang2017positivity,liu2023positivity} for more details. For simplicity, we only consider $\IQ^k$ polynomial basis on uniform rectangular meshes, and there is no essential difficulty to extend the main results in this paper to unstructured meshes. For example, for preserving conservation and positivity, the constraint optimization-based postprocessing approach discussed in Section~\ref{sec:simplelimiter} is also applicable to $\IP^k$ polynomials on unstructured meshes. 

\paragraph{\bf Mesh, approximation spaces, and quadratures}
Let $\setE_h = \{K_i\}$ be a uniform partition of the computational domain $\Omega$ by square elements (cells) with the element diameter $h$. The unit outward normal of a cell $K$ is denoted by $\normal_K$. Let $\Gammah$ be the set of interior faces. For each interior face $e \in \Gammah$ shared by cells $K_{i^-}$ and $K_{i^+}$, with $i^- < i^+$, we define a unit normal vector $\normal_e$ that points from $K_{i^-}$ into $K_{i^+}$. For a boundary face $e = \partial K_{i^-} \cap \partial\Omega$, the normal $\normal_e$ is taken to be the unit outward vector to $\partial\Omega$. 
\par
Let $\IQ^k(K)$ be the space of polynomials of order at most $k$ for each variable defined on a cell $K$. Define the following piecewise polynomial spaces:
\begin{align*}
M_h^k &= \big\{\chi_h\in L^2(\Omega):~\forall K \in \setE_h,\, \on{\chi_h}{K} \in \IQ^k(K) \big\}, \\
\mathbf{X}_h^k &= \big\{\vec{\theta}_h \in L^2(\Omega)^d:~\forall K \in \setE_h,\, \on{\vec{\theta}_h}{K} \in \IQ^k(K)^d \big\}.
\end{align*}
On a reference element $\hat{K}=[-\frac{1}{2},\frac{1}{2}]^d$, we use $(k+1)^d$ Gauss--Lobatto points to construct Lagrange interpolation polynomials $\hat{\varphi}_j$. The basis functions on each cell $K_i\in\setE_h$ are defined by $\varphi_{ij}=\hat{\varphi}_j\circ\vec{F}_i^{-1}$, where $\vec{F}_i: \hat{K}\rightarrow K$ is an invertible mapping from the reference element to $K_i$.
These basis are numerically orthogonal with respect to the $(k+1)^d$-point Gauss--Lobatto quadrature rule.
\par
We summarize the quadrature rules employed in solving the hyperbolic and parabolic subproblems as well as the points to be used in the positivity-preserving limiter as follows:
\begin{enumerate}[leftmargin=0.5cm]
\item For face and volume integrals in ($\mathrm{H}$), we utilize a quadrature rule that is constructed by the tensor product of $(k+1)$-point Gauss quadrature. Denote the set of associated quadrature points here by $S_K^{\mathrm{H, int}}$ on a cell $K$.
\item For face and volume integrals in ($\mathrm{P}$), we utilize a quadrature rule that is constructed by the tensor product of $(k+1)$-point Gauss–Lobatto quadrature. Denote the set of associated quadrature points here by $S_K^{\mathrm{P}}$ on a cell $K$.
\item The points for weak positivity of ($\mathrm{H}$) are constructed by $(k+1)$-point Gauss quadrature tensor product with $L$-point Gauss--Lobatto quadrature in both $x$ and $y$ directions and we need $2L-3\geq k$ so that the $L$-point Gauss--Lobatto quadrature is exact for integrating DG  polynomials of degree $k$  \cite{zhang2017positivity}. Denote the set of associated quadrature points here by $S_K^{\mathrm{H, aux}}$ on a cell $K$. Though these points form a quadrature, we do not use them for computing any integrals. Instead, they are the points to be used in the positivity-preserving limiter \cite{zhang2010positivity, zhang2011positivity, zhang2017positivity}. 
\end{enumerate}
See Figure~\ref{fig:quadrature_rule} for an illustration the location of these quadrature points in the $\IQ^4$ scheme.
\begin{figure}[ht!]
\begin{center}
\begin{tabularx}{0.975\linewidth}{@{}C@{~}C@{~}C@{~}C@{~~}C@{}}
\includegraphics[width=0.18\textwidth]{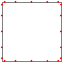} &
\includegraphics[width=0.18\textwidth]{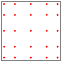}  &
\includegraphics[width=0.18\textwidth]{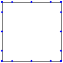} &
\includegraphics[width=0.18\textwidth]{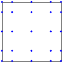}  &
\includegraphics[width=0.18\textwidth]{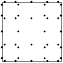} \\
\end{tabularx}
\caption{An illustration of the quadratures used in the $\IQ^4$ scheme. From left to right: the quadrature points for face integrals in ($\mathrm{H}$), volume integrals in ($\mathrm{H}$), face integrals in ($\mathrm{P}$), volume integrals in ($\mathrm{P}$), and the quadrature points for weak positivity. The black points are used only in defining the positivity-preserving limiter, and they are not used in calculating any numerical integration.}
\label{fig:quadrature_rule}
\end{center}
\end{figure}

\paragraph{\bf Hyperbolic subproblem}
One of the most popular high order accurate positivity-preserving approaches for solving compressible Euler equations $\partial_t{\vec{U}} + \div{\vec{F}^\mathrm{a}(\vec{U})} = \vec{0}$ was introduced by Zhang and Shu in \cite{zhang2010positivity}, also see \cite{zhang2017positivity}. We utilize the same scheme to solve $(\mathrm{H})$, which is defined as follows. For any piecewise polynomial test function $\varPsi_h$, find the piecewise polynomial solution $\vec{U}_h$, such that
\begin{align}\label{eq:hyper_space_dis}
\frac{\dd}{\dd t}(\vec{U}_h, \varPsi_h)
= (\vec{F}^{\mathrm{a}}(\vec{U}_h), \grad{\varPsi_h})
- \int_{\partial{K}} \widehat{\vec{F}^{\mathrm{a}}\cdot\normal_K}(\vec{U}_h^{-},\vec{U}_h^{+})\varPsi_h,
\end{align}
where $\widehat{\vec{F}^\mathrm{a}\cdot\normal_K}$ is any monotone flux  for $\vec{F}^\mathrm{a}$, e.g.,  a Lax--Friedrichs type flux. On a face $e\subset \partial K$, the local Lax--Friedrichs flux is defined by
\begin{align*}
\widehat{\vec{F}^\mathrm{a}\cdot\normal_K}(\vec{U}_h^-,\vec{U}_h^+) 
= \frac{\vec{F}^{\mathrm{a}}(\vec{U}_h^-) + \vec{F}^{\mathrm{a}}(\vec{U}_h^+)}{2}\cdot\normal_K 
- \frac{\alpha_e}{2}(\vec{U}_h^+ - \vec{U}_h^-),
\end{align*}
where the $\vec{U}_h^-$ (resp. $\vec{U}_h^+$) denotes the trace of $\vec{U}_h$ on the face $\partial{K}$ coming from the interior (resp. exterior) of $K$. 
The factor $\alpha_e$ denotes the maximum wave speed with maximum taken over all $\vec{U}_h^-$ and $\vec{U}_h^+$ along the face $e$, namely the largest magnitude of the eigenvalues of the Jacobian matrix $\frac{\partial \vec{F}^\mathrm{a}}{\partial \vec{U}}$, which equals to the wave speed $\abs{\vec{u}\cdot\normal_K}+\sqrt{\gamma\frac{p}{\rho}}$ for ideal gas equation of state.
\par
By convention, we replace $\vec{U}_h^+$ by an appropriate boundary function which realizes the boundary conditions when $\partial{K}\cap\partial{\Omega}\neq\emptyset$. For instance, if purely inflow condition $\vec{U} = \vec{U}_\mathrm{D}$ is imposed on $\partial{K}$, then $\vec{U}_h^+$ is replaced by $\vec{U}_\mathrm{D}$; if purely outflow condition is imposed on $\partial{K}$, then set $\vec{U}_h^+=\vec{U}_h^-$; and if reflective boundary condition for fluid--solid interfaces is imposed on $\partial{K}$, then set $\vec{U}_h^+ = \transpose{[\rho_h^-,\vec{m}_h^{-}-2(\vec{m}_h^-\cdot\normal_K)\normal_K,E_h^-]}$.

\paragraph{\bf Parabolic subproblem}
We use the interior penalty DG method for discretizing $\mathrm{(P)}$.
For convenience of introducing discrete forms in parabolic subproblem, we partition the boundary of the domain $\Omega$ into the union of two disjoint sets, namely $\partial{\Omega} = \partial{\Omega}_\mathrm{D} \cup \partial{\Omega}_\mathrm{N}$, where the Dirichlet boundary conditions ($\vec{u}=\vec{u}_\mathrm{D}$ and $e=e_\mathrm{D}$) are applied on $\partial{\Omega}_\mathrm{D}$ and the Neumann-type boundary conditions ($\vec{\tau}(\vec{u})\cdot\normal=\vec{0}$ and $\grad{e}\cdot\normal=0$) are applied on $\partial{\Omega}_\mathrm{N}$. Here, $\normal$ denotes the unit outer normal of domain $\Omega$.
\par
The average and jump operators of any vector quantity $\vec{u}$ on a boundary face coincide with its trace; and on interior faces they are defined by
\begin{align*}
\on{\avg{\vec{u}}}{e} = \frac{1}{2}\on{\vec{u}}{K_{i^-}} + \frac{1}{2}\on{\vec{u}}{K_{i^+}}, \quad
\on{\jump{\vec{u}}}{e} = \on{\vec{u}}{K_{i^-}} - \on{\vec{u}}{K_{i^+}}, \quad 
e = \partial K_{i^-} \cap \partial K_{i^+}.
\end{align*}
The related definitions of any scalar quantity are similar. For more details see \cite{Rivierebook}.
We employ the non-symmetric interior penalty DG (NIPG) method to discretize the terms $-2\div{\strain{(\vec{u})}}$ and $\div{((\div{\vec{u})\vecc{I}})}$. The associated bilinear forms $a_\strain$ and $a_\lambda$ are defined as follows:
\begin{align*}
a_\strain(\vec{u}, \vec{\theta}) &= 
2\sum_{K\in\setE_h} \int_K \strain{(\vec{u})}:\strain{(\vec{\theta})} 
- 2\sum_{e\in\Gammah\cup\partial\Omega_\mathrm{D}} \int_e\avg{\strain{(\vec{u})}\,\normal_e}\cdot\jump{\vec{\theta}}\nonumber\\
&+2\sum_{e\in\Gammah\cup\partial\Omega_\mathrm{D}} \int_e\avg{\strain{(\vec{\theta})}\,\normal_e}\cdot\jump{\vec{u}}
+ \frac{\sigma}{h}\sum_{e\in\Gammah\cup\partial\Omega_\mathrm{D}} \int_e \jump{\vec{u}}\cdot\jump{\vec{\theta}},\\
a_\lambda(\vec{u}, \vec{\theta}) &= -\sum_{K\in\setE_h}\int_K(\div{\vec{u}})(\div{\vec{\theta}})
+\!\! \sum_{e\in\Gammah\cup\partial\Omega_\mathrm{D}}\!\int_e\avg{\div{\vec{u}}}\jump{\vec{\theta}\cdot\normal_e} 
-\!\! \sum_{e\in\Gammah\cup\partial\Omega_\mathrm{D}}\!\int_e\avg{\div{\vec{\theta}}}\jump{\vec{u}\cdot\normal_e}.
\end{align*}
And the linear form $b_{\vec{\tau}}$ associated with term $-\div{\vec{\tau}(\vec{u})}$ for the Dirichlet boundary $\partial \Omega_\mathrm{D}$ in \eqref{eq:time:step3_1} is defined by
\begin{align*}
b_{\vec{\tau}}(\vec{\theta}) 
= 2\sum_{e\in\partial\Omega_\mathrm{D}} \int_e (\strain{(\vec{\theta})}\,\normal)\cdot\vec{u}_\mathrm{D}
+ \frac{\sigma}{h}\sum_{e\in\partial\Omega_\mathrm{D}} \int_e \vec{u}_\mathrm{D}\cdot\vec{\theta}
- \frac{2}{3}\sum_{e\in\partial\Omega_\mathrm{D}}\int_e\div{\vec{\theta}}\,(\vec{u}_\mathrm{D}\cdot\normal).
\end{align*}
We employ the incomplete interior penalty DG (IIPG) method to discretize the term $-\laplace{e}$ in \eqref{eq:time:step3_2}. The bilinear form $a_{\mathcal{D}}$ and the linear form $b_{\mathcal{D}}$ for term $-\laplace{e}$ are defined as follows:
\begin{align*}
a_{\mathcal{D}}(e,\chi) &=
\sum_{K\in\setE_h} \int_K \grad e \cdot \grad \chi
-\sum_{e\in\Gammah\cup\partial{\Omega_{\mathrm{D}}}} \int_e \avg{\grad e \cdot \normal_e} \jump{\chi}
+ \frac{\tilde{\sigma}}{h} \sum_{e\in\Gammah\cup\partial{\Omega_{\mathrm{D}}}}\int_e \jump{e}\jump{\chi},\\
b_{\mathcal{D}}(\chi) &=
\frac{\tilde{\sigma}}{h}\sum_{e\in\partial{\Omega_{\mathrm{D}}}}\int_e e_\mathrm{D}\chi.
\end{align*}
For the sake of global conservation of total energy, to discrete term $\vec{\tau}(\vec{u}):\grad{\vec{u}} = 2\strain{(\vec{u})}:\grad{\vec{u}} - \frac{2}{3}((\div{\vec{u}})\vecc{I}):\grad{\vec{u}}$ in \eqref{eq:time:step3_2}, by using the tensor identity $\strain{(\vec{u})}:\grad{\vec{u}} = \strain{(\vec{u})}:\strain{(\vec{u})}$, the DG forms $b_\strain$ and $b_\lambda$ are designed for terms $2\strain{(\vec{u})}:\grad{\vec{u}}$ and $-((\div{\vec{u})\vecc{I}}):\grad{\vec{u}}$, respectively.
\begin{align*}
b_\strain(\vec{u}, \chi) &=
2\sum_{K\in\setE_h} \int_K \strain{(\vec{u})}:\strain{(\vec{u})}\chi
+ \frac{\sigma}{h}\sum_{e\in\Gammah} \int_e \jump{\vec{u}}\cdot\jump{\vec{u}}\avg{\chi}
+ \frac{\sigma}{h}\sum_{e\in\partial\Omega_\mathrm{D}} \int_e (\vec{u}-\vec{u}_\mathrm{D})\cdot(\vec{u}-\vec{u}_\mathrm{D}) \chi,\\
b_\lambda(\vec{u}, \chi) &= -\sum_{K\in\setE_h}\int_K(\div{\vec{u}})(\div{\vec{u}})\chi.
\end{align*}
The DG forms above employ penalty parameters $\sigma$ and $\tilde{\sigma}$. For any $\sigma\geq0$, the NIPG bilinear form is coercive. In particular, NIPG0 refers to the choice $\sigma=0$, e.g., the penalty term is removed. The NIPG0 method is convergent for polynomial degrees greater than or equal to two in two dimension \cite{Rivierebook}. And more importantly, the NIPG0 method eliminates the face penalties, thereby reducing the numerical viscosity.
For IIPG method, the penalty $\tilde{\sigma}$ needs to be large enough to achieve coercivity.

\subsection{The simple positivity-preserving limiter}
\label{sec:simplelimiter}
The Zhang--Shu limiter \cite{zhang2010maximum, zhang2010positivity} is a simple limiter for enforcing positivity of the approximation polynomial on a finite set $S$ when the polynomial cell average is positive.
 Let $\vec{U}_K(\vec{x}) = \transpose{[\rho_K, \vec{m}_K, E_K]}$ be the DG polynomial on cell $K$. 
A simplified version of the limiter \cite{zhang2017positivity} modifies the DG polynomial $\vec{U}_K(\vec{x})$ with the following steps under the assumption that $\overline{\vec{U}}_K = \frac{1}{\abs{K}}\int_K \vec{U}_K\in G^\epsilon$.
\begin{itemize}[leftmargin=0.5cm]
\item[] 1. First enforce positivity of density  by
\begin{align*}
\widehat{\rho}_K = \theta_\rho (\rho_K - \overline{\rho}_K) + \overline{\rho}_K,
\quad
\theta_{\rho}=\min\biggl\{1,\, \frac{\overline{\rho}_K-\epsilon}{\overline{\rho}_K-\min\limits_{\vec{x}_q\in S_K}\rho_K(\vec{x}_q)}\biggr\},
\end{align*}
where $\overline{\rho}_K$ denotes the cell average of $\rho_K$ on cell $K$. Notice that $\widehat{\rho}_K$ and $\rho_K$ have the same cell average, and $\widehat{\rho}_K = {\rho}_K$ if $\min\limits_{\vec{x}_q\in S_K}\rho_K(\vec{x}_q)\geq \epsilon.$
\item[] 2. Define $\widehat{\vec{U}}_h = \transpose{[\widehat{\rho}_h, \vec{m}_h, E_h]}$ and enforce positivity of internal energy by
\begin{align*}
\widetilde{\vec{U}}_K = \theta_e (\widehat{\vec{U}}_K - \overline{\vec{U}}_K) + \overline{\vec{U}}_K,
\quad
\theta_{e}=\min\biggl\{1,\, \frac{\overline {\rho e}_K - \epsilon}{\overline{\rho e}_K - \min\limits_{\vec{x}_q\in S_K}\rho e_K(\vec{x}_q)}\biggr\},
\end{align*}
where $\overline{\rho e}_K=\overline E_K - \frac{\|\overline{\vec{m}}_K\|^2}{2\overline\rho_K}$
and $\rho e_K(\vec{x}_q) = E_K(\vec{x}_q) - \frac{\|\vec{m}_K(\vec{x}_q)\|^2}{2\rho_K(\vec{x}_q)}$.
Notice that $\widetilde{\vec{U}}_K$ has the same cell average, the positivity is
implied by the Jensen's inequality satisfied by the concave internal energy function \cite{zhang2017positivity}.
\end{itemize}
We refer to \cite{zhang2010positivity, zhang2017positivity,xu2017bound} on the justification of its   high order accuracy.

\subsection{The fully discrete scheme}
\label{sec:flowchart}
Let $(\cdot,\cdot)$ denote the $L^2$ inner product over domain $\Omega$ evaluated by Gauss quadrature in ($\mathrm{H}$) and $\langle\cdot,\cdot\rangle$ denote the $L^2$ inner product over domain $\Omega$ evaluated by Gauss--Lobatto quadrature in ($\mathrm{P}$).

Given the DG solution $\vec{U}_h^n$ at time $t^n$ ($n\geq0$), a schematic flowchart for evolving to time $t^{n+1} = t^n + \Delta t$ is given as:
\begin{align*}
\vec{U}_h^n
\xrightarrow[\text{step size}~\frac{\Delta t}{2}]{\text{solve}~(\mathrm{H})} \vec{U}_h^\mathrm{H}
\xrightarrow[]{\text{$L^2$ proj.}} (\vec{u}_h^\mathrm{H},e_h^\mathrm{H})
\xrightarrow[\text{step size}~\Delta t]{\text{solve}~(\mathrm{P})} (\vec{u}_h^\mathrm{P},e_h^\mathrm{P})
\xrightarrow[]{\text{$L^2$ proj.}} \vec{U}_h^\mathrm{P}
\xrightarrow[\text{step size}~\frac{\Delta t}{2}]{\text{solve}~(\mathrm{H})} \vec{U}_h^{n+1},
\end{align*}
where the optimization-based postprocessing will be applied to $\vec{U}_h^\mathrm{P}$, as will be described in Step 4 below.
For any $n\geq0$, our fully discrete scheme for solving \eqref{eq:CNS:model} in one step consists of the following steps.
\begin{itemize}[leftmargin=0.5cm]
\item[] Step~1. Given $\vec{U}_h^n \in M_h^k\times \mathbf{X}_h^k\times M_h^k$, compute $\vec{U}_h^\mathrm{H}\in M_h^k\times \mathbf{X}_h^k\times M_h^k$ by the DG method \eqref{eq:hyper_space_dis} with the positivity-preserving SSP Runge--Kutta \eqref{eq:rk3} \cite{zhang2010positivity, zhang2017positivity} using step size $\frac{\Delta t}{2}$. After each Runge--Kutta stage, apply the Zhang--Shu positivity-preserving limiter to ensure that all point values at $S_K^{\mathrm{H, int}}$ and $S_K^{\mathrm{H, aux}}$ have positive density and internal energy.
\item[] Step~2. 
Use the Zhang--Shu positivity-preserving limiter to ensure that  all point values at $S_K^{\mathrm{P}}$ have positive density and internal energy.
Given $\vec{U}_h^\mathrm{H} \in M_h^k\times \mathbf{X}_h^k\times M_h^k$, compute $(\vec{u}_h^\mathrm{H}, e_h^\mathrm{H}) \in \mathbf{X}_h^k\times M_h^k$ by $L^2$ projection
\begin{align}\label{eq:CNS:P_full_dis_L2proj1}
\langle\vec{m}_h^\mathrm{H},\vec{\theta}_h \rangle = \langle\rho_h^\mathrm{H} \vec{u}_h^\mathrm{H},\vec{\theta}_h \rangle,~~
\forall \vec{\theta}_h \in \mathbf{X}_h^k \quad\text{and}\quad
\langle E_h^\mathrm{H},\chi_h \rangle = \langle\rho_h^\mathrm{H}e_h^\mathrm{H},\chi_h \rangle + \langle\frac{\vec{m}_h^{\mathrm{H}}}{2\rho_h^\mathrm{H}},\vec{m}_h^{\mathrm{H}}\chi_h \rangle,~~ \forall \chi_h\in M_h^k.
\end{align}
\item[] Step~3. Given $(\rho^{\mathrm{H}}_h, \vec{u}^{\mathrm{H}}_h) \in M_h^k \times \mathbf{X}_h^k$, set $\rho^{\mathrm{P}}_h = \rho^{\mathrm{H}}_h$ and solve $(\vec{u}^{\ast}_h, \vec{u}^{\mathrm{P}}_h)\in \mathbf{X}_h^k\times\mathbf{X}_h^k$, such that for all $\vec{\theta}_h\in \mathbf{X}_h^k$
\begin{subequations}\label{eq:space:step3}
\begin{align}
\langle\rho_h^\mathrm{P}\vec{u}_h^\ast,\vec{\theta}_h\rangle + \frac{\Delta t}{2\Rey}a_\strain(\vec{u}^{\ast}_h,\vec{\theta}_h) + \frac{\Delta t}{3\Rey}a_\lambda(\vec{u}^{\ast}_h,\vec{\theta}_h)
&= \langle\rho_h^\mathrm{H}\vec{u}_h^\mathrm{H},\vec{\theta}_h\rangle +\frac{\Delta t}{2\Rey} b_{\vec{\tau}}(\vec{\theta}_h),\label{eq:CNS:P_full_dis2_1}\\
\vec{u}^{\mathrm{P}}_h &= 2\vec{u}^{\ast}_h - \vec{u}^{\mathrm{H}}_h.\label{eq:CNS:P_full_dis2_2}
\end{align}
Then given $(\rho^{\mathrm{H}}_h, \rho^{\mathrm{P}}_h, \vec{u}^{\ast}_h, e^{\mathrm{H}}_h) \in M_h^k\times M_h^k\times \mathbf{X}_h^k\times M_h^k$, solve for $(e^{\ast}_h, e^{\mathrm{P}}_h) \in M_h^k\times M_h^k$, such that for all $\chi_h\in M_h^k$
\begin{align}\label{eq:CNS:P_full_dis3}
\langle\rho^\mathrm{P}_h e_h^\ast,\chi_h\rangle + \frac{\theta\Delta t\lambda}{\Rey}a_{\mathcal{D}}(e_h^\ast,\chi_h) &= \langle\rho_h^\mathrm{H} e_h^\mathrm{H},\chi_h\rangle + \frac{\theta\Delta t}{\Rey}b_\strain(\vec{u}_h^{\ast},\chi_h) 
+ \frac{2\theta\Delta t}{3\Rey}b_\lambda(\vec{u}_h^{\ast},\chi_h) + \frac{\theta\Delta t\lambda}{\Rey}b_\mathcal{D}(\chi_h),\\
e^\mathrm{P}_h &= \frac{1}{\theta} e^\mathrm{\ast}_h + (1-\frac{1}{\theta})e^\mathrm{H}_h.\label{full-scheme-eP}
\end{align}
\end{subequations}
\item[] Step~4. Given $(\rho_h^{\mathrm{P}}, \vec{u}_h^{\mathrm{P}}, e_h^{\mathrm{P}})\in M_h^k\times \mathbf{X}_h^k\times M_h^k$, compute $(\vec{m}_h^\mathrm{P}, E_h^\mathrm{P})\in \mathbf{X}_h^k\times M_h^k$ by $L^2$ projection
\begin{align}\label{eq:CNS:P_full_dis_L2proj2}
\langle\vec{m}_h^\mathrm{P},\vec{\theta}_h\rangle = \langle\rho_h^\mathrm{P}\vec{\vec{u}}_h^\mathrm{P},\vec{\theta}_h\rangle,~~
\forall \vec{\theta}_h \in \mathbf{X}_h^k
\quad\text{and}\quad
\langle E_h^\mathrm{P},\chi_h\rangle = \langle\rho_h^\mathrm{P}e_h^\mathrm{P},\chi_h\rangle + \langle\frac{\vec{m}_h^{\mathrm{P}}}{2\rho_h^\mathrm{P}},\vec{m}_h^{\mathrm{P}}\chi_h\rangle,~~ \forall \chi_h\in M_h^k.
\end{align}
Postprocess $\vec{U}_h^\mathrm{P}$ by the constraint optimization-based limiting strategy, see Section~\ref{sec:limiter}. Then the cell averages have positive states, and we can apply the Zhang--Shu positivity-preserving limiter to ensure that all point values at $S_K^{\mathrm{H, int}}$ and $S_K^{\mathrm{H, aux}}$ have positive density and internal energy.
\item[] Step~5. Given $\vec{U}_h^\mathrm{P}\in M_h^k\times \mathbf{X}_h^k\times M_h^k$, compute $\vec{U}_h^{n+1}\in M_h^k\times \mathbf{X}_h^k\times M_h^k$ by the DG method \eqref{eq:hyper_space_dis} with the positivity-preserving SSP Runge--Kutta \eqref{eq:rk3} \cite{zhang2010positivity, zhang2017positivity} using step size $\frac{\Delta t}{2}$. After each Runge--Kutta stage, apply the Zhang--Shu positivity-preserving limiter to ensure that all point values at $S_K^{\mathrm{H, int}}$ and $S_K^{\mathrm{H, aux}}$ have positive density and internal energy.
\end{itemize}
The $\vec{U}_h^0$ is obtained through the $L^2$ projection of the initial data $\vec{U}^0$, followed by postprocessing it with the Zhang--Shu limiter \cite{zhang2010positivity}. Thus, $\vec{U}_h^0$ belongs to the set of admissible states. In addition, we highlight in each time step only two decoupled linear systems \eqref{eq:CNS:P_full_dis2_1} and \eqref{eq:CNS:P_full_dis3} need to be solved sequentially.
\begin{remark}
For $\IQ^k$ scheme, the $\IQ^k$ Lagrangian basis functions defined at Gauss--Lobatto points are orthogonal at the $(k+1)^d$-point Gauss--Lobatto quadrature points. Thus, in Step~2 and Step~4, no linear systems need to be solved for computing the $L^2$ projection.
\end{remark}

\subsection{Global conservation of the fully discrete scheme}

We first discuss the global conservation of momentum and total energy. 
Notice that the local conservation for mass is naturally inherited from the Runge--Kutta DG method solving compressible Euler equations. 
For simplicity, we only discuss conservation in the context of periodic boundary conditions. It is straightforward to extend the discussion to many other types of boundary conditions, such as the ones implemented in the numerical tests in this paper.
\par
The following result is essentially the same as \cite[Theorem 1]{liu2023positivity}.
However, the time discretization used in this paper is the $\theta$-scheme for the internal energy equation, whereas the time discretization in \cite[Theorem 1]{liu2023positivity} is the backward Euler scheme. In addition, the spatial discretization in this paper is a DG scheme, while the spatial discretization in \cite{liu2023positivity} is a combination of DG and continuous finite element method. Thus, for completeness, we include the proof of the global conservation. 

\par
\begin{theorem}\label{thm:dis_momentum_conv}
Assume $\vec{U}_h^\mathrm{P}(\vec{x}_q)$ belongs to the set of admissible states for all $\vec{x}_q\in S_h$, then the fully discrete scheme conserves density, momentum, and total energy. We have
\begin{align*}
(\rho_h^{n},1) = (\rho_h^{n+1},1),\quad (\vec{m}_h^{n},\vec{1}) = (\vec{m}_h^{n+1},\vec{1}), \quad (E_h^{n},1) = (E_h^{n+1},1).
\end{align*}
\end{theorem}
\begin{proof}
Both the explicit Runge--Kutta DG method for hyperbolic subproblem ($\mathrm{H}$) and the Zhang--Shu limiter conserve mass, momentum, and total energy \cite{zhang2010positivity,zhang2017positivity}. We have 
\begin{align*}
(\rho_h^{n},1) = (\rho_h^{\mathrm{H}},1), \quad
(\vec{m}_h^{n},\vec{1}) = (\vec{m}_h^{\mathrm{H}},\vec{1}),\quad
(E_h^{n},1) = (E_h^{\mathrm{H}},1).
\end{align*}
It is easy to verify the discrete mass conservation, since $(\rho_h^{n+1},1) = (\rho_h^{\mathrm{P}},1)$ and we set $\rho_h^{\mathrm{H}} = \rho_h^{\mathrm{P}}$ in Step~3.
\par
For the discrete momentum conservation, we have $(\vec{m}_h^{n},\vec{1}) = (\vec{m}_h^{\mathrm{H}},\vec{1})$ and $(\vec{m}_h^{n+1},\vec{1}) = (\vec{m}_h^{\mathrm{P}},\vec{1})$.
For $\IQ^k$ scheme, the quadrature rules in subproblems ($\mathrm{H}$) and ($\mathrm{P}$) are both exact for integrating polynomials of degree $k$,
Thus, we also have $(\vec{m}_h^{\mathrm{H}},\vec{1}) = \langle\vec{m}_h^{\mathrm{H}},\vec{1}\rangle$ and $(\vec{m}_h^{\mathrm{P}},\vec{1}) = \langle\vec{m}_h^{\mathrm{P}},\vec{1}\rangle$.
Take $\vec{\theta}_h = \vec{1}$ in \eqref{eq:CNS:P_full_dis_L2proj1} and \eqref{eq:CNS:P_full_dis_L2proj2}, we get $\langle\vec{m}_h^{\mathrm{H}},\vec{1}\rangle = \langle\rho_h^\mathrm{H} \vec{u}_h^\mathrm{H},\vec{1}\rangle$ and $\langle\vec{m}_h^{\mathrm{P}},\vec{1}\rangle = \langle\rho_h^\mathrm{P} \vec{u}_h^\mathrm{P},\vec{1}\rangle$.
The above identities indicate $(\vec{m}_h^{n},\vec{1}) = \langle\rho_h^\mathrm{H} \vec{u}_h^\mathrm{H},\vec{1}\rangle$ and $(\vec{m}_h^{n+1},\vec{1}) = \langle\rho_h^\mathrm{P} \vec{u}_h^\mathrm{P},\vec{1}\rangle$. 
By selecting $\vec{\theta}_h = \vec{1}$ in \eqref{eq:CNS:P_full_dis2_1}, we obtain $\langle\rho_h^\mathrm{H} \vec{u}_h^\mathrm{H},\vec{1}\rangle = \langle\rho_h^\mathrm{P} \vec{u}_h^\mathrm{P},\vec{1}\rangle$, namely $(\vec{m}_h^{n},\vec{1}) = (\vec{m}_h^{n+1},\vec{1})$ holds.
\par
For the discrete energy conservation, notice the basis are numerically orthogonal and similar to above, we have $(E_h^{n},1) = \langle\rho_h^\mathrm{H} e_h^\mathrm{H}, 1\rangle + \frac{1}{2}\langle\rho_h^\mathrm{H}\vec{u}_h^{\mathrm{H}},\vec{u}_h^{\mathrm{H}}\rangle$ and $(E_h^{n+1},1) = \langle\rho_h^\mathrm{P} e_h^\mathrm{P},1\rangle + \frac{1}{2}\langle\rho_h^\mathrm{P}\vec{u}_h^{\mathrm{P}},\vec{u}_h^{\mathrm{P}}\rangle$.
Recall that $b_{\vec{\tau}}(\vec{\theta})=0$ and $b_\mathcal{D}(\chi)=0$ for periodic boundary conditions, thus
by \eqref{eq:CNS:P_full_dis2_2} and $\rho_h^\mathrm{H} = \rho_h^\mathrm{P}$, the \eqref{eq:CNS:P_full_dis2_1} can be written as
\begin{align*}
\langle\rho_h^\mathrm{P}\vec{u}_h^\mathrm{P},\vec{\theta}_h\rangle + \frac{\Delta t}{\Rey}a_\strain(\vec{u}^{\ast}_h,\vec{\theta}_h) + \frac{2\Delta t}{3\Rey}a_\lambda(\vec{u}^{\ast}_h,\vec{\theta}_h) &= \langle\rho_h^\mathrm{H}\vec{u}_h^\mathrm{H},\vec{\theta}_h\rangle.
\end{align*}
Plugging in $\vec{\theta}_h = (\vec{u}^{\mathrm{P}}_h + \vec{u}^{\mathrm{H}}_h)/2 = \vec{u}^{\ast}_h$, we have
\begin{align}\label{eq:conervation_proof_1}
\frac{1}{2}\langle\rho_h^\mathrm{P}\vec{u}_h^\mathrm{P},\vec{u}_h^\mathrm{P}\rangle + \frac{\Delta t}{\Rey}a_\strain(\vec{u}^{\ast}_h,\vec{u}^{\ast}_h) + \frac{2\Delta t}{3\Rey}a_\lambda(\vec{u}^{\ast}_h,\vec{u}^{\ast}_h) = \frac{1}{2}\langle\rho_h^\mathrm{H}\vec{u}_h^\mathrm{H},\vec{u}_h^\mathrm{H}\rangle.
\end{align}
Taking $\chi_h = 1$ in \eqref{eq:CNS:P_full_dis3}, we have
\begin{align*}
\langle\rho^\mathrm{P}_h e_h^\ast,1\rangle + \frac{\theta\Delta t\lambda}{\Rey}a_{\mathcal{D}}(e_h^\ast,1) = \langle\rho_h^\mathrm{H} e_h^\mathrm{H},1\rangle + \frac{\theta\Delta t}{\Rey}b_\strain(\vec{u}_h^{\ast},1) + \frac{2\theta\Delta t}{3\Rey}b_\lambda(\vec{u}_h^{\ast},1).
\end{align*}
Recall that $e^\ast = \theta e^\mathrm{P} + (1-\theta)e^\mathrm{H}$, we have
\begin{align}\label{eq:conervation_proof_2}
\langle\rho^\mathrm{P}_h e_h^\mathrm{P},1\rangle + \frac{\Delta t\lambda}{\Rey}a_{\mathcal{D}}(e_h^\ast,1) = \langle\rho_h^\mathrm{H} e_h^\mathrm{H},1\rangle + \frac{\Delta t}{\Rey}b_\strain(\vec{u}_h^{\ast},1) + \frac{2\Delta t}{3\Rey}b_\lambda(\vec{u}_h^{\ast},1).
\end{align}
Adding two equations \eqref{eq:conervation_proof_1} and \eqref{eq:conervation_proof_2}, with the fact that $a_{\mathcal{D}}(e_h^\ast,1) = 0$ and the identities $a_\strain(\vec{u}_h^{\ast},\vec{u}_h^{\ast}) = b_\strain(\vec{u}_h^{\ast},1)$ and $a_\lambda(\vec{u}_h^{\ast},\vec{u}_h^{\ast}) = b_\lambda(\vec{u}_h^{\ast},1)$,
we obtain
\begin{align*}
\langle\rho^{\mathrm{H}}_h e^{\mathrm{H}}_h, 1\rangle + \frac{1}{2}\langle\rho^{\mathrm{H}}_h\vec{u}^{\mathrm{H}}_h, \vec{u}_h^{\mathrm{H}}\rangle = \langle\rho^{\mathrm{P}}_h e^{\mathrm{P}}_h, 1\rangle + \frac{1}{2}\langle\rho^{\mathrm{P}}_h\vec{u}^{\mathrm{P}}_h, \vec{u}^{\mathrm{P}}_h\rangle.
\end{align*}
Therefore, we obtain $(E_h^{n},1) = (E_h^{n+1},1)$.
\end{proof}

\section{A globally conservative and positivity-preserving postprocessing procedure}\label{sec:limiter}

For  Runge–Kutta DG method solving the hyperbolic subproblem (H), i.e.,  compressible Euler equations, 
it is well understood that  the  simple Zhang--Shu limiter can preserve the positivity without destroying conservation and high order accuracy \cite{zhang2010positivity, zhang2017positivity}.
Let $S_h$ be the union of sets $S_K^{\mathrm{H, int}}$ and $S_K^{\mathrm{H, aux}}$ for all $K\in\setE_h$. By the results in \cite{zhang2010positivity, zhang2017positivity}, for Step 1 and Step 5 in the fully discrete scheme in Section \ref{sec:flowchart}, we have
\begin{enumerate}
\item The DG polynomial $\vec{U}_h^n(\vec{x}_q)\in G$ for all $\vec{x}_q\in S_h$ gives $\vec{U}_h^\mathrm{H}(\vec{x}_q)\in G$ for all $\vec{x}_q\in S_h$. 
\item If $\vec{U}_h^\mathrm{P}(\vec{x}_q)\in G$ for all $\vec{x}_q\in S_h$, then the DG polynomial $\vec{U}_h^{n+1}(\vec{x}_q)\in G$ for all $\vec{x}_q\in S_h$.
\end{enumerate}

Moreover, by \cite[Lemma 1]{liu2023positivity}, the $L^2$ projection step  
\eqref{eq:CNS:P_full_dis_L2proj1} in Step 2 does not affect the positivity, i.e., the  positivity of $e_h^\mathrm{H}$ is ensured if conserved variables are in the invariant domain. 
Therefore, in order to construct a conservative and positivity-preserving scheme, we only need to enforce $\vec{U}_h^\mathrm{P}(\vec{x}_q)\in G^\epsilon$ for all $\vec{x}_q\in S_h$ in Step~4 without affecting the global conservation in the fully discrete scheme in Section~\ref{sec:flowchart}.
\par
When using the backward Euler time discretization (e.g., $\theta=1$) in Step~3, positivity can be achieved if the discrete Laplacian is monotone \cite{liu2023positivity}. For example, the discrete Laplacian from $\IQ^1$ IIPG forms an M-matrix unconditionally. Moreover,  the monotonicity of $\IQ^k$ spectral element method (continuous finite element with Gauss--Lobatto quadrature) for $k= 1, 2, 3$ is proven in \cite{li2020monotonicity,cross2023monotonicityQ2,cross2023monotonicityQ3}, see also \cite{hu2023positivity,shen2021discrete, liu2024structure}, and such a result was used in
 \cite{liu2023positivity} for solving \eqref{eq:CNS:model}.
 
To improve the time accuracy, the Crank--Nicolson scheme with $\theta=\frac{1}{2}$ can be used in Step~3. However, in this case, a monotone system matrix no longer implies the positivity of internal energy, which poses a significant challenge, though positivity might still be ensured under a small time step $\Delta t=\mathcal O(\Rey \Delta{x}^2)$.
Instead, we consider a postprocessing procedure based on constraint optimization to ensure global conservation and positivity. The constraint optimization-based cell average limiter can be formulated as a nonsmooth convex minimization problem and efficiently solved by utilizing the generalized Douglas--Rachford splitting method \cite{liu2023simple}.

\subsection{A cell average postprocessing approach}
By Theorem \ref{thm:dis_momentum_conv}, the DG polynomial $\vec{U}_h^\mathrm{P}$ preserves the global conservation. But it may violate the positivity of internal energy. The following two-stage limiting strategy can be employed to enforce $\vec{U}_h^\mathrm{P}(\vec{x}_q)$ in the set of admissible states for any quadrature points $\vec{x}_q\in S_h$ without losing high order accuracy and conservation.
\begin{itemize}[leftmargin=0.5cm]
\item[] Step~1. Given $\vec{U}_h^\mathrm{P}$,  if any cell average has negative internal energy, then post process all cell averages of the total energy variable without losing global conservation such that each cell average of the DG polynomial $\vec{U}_h^\mathrm{P}$ stays in the admissible state set $G^\epsilon$.
\item[] Step~2. Apply the Zhang--Shu limiter to the postpocessed DG polynomial to ensure internal energy  at any quadrature points in $S_h$ is positive.
\end{itemize}

For a postprocessing procedure, minimal modifications to the original DG polynomial is often  preferred.
In our scheme, the density $\rho_h^\mathrm{P} = \rho_h^\mathrm{H}$ is already positive, ensured by a high order accurate positivity-preserving compressible Euler solver. Consider the scheme for solving the subproblem ($\mathrm{P}$), which is fully decoupled. The momentum $\vec{m}_h^\mathrm{P}$ or velocity $\vec{u}_h^\mathrm{P}$ is  stably approximated.
With the given $\rho_h^\mathrm{P}$ and $\vec{u}_h^\mathrm{P}$, when solving \eqref{eqn-internal-e}, which is a heat equation in the parabolic subproblem, a high order scheme may not preserve positivity in general.
To this end, we consider a simple approach by only post processing the total energy variable  $E_h^\mathrm{P}$ to enforce the positivity of internal energy, without losing conservation for $E_h^\mathrm{P}$.
\par
Let $K_i$ $(i=1,\cdots, N)$ be all the cells and $\overline{\vec{U}_i^\mathrm{P}} = \transpose{[\overline{\rho_i^\mathrm{P}},\overline{\vec{m}_i^\mathrm{P}}, \overline{E_i^\mathrm{P}}]}$ be a vector denoting the cell average of the DG polynomial $\overline{\vec{U}_h^\mathrm{P}}$ on the $i$-th cell $K_i$, namely $\overline{\vec{U}_i^\mathrm{P}} =\frac{1}{\abs{K_i}}\int_{K_i} \vec{U}_h^\mathrm{P}$.
 
Then we apply the globally conservative postprocessing procedure \eqref{postprocessing} only to the total energy DG polynomial such that the modified DG polynomials have good cell averages, which have positive internal energy.

\subsection{The accuracy of the postprocessing}
\label{sec:accuracy-pp}
It is obvious that the minimizer to \eqref{total-energy-opt} preserves the global conservation of total energy and the positivity of internal energy, since these two are the constraints. Next,  we discuss the accuracy of the postprocessing step \eqref{total-energy-opt}. 
\par
To understand how \eqref{total-energy-opt} affects accuracy,   consider evolving \eqref{eqn-internal-e} with given $\rho(\vec{x},t) = \rho^{\mathrm{P}}_h(\vec{x})$ and $\vec{u}(\vec{x},t)=\vec{u}_h^\ast(\vec{x})$, $\forall t$ by one time step in the Strang splitting \eqref{algorithm-splitting}, i.\,e., we consider the initial value problem
\begin{align}
\begin{cases} \rho_h^\mathrm{P} \partial_t e - \frac{\lambda}{\Rey}\laplace{e} = \frac{1}{\Rey}\vec{\tau}(\vec{u}_h^\ast):\grad{\vec{u}_h^\ast}, \quad t\in (t^n, t^n+\Delta t),\\
e(\vec{x}, t^n)=e^\mathrm{H}_h(\vec{x}).
\end{cases}
\label{e-PDE}
\end{align}
Due to the   inequality $\strain(\vec{u}):\strain(\vec{u})\geq\frac{1}{d}(\div{\vec{u}})^2$,  which can be easily verified by calculations (e.g., for $d=2$, $\strain(\vec{u}):\strain(\vec{u})\geq\frac{1}{2}(\div{\vec{u}})^2\Leftrightarrow  \frac14(u_x-u_y)^2+\frac12(u_x+u_y)^2\geq 0$), we know $\vec{\tau}(\vec{u}_h^\ast):\grad{\vec{u}_h^\ast}=2\big(\strain(\vec{u}_h^\ast):\strain(\vec{u}_h^\ast)-\frac{1}{3}(\div{\vec{u}_h^\ast})^2\big)\geq0$.
We mention that a similar property also holds for the interior penalty DG scheme at the discrete level, i.\,e., the right hand side of \eqref{eq:CNS:P_full_dis3} is also positive, see \cite[Lemma 3]{liu2023positivity}.  Let $e$ denote the exact solution to \eqref{e-PDE}. Since the right-hand side of \eqref{e-PDE} is non-negative, the exact solution to \eqref{e-PDE} with an initial condition $e^\mathrm{H}_h>0$ is positive, thus we assume $e(\vec{x},t)\geq \epsilon_2>0$.  
\par

Noticing that $\rho_h^\mathrm{P}$ is time independent, we have $\rho_h^\mathrm{P} \partial_t e = \partial_t (\rho_h^\mathrm{P} e)$. Integrating \eqref{e-PDE} over the spatial domain $\Omega$ and using boundary condition $\grad{e}\cdot\normal=0$, we get
\begin{align*}
\frac{\dd}{\dd t}\left(\int_\Omega \rho_h^\mathrm{P} e \dd\vec{x}\right)= \frac{1}{\Rey}\int_\Omega\vec{\tau}(\vec{u}_h^\ast):\grad{\vec{u}_h^\ast}\dd\vec{x}.
\end{align*}
Integrating the equation above over the time interval $[t^n, t^{n}+\Delta{t}]$, we have
\begin{align}
\label{e-integral}
 \int_\Omega \rho_h^\mathrm{P}(\vec{x}) e(\vec{x}, t^{n}+\Delta{t}) \dd\vec{x} = \int_\Omega \rho_h^\mathrm{H} e_h^\mathrm{H} \dd\vec{x} + \frac{\Delta{t}}{\Rey}\int_\Omega\vec{\tau}(\vec{u}_h^\ast):\grad{\vec{u}_h^\ast}\dd\vec{x}.
\end{align}
Consider the NIPG0 method for velocity, i.\,e., the NIPG method with zero penalty, which is the scheme \eqref{eq:CNS:P_full_dis3} we utilized in our numerical experiments. Recall $(k+1)^d$ Gauss--Lobatto quadrature is accurate for $(2k-1)$-order polynomial. Taking $\chi_h=1$ in \eqref{eq:CNS:P_full_dis3},  with \eqref{e-integral} and the quadrature error for integrals, we have
\begin{align*}
\int_\Omega \rho_h^\mathrm{P}(\vec{x}) e(\vec{x}, t^{n}+\Delta{t}) \dd\vec{x}= \langle\rho_h^\mathrm{P} e_h^\mathrm{P},1\rangle +Ch^{2k}.
\end{align*}
Let $  e_I(\vec{x})$ be the piecewise $\IQ^k$ interpolation polynomial of the exact solution $e(\vec{x}, t^n+\Delta t)$ at $(k+1)^d$ Gauss--Lobatto points at each cell.  We have
\begin{align*}
\langle\rho^\mathrm{P}_h  {e}_I, 1\rangle 
= \int_\Omega \rho_h^\mathrm{P}(\vec{x}) e(\vec{x}, t^{n}+\Delta{t}) \dd\vec{x} + Ch^{2k}
= \langle\rho^\mathrm{P}_h e_h^\mathrm{P},1\rangle +Ch^{2k}.
\end{align*}
Let $\tilde e_h(\vec{x}) = e_I(\vec{x})-\frac{C}{\langle\rho^\mathrm{P}_h,1\rangle} h^{2k}$, 
then $\tilde e_h(\vec{x})=e(\vec{x})+\mathcal O(h^{k+1})$ 
and $\langle \rho_h^{\mathrm{P}} \tilde e_h,1\rangle=\langle \rho_h^{\mathrm{P}} e^{\mathrm{P}}_h,1\rangle$. 
Define $(\vec{m}_h^\mathrm{P}, E^\mathrm{Interp}_h)\in \mathbf{X}_h^k\times M_h^k$ as an $L^2$ projection of $(\rho_h^{\mathrm{P}}, \vec{u}_h^{\mathrm{P}}, \tilde e_h)\in M_h^k\times \mathbf{X}_h^k\times M_h^k$:
\begin{align}
\langle\vec{m}_h^\mathrm{P},\vec{\theta}_h\rangle = \langle\rho_h^\mathrm{P}\vec{\vec{u}}_h^\mathrm{P},\vec{\theta}_h\rangle,~~
\forall \vec{\theta}_h \in \mathbf{X}_h^k
\quad\text{and}\quad
\langle E^\mathrm{Interp}_h,\chi_h\rangle = \langle \rho_h^{\mathrm{P}}\tilde  e_h,\chi_h\rangle + \langle\frac{\vec{m}_h^{\mathrm{P}}}{2\rho_h^\mathrm{P}},\vec{m}_h^{\mathrm{P}}\chi_h\rangle,~~ \forall \chi_h\in M_h^k.
\label{eq:CNS:P_full_dis_L2proj2-2}
\end{align}
Notice that $\vec{m}_h^\mathrm{P}$ in \eqref{eq:CNS:P_full_dis_L2proj2-2} is exactly the same as $\vec{m}_h^\mathrm{P}$ in \eqref{eq:CNS:P_full_dis_L2proj2}, and only $E^\mathrm{Interp}_h$ is different.


Let $\overline{E_i^\mathrm{Interp}}$ be the cell average of $E^\mathrm{Interp}_h$ at the $i$-th cell and $\overline{E_h^\mathrm{Interp}}=\transpose{[\overline{E_1^\mathrm{Interp}}, \overline{E_2^\mathrm{Interp}}, \cdots, \overline{E_N^\mathrm{Interp}}]}$.
Next, we verify that $\overline{E_h^\mathrm{Interp}}$ satisfies both constraints 
in \eqref{total-energy-opt}, when the mesh size $h$ is small. 
\begin{itemize}
\item First, by taking $\chi_h=1$ in \eqref{eq:CNS:P_full_dis_L2proj2} and \eqref{eq:CNS:P_full_dis_L2proj2-2}, we obtain the global conservation of total energy:
\[ \sum_{i=1}^N \overline{E_i^\mathrm{Interp}} |K_i| =\langle E^\mathrm{Interp}_h,1\rangle =\langle E^{\mathrm{P}}_h,1\rangle= \sum_{i=1}^N \overline{E_i^{\mathrm{P}}} |K_i|. \]
\item Second, for small enough $h$ such that $\frac{\abs{C}}{\langle\rho^\mathrm{P}_h,1\rangle} h^{2k}\leq  \frac12 \epsilon_2 $, we can take $\epsilon\leq  \frac12 \epsilon_2\rho_h^\mathrm{P}$ to have 
$$\left(\epsilon_2-\frac{|C|}{\langle\rho^\mathrm{P}_h,1\rangle} h^{2k}\right)\rho_h^\mathrm{P}\geq \frac12 \epsilon_2 \rho_h^\mathrm{P}\geq \epsilon.$$ 
\end{itemize}
Then following the proof of Lemma~2 in \cite[Section 3.2]{liu2023positivity}, we have 
\[\overline{E_i^\mathrm{Interp}} - \frac{1}{2}\frac{\|\overline{\vec{m}_i^{\mathrm{P}}}\|}{\overline{\rho^{\mathrm{P}}_i}}\geq \epsilon.\]
Since $\overline E_h^{\,\ast}$ is the minimizer to \eqref{total-energy-opt} and $\transpose{[\overline{\rho^{\mathrm{P}}}_i,\overline{\vec{m}^{\mathrm{P}}}_i,\overline{E_i^\mathrm{Interp}}]}$ satisfies the
constraints of \eqref{total-energy-opt}, we have
\begin{equation}
\left\| \overline E_h^{\,\ast}-\overline{E^\mathrm{Interp}_h}\right\| 
\leq \left\|\overline E_h^{\,\ast}-\overline{E_h^\mathrm{P}}\right\| + \left\|\overline{E_h^\mathrm{P}}-\overline{E^\mathrm{Interp}_h}\right\|
\leq 2 \left\|\overline{E_h^\mathrm{P}}-\overline{E^\mathrm{Interp}_h}\right\|. \label{eqn:totalenergy-error}
\end{equation}
\par
To summarize the discussion for accuracy, we conclude that the accuracy of the postprocessing \eqref{total-energy-opt} 
can be understood in the sense of \eqref{eqn:totalenergy-error}. In other words, if considering the error approximating the exact solution of \eqref{e-PDE} in Strang splitting, 
then the minimizer to \eqref{total-energy-opt} is not significantly worse than the DG solution $E^{\mathrm{P}}_h$.

\subsection{An efficient solver by Douglas--Rachford splitting with nearly optimal parameters}
\label{sec: DR-solver}
\par
The key computational issue here is how to solve \eqref{total-energy-opt} efficiently, and the same approach in \cite{liu2023simple} can be used. For completeness, we briefly describe the main algorithm and result in \cite{liu2023simple}. 
For convenience,
we rewrite the minimization problem \eqref{total-energy-opt} in matrix-vector form using different names for variables.

{\bf For simplicity, we only consider a uniform mesh with $|K_i|=h^d$}. Extensions to non-uniform meshes are straightforward. Thus we define a matrix $\vecc{A}=[1,1,\cdots,1]\in\IR^{1\times N}$, where $N$ is the total number of cells. A vector $\vec{w}\in\IR^{N}$ is introduced to store the cell averages of DG polynomial $E_h^\mathrm{P}$, namely the $i^\mathrm{th}$ entry of $\vec{w}$ equals $\overline{E_i^\mathrm{P}}$. 
Define the constant $b = \vecc{A}\vec{w}$, which is the summation of all cell averages.
The indicator function in constraint optimization is defined as $\iota_\Lambda$ for a set $\Lambda$: $\iota_\Lambda(\vec{x}) = 0$ if $\vec{x}\in \Lambda$ and $\iota_\Lambda(\vec{x}) = +\infty$ if $\vec{x}\notin \Lambda$. 
Then \eqref{total-energy-opt} is equivalent to the following minimization:
\begin{equation}\label{eq:opt_model2}
\min_{\vec{x}\in\IR^{N}}{\frac{\alpha}{2}\norm{\vec{x}-\vec{w}}{2}^2}
+ \iota_{\Lambda_1}(\vec{x})
+ \iota_{\Lambda_2}(\vec{x}).
\end{equation}
where
 $\alpha >0$ is a constant, and the conservation constraint and the positivity-preserving constraint give two sets 
\begin{align*}
\Lambda_1=\{\vec{x}: \vecc{A}\vec{x} = b\}
\quad\text{and}\quad
\Lambda_2=\{\vec{x}: x_i - \frac{\norm{\overline{\vec{m}}_i}{}^2}{2\overline{\rho}_i} \geq \epsilon,~ \forall i = 1,\, \cdots,\, N\}.
\end{align*}

Splitting algorithms naturally arise when solving minimization problem of the form $\min_{\vec{x}}f(\vec{x}) + g(\vec{x})$, where functions $f$ and $g$ are convex, lower semi-continuous (but not otherwise smooth), and have simple subdifferentials and resolvents.
Let $F = \partial{f}$ and $G = \partial{g}$ denote the subdifferentials of $f$ and $g$. Then, a sufficient and necessary condition for $\vec{x}$ being a minimizer is $\vec{0}\in F(\vec{x}) + G(\vec{x})$.
The resolvents $\mathrm{J}_{\gamma F} = (\mathrm{I}+\gamma F)^{-1}$ and $\mathrm{J}_{\gamma G} = (\mathrm{I}+\gamma G)^{-1}$ are also called proximal operators, as $\mathrm{J}_{\gamma F}$ maps $\vec{x}$ to $\mathrm{argmin}_{\vec{z}} \gamma f(\vec{z}) + \frac{1}{2}\norm{\vec{z}-\vec{x}}{2}^2$ and $\mathrm{J}_{\gamma G}$ is defined similarly. 
The reflection operators are defined as $\mathrm{R}_{\gamma F} = 2\mathrm{J}_{\gamma F} - \mathrm{I}$ and $\mathrm{R}_{\gamma G} = 2\mathrm{J}_{\gamma G} - \mathrm{I}$, where $\mathrm{I}$ is the identity operator.
\par
The generalized Douglas--Rachford splitting method for solving the minimization problem $\min_{\vec{x}}f(\vec{x}) + g(\vec{x})$ can be written as: 
\begin{equation}\label{eq:DR_algorithm}
\begin{cases}\displaystyle
\vec{y}^{k+1} 
= \lambda\frac{\mathrm{R}_{\gamma F}\mathrm{R}_{\gamma G} + \mathrm{I}}{2} \vec{y}^k + (1-\lambda) \vec{y}^k, \\
\vec{x}^{k+1} = \mathrm{J}_{\gamma G}(\vec{y}^{k+1}),
\end{cases}
\end{equation}
where $\vec{y}$ is an auxiliary variable, $\lambda$ belongs to $(0,2]$ is a parameter, and $\gamma>0$ is step size. 
We get the Douglas--Rachford splitting when $\lambda=1$ in \eqref{eq:DR_algorithm}. In the limiting case, $\lambda=2$ is the Peaceman--Rachford splitting.
For two convex functions $f(\vec x)$ and $g(\vec x)$, the sequence in \eqref{eq:DR_algorithm} converges for any positive step size $\gamma$ and any fixed $\lambda\in(0,2)$, see \cite{lions1979splitting}. If one function is strongly convex, then $\lambda=2$ also leads to convergence.
Using the definition of reflection operators, \eqref{eq:DR_algorithm} can be expressed as follows:
\begin{equation}\label{eq:DR_algorithm1}
\begin{cases}\displaystyle
\vec{y}^{k+1} 
= \lambda\mathrm{J}_{\gamma F}(2\vec{x}^k - \vec{y}^k) + \vec{y}^k - \lambda\vec{x}^k, \\
\vec{x}^{k+1} = \mathrm{J}_{\gamma G}(\vec{y}^{k+1}).
\end{cases}
\end{equation}
\par
We split the objective function in \eqref{eq:opt_model2} into 
\begin{align*}
f(\vec{x}) = \frac{\alpha}{2}\norm{\vec{x}-\vec{w}}{}^2 + \iota_{\Lambda_1}(\vec{x})
\quad\text{and}\quad
g(\vec{x}) = \iota_{\Lambda_2}(\vec{x}).
\end{align*}
Linearity implies that the set $\Lambda_1$ is convex. With ideal gas equation of state, the function $\rho e$ is concave, see \cite{zhang2010positivity, zhang2017positivity} and references therein. Thus, by Jensen's inequality, the set $\Lambda_2$ is also convex. Therefore, the function $f$ is strongly convex and the function $g$ is convex, given that \eqref{eq:DR_algorithm1} converges to the unique minimizer. 
After applying \eqref{eq:DR_algorithm1} to solve the minimization to machine precision, the positivity constraint is strictly satisfied and the conservation constraint is enforced up to round-off error.
The subdifferentials and the associated resolvents are given as follows:
\begin{itemize}
\item The subdifferential of function $f$ is 
\begin{align*}
\partial{f}(\vec{x}) = \alpha(\vec{x} - \vec{w}) + \mathcal{R}(\transpose{\vecc{A}}),
\end{align*}
where $\mathcal{R}(\transpose{\vecc{A}})$ denotes the range of the matrix $\transpose{\vecc{A}}$.
\item The subdifferential of function $g$ is 
\begin{align*}
[\partial{g}(\vec{x})]_i =
\begin{cases}
0,           & \text{if}~ x_i > \frac{\norm{\overline{\vec{m}}_i}{}^2}{2\overline{\rho}_i} + \epsilon,\\
[-\infty,0], & \text{if}~ x_i = \frac{\norm{\overline{\vec{m}}_i}{}^2}{2\overline{\rho}_i} + \epsilon.
\end{cases}
\end{align*}
\item For the function $f(\vec{x}) = \frac{\alpha}{2}\norm{\vec{x}-\vec{w}}{2}^2 + \iota_{\Lambda_1}(\vec{x})$, the associated resolvent is
\begin{align}\label{eq:resolvent_F}
\mathrm{J}_{\gamma F}(\vec{x}) 
= \frac{1}{\gamma\alpha+1}\big(\vecc{A}^+(b - \vecc{A}\vec{x}) + \vec{x}\big) + \frac{\gamma\alpha}{\gamma\alpha+1}\vec{w},
\end{align}
where $\vecc{A}^{+} = \transpose{\vecc{A}}(\vecc{A}\transpose{\vecc{A}})^{-1}$ denotes the pseudo inverse of the matrix $\vecc{A}$. 
\item For the function $g(\vec{x}) = \iota_{\Lambda_2}(\vec{x})$, the associated resolvent is $\mathrm{J}_{\gamma G}(\vec{x}) = \mathrm{S}(\vec{x})$, where $\mathrm{S}$ is a cut-off operator defined by 
\begin{align}\label{eq:resolvent_G}
[\mathrm{S}(\vec{x})]_i = \max{\Big(x_i, \frac{\norm{\overline{\vec{m}}_i}{}^2}{2\overline{\rho}_i} + \epsilon\Big)},\quad \forall i = 1, \cdots, N.
\end{align}
\end{itemize}
Define parameter $c = \frac{1}{\gamma\alpha+1}$, which gives $\frac{\gamma\alpha}{\gamma\alpha+1} = 1-c$. Using the expressions of resolvents in \eqref{eq:resolvent_F} and \eqref{eq:resolvent_G}, we obtain the generalized Douglas--Rachford splitting method for solving the minimization problem \eqref{eq:opt_model2} in matrix-vector form:
\begin{align}\label{eq:DR_algorithm2}
\begin{cases}\displaystyle
\vec{z}^k = 2\vec{x}^k - \vec{y}^k,\\
\vec{y}^{k+1} = \lambda c \big(\vecc{A}^+(b - \vecc{A}\vec{z}^k) + \vec{z}^k\big) + \lambda(1-c) \vec{w} + \vec{y}^k - \lambda\vec{x}^k, \\
\vec{x}^{k+1} = \mathrm{S}(\vec{y}^{k+1}).
\end{cases}
\end{align}
As a brief summary, after obtaining the DG polynomial $E_h^\mathrm{P}$, compute cell averages to generate vector $\vec{w}$, where the $i^\mathrm{th}$ entry of $\vec{w}$ equals $\overline{E_h^\mathrm{P}}|_{K_i}$, then our cell average limiter can be implemented as follows. 
\begin{itemize}[leftmargin=0.5cm]
\item[] {\bf Algorithm~DR}. To start the generalized Douglas--Rachford iteration, set $\vec{y}^0 = \vec{w}$, $\vec{x}^0 = \mathrm{S}(\vec{w})$, and $k=0$. Compute parameters $c$ and $\lambda$ by using formula in Remark~\ref{opt-parameter}, and select a small $\epsilon$ for numerical tolerance of the conservation error.
\item[] Step~1. Compute intermediate variable $\vec{z}^k = 2\vec{x}^k - \vec{y}^k$.
\item[] Step~2. Compute auxiliary variable $\vec{y}^{k+1} = \lambda c \big(\vecc{A}^+(b - \vecc{A}\vec{z}^k) + \vec{z}^k\big) + \lambda(1-c) \vec{w} + \vec{y}^k - \lambda\vec{x}^k$.
\item[] Step~3. Compute $\vec{x}^{k+1} = \mathrm{S}(\vec{y}^{k+1})$.
\item[] Step~4.  It is convenient to employ the norm $\norm{\cdot}{h} = h^{d/2}\norm{\cdot}{}$ to measure the conservation error, which is an approximation to the $L^2$-norm. If stopping criterion $\norm{\vec{y}^{k+1} - \vec{y}^k}{h} < \epsilon$ is satisfied, then terminate and output $\vec{x}^\ast = \vec{x}^{k+1}$, otherwise set $k\leftarrow k+1$ and go to Step~1.
\end{itemize}
In the algorithm above,  $2\vec{x}^k$ can be regarded as $\vec{x}^k+\vec{x}^k$; the $\lambda(1-c)\vec w$ remains unchanged during iteration; and each entry of $\vecc{A}^+(b - \vecc{A}\vec{z}^k) + \vec{z}^k$ can be computed by $z^k_i + \frac{1}{N}(b-\sum_iz_i^k)$, thus if only counting number of computing multiplications and taking maximum, the computational complexity of each iteration is $3N+1$.
\begin{remark}\label{opt-parameter}
The analysis in \cite{liu2023simple} proves the asymptotic linear convergence and suggests a simple choice of nearly optimal parameters $c$ and $\lambda$ in \eqref{eq:DR_algorithm2}. Let $\hat r$ be the number of bad cells defined by $\overline{\vec{U}_i^\mathrm{P}}\notin G^\epsilon$ and let $\hat \theta=\cos^{-1}\sqrt{\frac{\hat r}{N}}$, then we have:
\begin{align}
\label{DR_parameter}
\begin{cases}
c=\frac{1}{2},~ \lambda = \frac{4}{2-\cos{(2\hat \theta)}}, &\quad \mbox{if } \hat \theta \in(\frac{3}{8}\pi,\frac{1}{2}\pi],\\
c=\frac{1}{(\cos\hat\theta + \sin\hat\theta)^2},~ \lambda =\frac{2}{1+\frac{1}{1+\cot\hat\theta}-\frac{1}{(\cos\hat\theta + \sin\hat\theta)^2}}, &\quad \mbox{if } \hat \theta \in(\frac{1}{4}\pi,\frac{3}{8}\pi],\\
c=\frac{1}{(\cos\hat\theta + \sin\hat\theta)^2},~ \lambda =2,  &\quad \mbox{if } \hat \theta \in(0,\frac{1}{4}\pi].
\end{cases}
\end{align}
\end{remark}
\begin{remark}
\label{rmk-3}
By splitting the   Navier-Stokes system into the Euler system and a parabolic system,
to preserve positivity, we may postprocess  only  a scalar variable, i.e, the total energy, which is the main advantage of the splitting approach. It is also possible to postprocess the DG solutions to the convection-diffusion Navier-Stokes system, by a similar Douglas--Rachford cell average limiter to preserve the invariant domain $G^\epsilon$, for which the operator $\mathrm{J}_{\gamma G}$ in \eqref{eq:DR_algorithm} becomes the projection to admissible set $G^\epsilon$.  
\end{remark}
\begin{remark}
\label{rmk-4}
Compared to other alternative methods for solving \eqref{eq:opt_model2} such as  breakpoint searching algorithms \cite{kiwiel2008breakpoint} and the method of Lagrangian multipliers in the Appendix, the  Douglas--Rachford algorithm \eqref{eq:DR_algorithm} is more flexible for other minimization models such as replacing the $\ell^2$-norm in \eqref{eq:opt_model2} by the $\ell^1$-norm, for which the Shrinkage operator would appear in  $\mathrm{J}_{\gamma F}$.
\end{remark}

\subsection{Implementation}
We provide details on implementing our scheme. 
The time-stepping strategy employed to solve subproblem ($\mathrm{H}$) is identical to the one described in Section~3.2 of \cite{wang2012robust}. For the sake of completeness, we include a list of the steps below.
\begin{itemize}[leftmargin=0.5cm]
\item[] {\bf Algorithm~H}. At time $t^n$, select a trial hyperbolic step size $\Delta{t}^\mathrm{H}$. The parameter $\epsilon$ is a prescribed small positive number for numerical admissible state set $G^\epsilon$.
The input DG polynomial $\vec{U}_h^n$ satisfies $\vec{U}_h^n(\vec{x}_q)\in G^\epsilon$, for all $\vec{x}_q\in S_h$. 
\item[] Step~H1. Given DG polynomial $\vec{U}_h^n$, compute the first stage to obtain $\vec{U}_h^{(1)}$.
\begin{itemize}[leftmargin=1.0cm, label=\labelitemi]
\item If the cell averages $\overline{\vec{U}}_K^{(1)} \in G^\epsilon$, for all $K\in\setE_h$, then apply Zhang--Shu limiter described in Section~\ref{sec:simplelimiter} to obtain $\widetilde{\vec{U}}_h^{(1)}$ and go to Step H2.
\item Otherwise, recompute the first stage with halved step size $\Delta{t}^\mathrm{H} \leftarrow \frac{1}{2}\Delta{t}^\mathrm{H}$. Notice, when $\Delta{t}^\mathrm{H}$ satisfies the positivity-preserving hyperbolic CFL proven in \cite{zhang2010positivity} (see also \cite{zhang2017positivity}), the $\overline{\vec{U}}_K^{(1)} \in G^\epsilon$ is guaranteed.
\end{itemize}
\item[] Step~H2. Given DG polynomial $\widetilde{\vec{U}}_h^{(1)}$, compute the second stage to obtain $\vec{U}_h^{(2)}$.
\begin{itemize}[leftmargin=1.0cm, label=\labelitemi]
\item If the cell averages $\overline{\vec{U}}_K^{(2)} \in G^\epsilon$, for all $K\in\setE_h$, then apply Zhang--Shu limiter to obtain $\widetilde{\vec{U}}_h^{(2)}$ and go to Step H3.
\item Otherwise, return to Step H1 and restart the computation with halved step size $\Delta{t}^\mathrm{H} \leftarrow \frac{1}{2}\Delta{t}^\mathrm{H}$. Notice that the results proven in   \cite{zhang2010positivity} ensure that there is not an infinite restarting loop, see \cite{zhang2017positivity}.
\end{itemize}
\item[] Step~H3. Given DG polynomial $\widetilde{\vec{U}}_h^{(2)}$, compute the third stage to obtain $\vec{U}_h^{(3)}$.
\begin{itemize}[leftmargin=1.0cm, label=\labelitemi]
\item If the cell averages $\overline{\vec{U}}_K^{(3)} \in G^\epsilon$, for all $K\in\setE_h$, then apply Zhang--Shu limiter to obtain $\vec{U}_h^{\mathrm{H}}$. We finish the current SSP Runge--Kutta.
\item Otherwise, return to Step H1 and restart the computation with halved step size $\Delta{t}^\mathrm{H} \leftarrow \frac{1}{2}\Delta{t}^\mathrm{H}$.  Notice that the results proven in   \cite{zhang2010positivity} ensure that there is not an infinite restarting loop, see \cite{zhang2017positivity}.
\end{itemize}
\end{itemize}
The time-stepping strategy for solving the compressible NS equations is as follows. The initial condition $\vec{U}_h^0$ is constructed by $L^2$ projection of $\vec{U}^0$ with Zhang--Shu limiter on $S_h$, e.g., we have $\vec{U}_h^0(\vec{x}_q)\in G^\epsilon$, for all $\vec{x}_q\in S_h$.
\begin{itemize}[leftmargin=0.5cm]
\item[] {\bf Algorithm~CNS}. At time $t^n$, select a desired time step size $\Delta{t}$. The parameter $\epsilon$ is a prescribed small positive number for numerical admissible state set $G^\epsilon$. The input DG polynomial $\vec{U}_h^n$ satisfies $\vec{U}_h^n(\vec{x}_q)\in G^\epsilon$, for all $\vec{x}_q\in S_h$.
\item[] Step~CNS1. Given DG polynomial $\vec{U}_h^n$, solve subproblem $(\mathrm{H})$ form time $t^n$ to $t^{n} + \frac{\Delta t}{2}$.
\begin{itemize}[leftmargin=1.0cm, label=\labelitemi]
\item Set $m=0$. Let $t^{n,0} = t^{n}$ and $\vec{U}_h^{n,0} = \vec{U}_h^n$.
\item Given $\vec{U}_h^{n,m}$ at time $t^{n,m}$, solve $(\mathrm{H})$ to compute $\vec{U}_h^{n,m+1}$ by the Algorithm~H. Let $t^{n,m+1} = t^{n,m} + \Delta{t}^\mathrm{H}$. 
If $t^{n,m+1} = t^{n} + \frac{\Delta t}{2}$, then apply Zhang--Shu limiter for $\vec{U}_h^{n,m+1}$ on all Gauss--Lobatto points in $S^\mathrm{P}_K$, for all $K\in\setE_h$, we obtain $\vec{U}_h^\mathrm{H}$. Go to Step~CNS2. Otherwise, set $m\leftarrow m+1$ and repeat solving $(\mathrm{H})$ by Algorithm~H until reaching $t^{n} + \frac{\Delta t}{2}$. 
Let $L$ be the smallest integer satisfying $2L-3\geq k$ for $\IQ^k$ basis, when using  $\IQ^k$ DG method to compute $\vec{U}_h^{n,m+1}$, we can take 
\begin{align*}
\Delta{t}^\mathrm{H} = \min\left\{a\frac{1}{\max_e \alpha_e}\frac{1}{L(L-1)}\Delta{x},~ t^{n} + \frac{\Delta t}{2}-t^{n,m}\right\}
\end{align*}
as a trial hyperbolic step size to start Algorithm~H. We refer to \cite{zhang2017positivity} for choosing the value of parameter $a$ on above.
\end{itemize}
Step~CNS2. Given DG polynomial $\vec{U}_h^\mathrm{H}$, take $L^2$ projection to compute $(\vec{u}_h^\mathrm{H}, e_h^\mathrm{H})$.
\item[] Step~CNS3. Given DG polynomials $(\rho_h^\mathrm{H}, \vec{u}_h^\mathrm{H}, e_h^\mathrm{H})$, solve subproblem $(\mathrm{P})$ form time $t^n$ to $t^{n} + \Delta t$. 
\item[] Step~CNS4. Given DG polynomials $(\rho_h^\mathrm{P}, \vec{u}_h^\mathrm{P}, e_h^\mathrm{P})$, take $L^2$ projection to compute $\vec{U}_h^\mathrm{P}$.
\begin{itemize}[leftmargin=1.0cm, label=\labelitemi]
\item  Notice that the postprocessing \eqref{postprocessing} can be applied to either the whole computational domain or a large enough local region containing negative cells. When possible, first define a local region of trouble cells defined by $\overline{\vec{U}_i^\mathrm{P}}\notin G^\epsilon$. Let $T\subseteq\{1,2,\cdots, N\}$
be the indices of the local region containing all cells with negative averages $\overline{\vec{U}_i^\mathrm{P}}\notin G^\epsilon$, 
and let $|T|$ be the number of cells in the local region marked by indices in the set $T$. Then the postprocessing on the local region is given by
\begin{subequations}
    \label{postprocessing-local}
\begin{align}
\label{total-energy-opt-2}
\min_{\overline E_i}
\sum_{i\in T}\left|\overline E_i-\overline{E_i^\mathrm{P}}\right|^2 
~\text{subjects to}~ 
\sum_{i\in T} \overline E_i |K_i| = \sum_{i\in T} \overline{E_i^\mathrm{P}} |K_i|
~~~\text{and}~~~
\transpose{[\overline{\rho_i^\mathrm{P}},\overline{\vec{m}_i^\mathrm{P}}, \overline E_i]}\!\! \in G^\epsilon,~ \forall i \in T.
\end{align}
Let $\overline E_h^{\,\ast}=\transpose{[\overline E_1^{\,\ast}, \cdots, \overline E_N^{\,\ast}]}$ be the minimizer. Then  we correct the DG polynomial cell averages for the total energy variable by a constant
\begin{equation}\label{update-average-limiter-2}
E_i(\vec{x})=E_i^\mathrm{P}(\vec{x})-\overline{E^\mathrm{P}_i}+\overline E_i^\ast,\quad \forall i\in T.
\end{equation} 
\end{subequations}
Notice that $T$ cannot contain only the negative cells, which will cause the feasible set in \eqref{total-energy-opt-2} to be empty, i.e., it is impossible to modify only negative cells to achieve positivity, without affecting conservation. If it is difficult to define such a set $T$, we can simply take $T=\{1,2,\cdots, N\}$, i.e., the whole computational domain. For certain problems, it is straightforward to define a proper $T$, see the remark below. 

\item  Solve \eqref{total-energy-opt-2} for  the region defined by indices in $T$ by the Douglas--Rachford splitting algorithm \eqref{eq:DR_algorithm2} with nearly optimal parameters \eqref{DR_parameter}  using  $\hat\theta=\cos^{-1}\sqrt{\frac{\hat r}{|T|}}$.
Then update or postprocess the cell averages of the DG polynomial $\vec{U}_h^\mathrm{P}$
by  \eqref{update-average-limiter-2}.   
\item With positive cell averages $\overline{\vec{U}_i^\mathrm{P}}\in G^\epsilon$ ensured by the postprocessing step \eqref{postprocessing}, we can apply the Zhang--Shu limiter to $\vec{U}_h^\mathrm{P}$ to ensure positivity on all points in $S_h$.
\end{itemize}
\item[] Step~CNS5. Given DG polynomial $\vec{U}_h^\mathrm{P}$, use adaptive time-stepping strategy to solve subproblem $(\mathrm{H})$ form time $t^n + \frac{\Delta t}{2}$ to $t^{n} + \Delta{t}$.
\end{itemize}

\begin{remark}
For the sake of robustness and efficiency, whenever possible, one should apply the postprocessing \eqref{postprocessing-local} to a subset of cells (i.e., $T$ is a strict subset of $\{1,2,\cdots, N\}$)  containing all trouble cells and also some good cells,  rather than the whole computational domain (i.e., $T=\{1,2,\cdots, N\}$). 
For example, in the 2D Sedov blast wave test in Section~\ref{sec:numercal_experiment:Sedov}, the initial total energy is $10^{-12}$ everywhere except in the cell at the lower left corner, and we can define $T$ as
\begin{equation}
\label{definition-T}
T = \left\{i: \mathrm{either}\quad  \overline{\vec{U}_i^\mathrm{P}}\notin G^\epsilon \quad \mathrm{or}\quad \overline{E^\mathrm{P}_i}-\frac{1}{2}\|\overline{\vec{m}^\mathrm{P}_i}\|/\overline{\rho^\mathrm{P}_i} \geq 10^{-10}\right \}.
\end{equation}
By such a definition of $T$ for each time step, the gray region in the Figure~\ref{fig:sedov_bad_cell_asymp_rate} will not be modified by the postprocessing.
Note, the number of cells contained in $T$ may various at each time step.
%
\begin{figure}[ht!]
\begin{center}
\begin{tabularx}{0.875\linewidth}{@{}c@{~~}c@{\quad\quad}c@{}}
\includegraphics[width=0.225\textwidth]{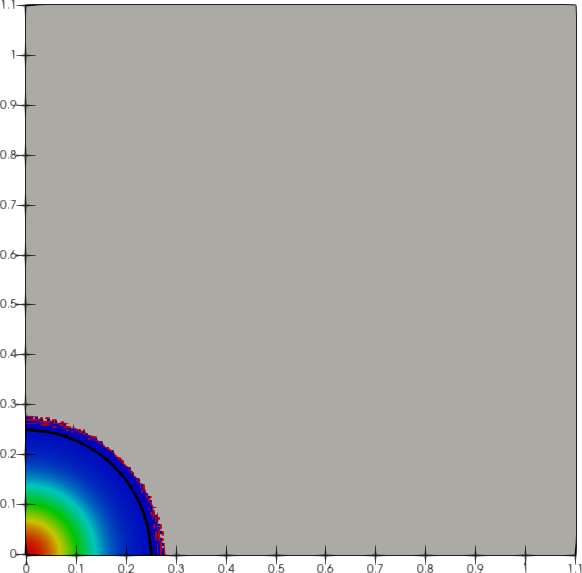} &
\includegraphics[width=0.2275\textwidth]{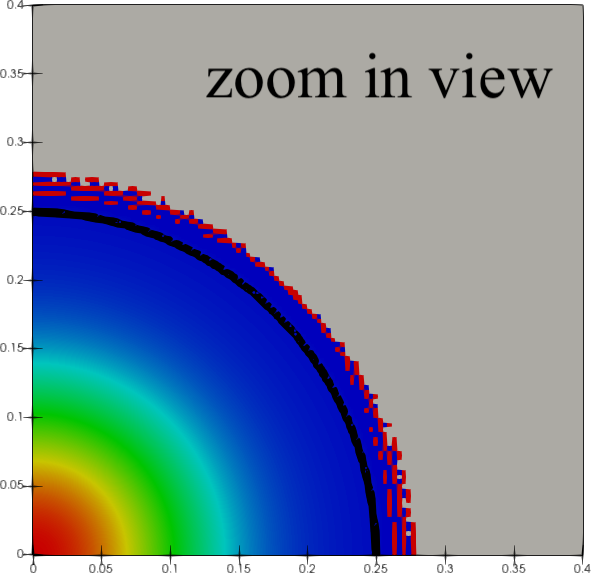} &
\includegraphics[width=0.35\textwidth]{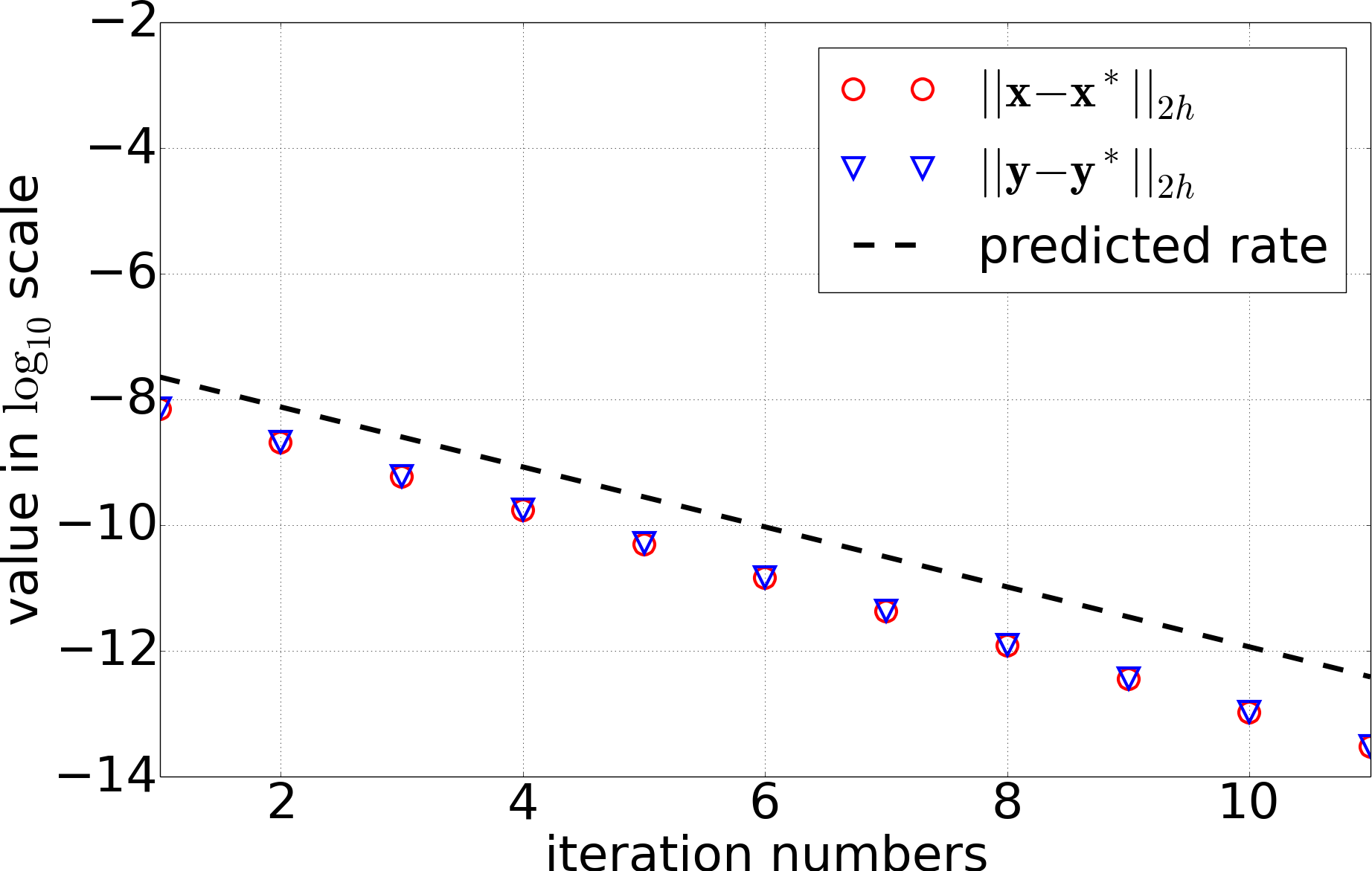} \\
\end{tabularx}
\caption{DG with $\IQ^2$ basis for 2D Sedov blast wave test. The middle figure is the zoom view of the left figure: the shock is marked black; the negative cells are highlighted by  the red marks; by the definition \eqref{definition-T}, $T$ does not include cells in the gray region in which the exact solution is supposed to be a constant. 
Right: the actual convergence rate of the Douglas--Rachford splitting algorithm \eqref{eq:DR_algorithm2} with nearly optimal parameters \eqref{DR_parameter} for solving \eqref{total-energy-opt-2} for the 2D Sedov problem (at one particular time step for the left figure) matches well the predicated rate from analysis (asymptotic linear convergence from analysis using the estimated principle angle $\hat\theta=\cos^{-1}\sqrt{\frac{\hat r}{|T|}}$), see \cite{liu2023simple} for more details on such a provable convergence rate.}
\label{fig:sedov_bad_cell_asymp_rate}
\end{center}
\end{figure}
\end{remark}

\section{Numerical experiments}\label{sec:numercal_experiment}

In this section, we validate our full numerical scheme through representative two-dimensional benchmark tests, including the Lax shock tube, double rarefraction, Sedov blast wave, shock diffraction, shock reflection-diffraction, and high Mach number astrophysical jet problems.
\par
For penalty parameters in interior penalty DG method for solving ($\mathrm{P}$), in the $\IQ^1$ scheme, we set $\sigma = 2$ on $\Gamma_h$, $\sigma = 4$ on $\partial{\Omega}$, and $\tilde{\sigma} = 2$; in the $\IQ^k$ ($k\geq2$) schemes, we set $\sigma = 0$ on all faces, namely using NIPG0 method for the velocity, and $\tilde{\sigma} = 2^k$ for the internal energy.
We take $\epsilon = 10^{-13}$ as the lower bound for the numerical admissible state set in all tests except the astrophysical jet simulations, where $\epsilon = 10^{-8}$ is used.
The ideal gas constant is $\gamma = 1.4$ and the Prandtl number is $\Pr = 0.72$.
{\bf The Reynolds number for all tests is $\Rey = 1000$ unless otherwise specified}.
\par 
In all physical simulations, we use $\theta = \frac{1}{2}$ in \eqref{eq:time_discretization}, namely utilizing the second order Crank--Nicolson method to solve ($\mathrm{P}$). The postprocessing step for total energy variable after solving ($\mathrm{P}$) is only triggered in the accuracy test in Section \ref{sec:test2}, the Sedov blast wave test, and astrophysical jets test.

\subsection{Accuracy tests}
We verify the order of accuracy of our numerical scheme by utilizing the method of manufactured smooth solutions. Let the computational domain $\Omega = [0,1]^2$ and select the end time $T = 0.1024$. The prescribed non-polynomial solutions are as follows:
\begin{align*}
\rho &= \exp{(-t)} \sin{2\pi (x+y)} + 2,\\
\vec{u} &= 
\begin{bmatrix}
\exp{(-t)} \cos{(2\pi x)} \sin{(2\pi y)} + 2\\
\exp{(-t)} \sin{(2\pi x)} \cos{(2\pi y)} + 2
\end{bmatrix},\\
e &= \frac{1}{2}\exp{(-t)} \cos{(2\pi (x+y))} + 1.
\end{align*}
Taking Reynolds number $\Rey=1$ and parameter $\lambda = 1$ in \eqref{eq:CNS:model}, the boundary conditions and the right-hand side of the compressible NS equations are computed by above manufactured solutions.
Define the discrete $L^2_h$ error of density by
\begin{align*}
\norm{\rho_h^{n}-\rho(t^n)}{L^2_h}^2 = {\Delta x}^2 \sum_{i=1}^{N} \sum_{\nu=1}^{N_\mathrm{q}^\mathrm{H, vol}}\omega_\nu \Big|\sum_{j=1}^{\Nloc}\rho_{ij}^n\,\hat{\varphi}_j(\hat{\vec{q}}_\nu) - \rho(t^n)\circ \vec{F}_i(\hat{\vec{q}}_\nu)\Big|^2,
\end{align*}
where $\omega_\nu$ and $\hat{\vec{q}}_\nu$ are the Gauss quadrature weights and points used in evaluating volume integrals in ($\mathrm{H}$).
The discrete $L^2_h$ errors for momentum and total energy are measured similarly.
In addition, the discrete $L^2_h$ for $\vec{U}_h^n$ is defined by
\begin{align*}
\norm{\vec{U}_h^n-\vec{U}(t^n)}{L^2_h}^2 = \norm{\rho_h^n-\rho(t^n)}{L^2_h}^2 + \norm{\vec{m}_h^n-\vec{m}(t^n)}{L^2_h}^2 + \norm{E_h^n-E(t^n)}{L^2_h}^2.
\end{align*}
If $\mathtt{err}_{\Delta x}$ denotes the error on a mesh with resolution $\Delta x$, then the rate is given by $\ln(\mathtt{err}_{\Delta x}/\mathtt{err}_{\Delta x/2})/\ln{2}$.
\par
For temporal convergence rate tests, we use $\IQ^3$ scheme and fix the mesh resolution $\Delta x = 1/64$ small enough such that the time error dominates. 
We choose NIPG method with $\sigma = 0$ to solve the second equation in subproblem ($\mathrm{P}$) and choose IIPG method with $\tilde\sigma = 8$ to solve the third equation in subproblem ($\mathrm{P}$). We observe the optimal temporal convergence rates, see Table~\ref{tab:accuravy_test_time}.
\par
For spatial convergence rate tests, we use $\theta=\frac{1}{2}$ and fix time step size $\Delta t = 3.125\times10^{-6}$ small enough such that the spatial error dominates and the hyperbolic CFL is satisfied. 
We choose NIPG method with $\sigma = 2$ on $\Gammah$ and $\sigma = 4$ on $\partial{\Omega}$ for $\IQ^1$ scheme; and $\sigma = 0$ for $\IQ^k$ ($k \geq 2$) scheme to solve the second equation in subproblem ($\mathrm{P}$). We choose IIPG method with $\tilde\sigma = 2^k$ to solve the third equation in subproblem ($\mathrm{P}$).
For $\IQ^1$, $\IQ^3$, and $\IQ^5$ schemes, we obtain the optimal spatial convergence rates, see Table~\ref{tab:accuravy_test_space}. For $\IQ^2$ and $\IQ^4$ schemes, the convergence is suboptimal, which is as expected, since the NIPG and IIPG methods are suboptimal for even order spaces.
\begin{table}[ht!]
\centering
\begin{tabularx}{\linewidth}{@{~~}c@{~~}|c@{~}C@{~~}|c@{~}C@{~}|c@{~~}|c@{~}C@{~}|c@{~~}}
\toprule
$\theta$ & $\Delta t$ & $\|\vec{U}_h^{N_T}-\vec{U}(T)\|_{L^2_h}$ & $\Delta t$ & $\|\vec{U}_h^{N_T}-\vec{U}(T)\|_{L^2_h}$ & rate & $\Delta t$ & $\|\vec{U}_h^{N_T}-\vec{U}(T)\|_{L^2_h}$ & rate \\
\midrule
1 & $4\cdot10^{-4}$ & $1.599\cdot10^{-2}$ & $2\cdot10^{-4}$ & $7.988\cdot10^{-3}$ & 1.001 & $1\cdot10^{-4}$ & $3.997\cdot10^{-3}$ & 0.999 \\
\midrule
$\frac{1}{2}$ & $4\cdot10^{-4}$ & $1.393\cdot10^{-3}$ & $2\cdot10^{-4}$ & $3.601\cdot10^{-4}$ & 1.952 & $1\cdot10^{-4}$ & $9.140\cdot10^{-5}$ & 1.978 \\
\bottomrule
\end{tabularx}
\caption{Test of accuracy. The temporal error and convergence rates. $\theta=1$ backward Euler scheme for internal energy in subproblem ($\mathrm{P}$). $\theta=\frac{1}{2}$ Crank--Nicolson scheme for internal energy in subproblem ($\mathrm{P}$).}
\label{tab:accuravy_test_time}
\end{table}
\begin{table}[ht!]
\centering
\begin{tabularx}{\linewidth}{@{~~}c@{~~}|c@{~}C@{~~}|c@{~}C@{~}|c@{~~}|c@{~}C@{~}|c@{~~}}
\toprule
$k$ & $\Delta x$ & $\|\vec{U}_h^{N_T}-\vec{U}(T)\|_{L^2_h}$ & $\Delta x$ & $\|\vec{U}_h^{N_T}-\vec{U}(T)\|_{L^2_h}$ & rate & $\Delta x$ & $\|\vec{U}_h^{N_T}-\vec{U}(T)\|_{L^2_h}$ & rate \\
\midrule
1 & $1/2^4$ & $1.209\cdot10^{-1}$ & $1/2^5$ & $3.071\cdot10^{-2}$ & 1.977 & $1/2^6$ & $7.728\cdot10^{-3}$ & 1.991 \\
\midrule
2 & $1/2^4$ & $5.116\cdot10^{-2}$ & $1/2^5$ & $1.413\cdot10^{-2}$ & 1.856 & $1/2^6$ & $3.718\cdot10^{-3}$ & 1.926 \\
\midrule
3 & $1/2^3$ & $4.945\cdot10^{-3}$ & $1/2^4$ & $2.974\cdot10^{-4}$ & 4.056 & $1/2^5$ & $1.813\cdot10^{-5}$ & 4.036 \\
\midrule
4 & $1/2^3$ & $3.221\cdot10^{-4}$ & $1/2^4$ & $1.677\cdot10^{-5}$ & 4.264 & $1/2^5$ & $1.012\cdot10^{-6}$ & 4.051 \\
\midrule
5 & $1/2^2$ & $7.374\cdot10^{-4}$ & $1/2^3$ & $1.387\cdot10^{-5}$ & 5.733 & $1/2^4$ & $2.087\cdot10^{-7}$ & 6.054 \\
\bottomrule
\end{tabularx}
\caption{Test of accuracy. The spatial error and convergence rates. From top to bottom: the $\IQ^1, \IQ^2, \cdots, \IQ^5$ schemes using a very small time step for a smooth solution.}
\label{tab:accuravy_test_space}
\end{table}

\subsection{Convergence study for testing of preserving positivity}
\label{sec:test2}
In this part, we verify our numerical algorithm preserves positivity. Let the computational domain $\Omega=[0,1]^2$ and the end time $T = 0.1024$. The prescribed manufactured solutions are as follows:
\begin{align*}
\rho = 1, \quad
\vec{u} = 
\begin{bmatrix}
0\\
0
\end{bmatrix}, \quad
e = \frac{1}{\gamma-1}(\sin^8{(2\pi(x+y))} + 10^{-12}).
\end{align*} 
Taking Reynolds number $\Rey=1$ and Prandtl number $\Pr = 1.4$, namely with $\gamma = 1.4$ we have $\lambda = 1$. 
The boundary conditions and the system right-hand side are defined by the prescribed solutions. We utilize the same $L^2_h$ norm to measure error.
\par
We use the second order Crank--Nicolson time discretization for internal energy in parabolic sub-problem. Fix the time step size $\Delta t = 3.125\times10^{-6}$ small enough such that the spatial error dominates.
We choose NIPG method with $\sigma = 2$ on $\Gammah$ and $\sigma = 4$ on $\partial{\Omega}$ for $\IQ^1$ scheme; and $\sigma = 0$ for $\IQ^k$ ($k \geq 2$) scheme to solve the second equation in subproblem ($\mathrm{P}$). We choose IIPG method with $\tilde\sigma = 2^k$ to solve the third equation in subproblem ($\mathrm{P}$).
We obtain the expected convergence rates, see Table~\ref{tab:positivity_test_space}.
\begin{table}[ht!]
\centering
\begin{tabularx}{\linewidth}{@{~~}c@{~~}|c@{~}C@{~~}|c@{~}C@{~}|c@{~~}|c@{~}C@{~}|c@{~~}|c@{~}}
\toprule
$k$ & $\Delta x$ & $\|\vec{U}_h^{N_T}-\vec{U}(T)\|_{L^2_h}$ & $\Delta x$ & $\|\vec{U}_h^{N_T}-\vec{U}(T)\|_{L^2_h}$ & rate & $\Delta x$ & $\|\vec{U}_h^{N_T}-\vec{U}(T)\|_{L^2_h}$ & rate & Postprocessing \\
\midrule
1 & $1/2^5$ & $2.858\cdot10^{-2}$ & $1/2^6$ & $6.804\cdot10^{-3}$ & 2.071 & $1/2^7$ & $1.692\cdot10^{-3}$ & 2.008 & Yes\\
\midrule
2 & $1/2^5$ & $6.301\cdot10^{-3}$ & $1/2^6$ & $1.518\cdot10^{-3}$ & 2.054 & $1/2^7$ & $3.749\cdot10^{-4}$ & 2.018 & Yes\\
\midrule
3 & $1/2^4$ & $2.018\cdot10^{-2}$ & $1/2^5$ & $2.063\cdot10^{-4}$ & 6.612 & $1/2^6$ & $9.680\cdot10^{-6}$ & 4.414 & No\\
\midrule
4 & $1/2^4$ & $2.320\cdot10^{-4}$ & $1/2^5$ & $1.121\cdot10^{-5}$ & 4.372 & $1/2^6$ & $6.245\cdot10^{-7}$ & 4.166 & Yes\\
\midrule
5 & $1/2^3$ & $4.614\cdot10^{-2}$ & $1/2^4$ & $5.697\cdot10^{-4}$ & 6.340 & $1/2^5$ & $7.187\cdot10^{-7}$ & 9.631 & No\\
\bottomrule
\end{tabularx}
\caption{Test of accuracy. The spatial error and convergence rates. From top to bottom: the $\IQ^1, \IQ^2, \cdots, \IQ^5$ schemes using a very small time step for a smooth solution. In last column, ``Yes'' indicates the postprocesing \eqref{postprocessing} is triggered, otherwise ``No''.}
\label{tab:positivity_test_space}
\end{table}

\subsection{Lax shock tube problem}
We choose the computational domain $\Omega = [-5,5]\times[0,2]$ and set the simulation end time $T = 1.3$. 
We uniformly partition domain $\Omega$ by square cells with mesh resolution $\Delta x = 1/100$.
The initial conditions for density $\rho^0$, velocity $\vec{u}^0 = \transpose{[u_x^0, u_y^0]}$, and pressure $p^0$ are prescribed as follows:
\begin{align*}
\transpose{[\rho^0, u_x^0, u_y^0, p^0]} = 
\begin{cases}
\transpose{[0.445,\, 0.698,\, 0,\, 3.528]} & \text{if}~~x\in[-5,0),\\
\transpose{[0.5,\, 0,\, 0,\, 0.571]} & \text{if}~~x\in[0,5].
\end{cases}
\end{align*}
The top and bottom boundaries are set to be reflective when solving subproblem ($\mathrm{H}$) and to be Neumann-type when solving subproblem ($\mathrm{P}$). Dirichlet boundary conditions are applied to the left and right boundaries for both subproblems ($\mathrm{H}$) and ($\mathrm{P}$), with values equal to the initials before the wave reaches the boundary.
The Figure~\ref{fig:Lax_shock_tube} shows snapshots of the density field at the simulation final time $T=1.3$ in mountain view.
\begin{figure}[ht!]
\begin{center}
\begin{tabularx}{\linewidth}{@{}c@{~}c@{~}c@{~}c@{}}
\includegraphics[width=0.31\textwidth]{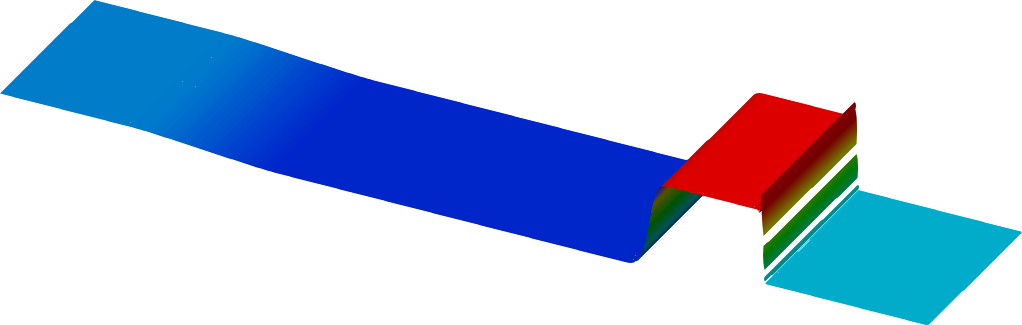}&
\includegraphics[width=0.31\textwidth]{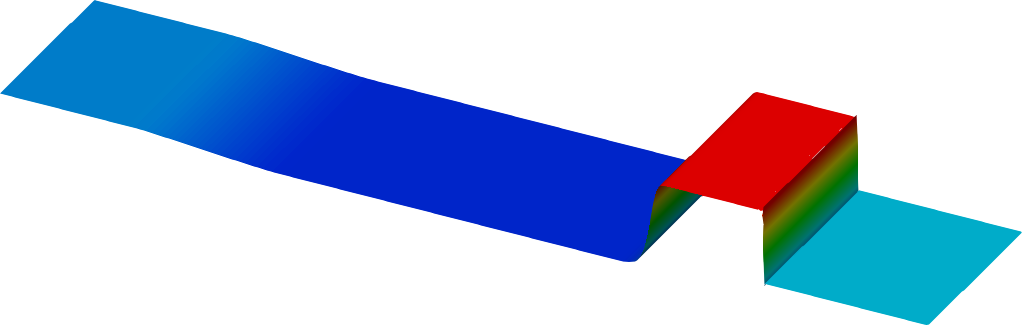}&
\includegraphics[width=0.31\textwidth]{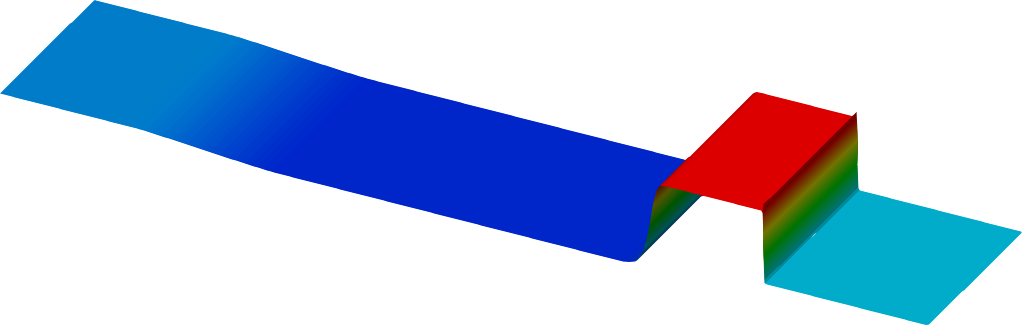}&
\includegraphics[width=0.04\textwidth]{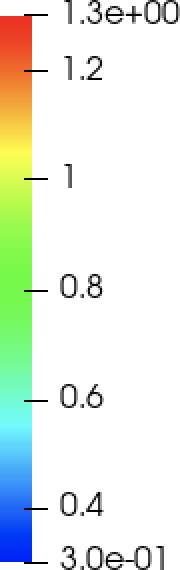} \\
$\IQ^1$ scheme & 
$\IQ^2$ scheme &
$\IQ^3$ scheme & ~ \\
\includegraphics[width=0.31\textwidth]{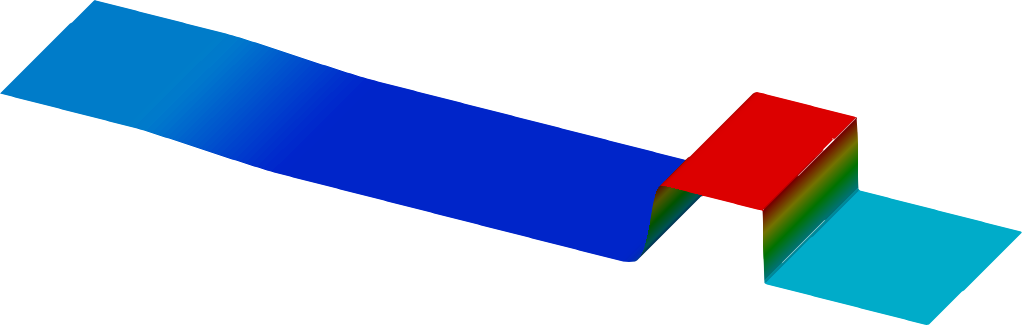}&
\includegraphics[width=0.31\textwidth]{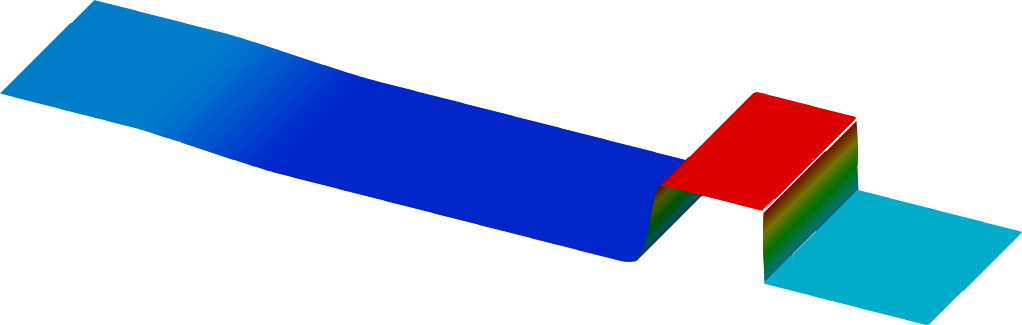}&
\includegraphics[width=0.31\textwidth]{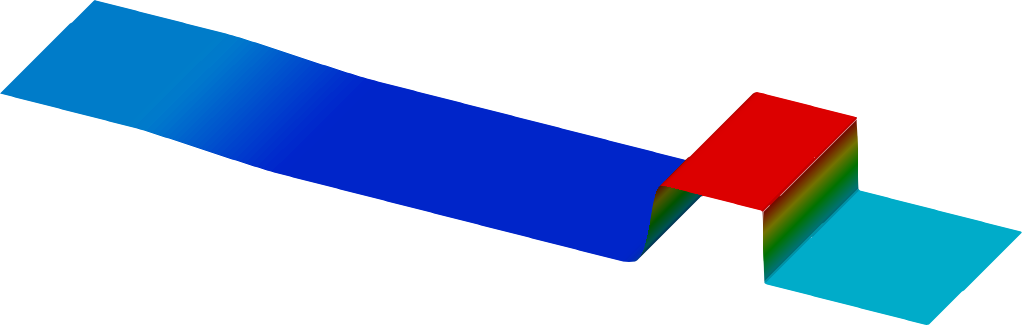}&
\includegraphics[width=0.04\textwidth]{Figures/lax_shock_tube/color_bar.png} \\
$\IQ^4$ scheme &
$\IQ^5$ scheme & 
$\IQ^6$ scheme & ~ \\
\end{tabularx}
\caption{Lax shock tube. The density field snapshots at time $T = 1.3$ are displayed in the mountain view.}
\label{fig:Lax_shock_tube}
\end{center}
\end{figure}

\subsection{Double rarefaction}
We choose the computational domain $\Omega = [-1,1]\times[0,1]$ and set the simulation end time $T = 0.6$. 
We uniformly partition domain $\Omega$ by square cells with mesh resolution $\Delta x = 1/640$ for $\IQ^1$ and $\IQ^2$ schemes, $\Delta x = 1/480$ for $\IQ^3$ and $\IQ^4$ schemes, and $\Delta x = 1/400$ for $\IQ^5$ and $\IQ^6$ schemes.
The initial conditions for density $\rho^0$, velocity $\vec{u}^0 = \transpose{[u_x^0, u_y^0]}$, and pressure $p^0$ are prescribed as follows:
\begin{align*}
\transpose{[\rho^0, u_x^0, u_y^0, p^0]} = 
\begin{cases}
\transpose{[7,\, -1,\, 0,\, 0.2]} & \text{if}~~x\in[-1,0),\\
\transpose{[7,\, 1,\, 0,\, 0.2]} & \text{if}~~x\in[0,1].
\end{cases}
\end{align*}
When solving subproblem ($\mathrm{H}$), reflective boundary conditions are set for the top and bottom boundaries, while outflow conditions are set for the left and right boundaries. When solving subproblem ($\mathrm{P}$), Neumann-type boundary conditions are applied to all boundaries.
The Figure~\ref{fig:doub_rarefaction} shows snapshots of density field at the simulation final time $T=0.6$ in mountain view.
\begin{figure}[ht!]
\begin{center}
\begin{tabularx}{0.95\linewidth}{@{}c@{~}c@{~}c@{~~}c@{}}
\includegraphics[width=0.29\textwidth]{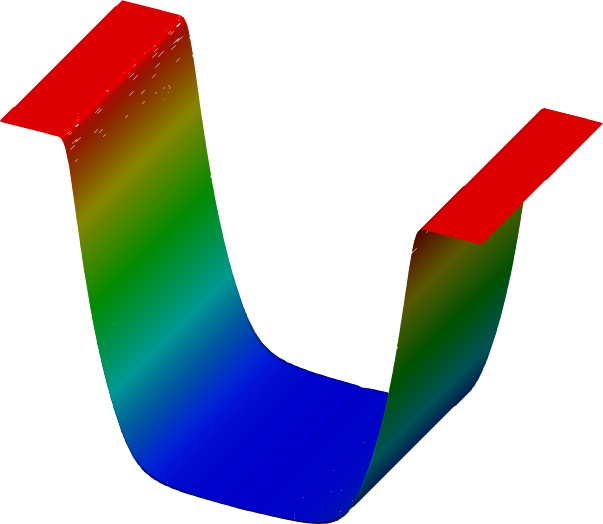}&
\includegraphics[width=0.29\textwidth]{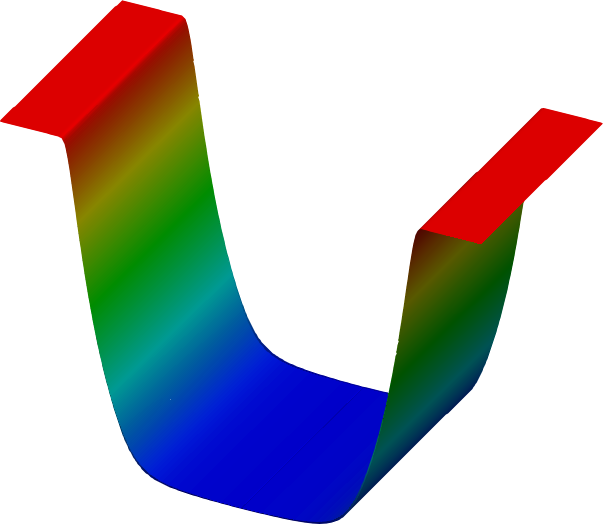}&
\includegraphics[width=0.29\textwidth]{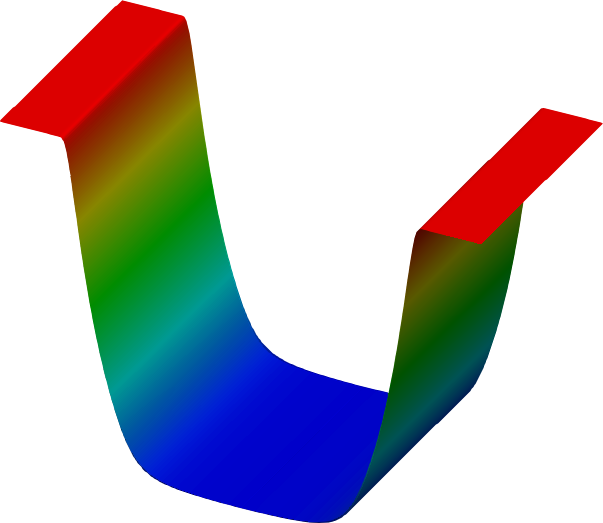}&
\includegraphics[width=0.04\textwidth]{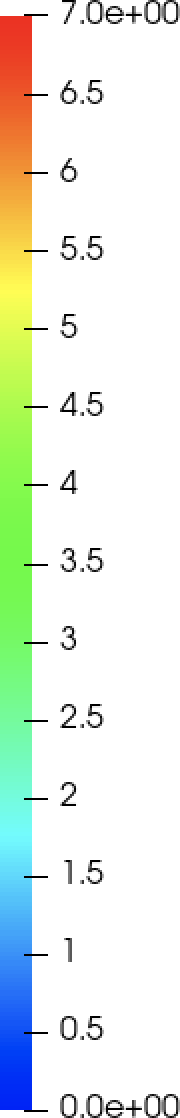} \\
$\IQ^1$ scheme & 
$\IQ^2$ scheme &
$\IQ^3$ scheme & ~ \\
\includegraphics[width=0.29\textwidth]{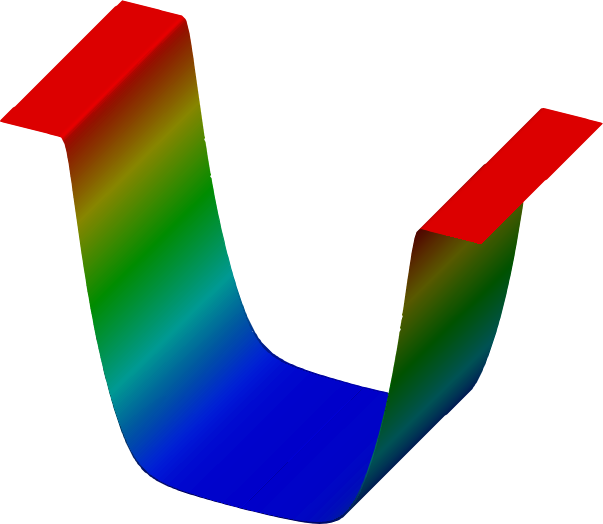}&
\includegraphics[width=0.29\textwidth]{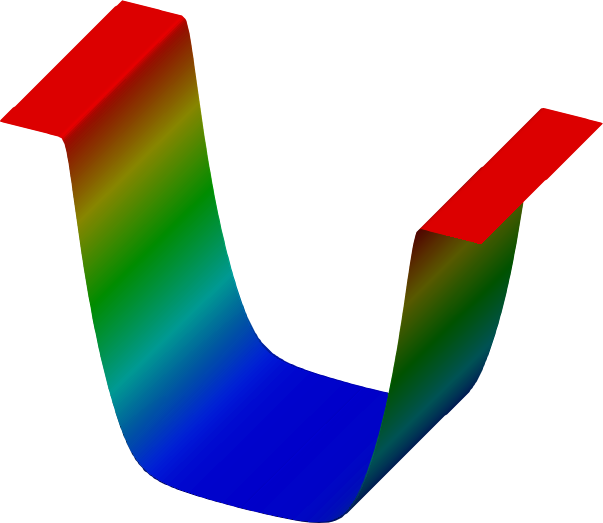}&
\includegraphics[width=0.29\textwidth]{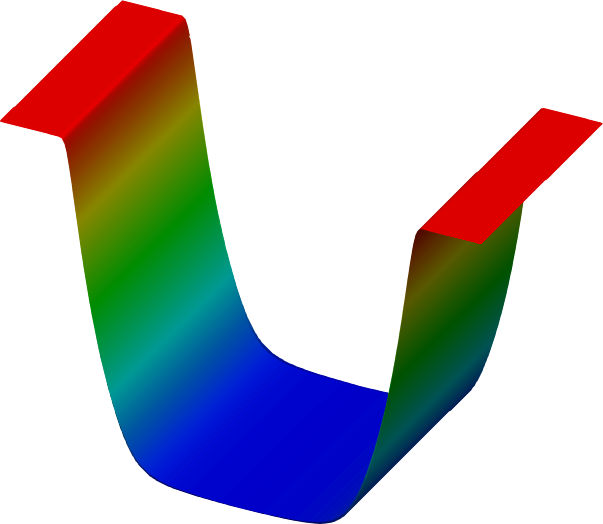}&
\includegraphics[width=0.04\textwidth]{Figures/doub_rarefaction/color_bar.png} \\
$\IQ^4$ scheme &
$\IQ^5$ scheme & 
$\IQ^6$ scheme & ~ \\
\end{tabularx}
\caption{Double rarefaction. The density field snapshots at time $T = 0.6$ are displayed in the mountain view.}
\label{fig:doub_rarefaction}
\end{center}
\end{figure}

\subsection{Sedov blast wave}\label{sec:numercal_experiment:Sedov}
The Sedov blast wave test is a standard benchmark in hyperbolic conservation law. It involves a blast wave generated by a strong explosion, which involves low density, low pressure, and a strong shock. This test holds great value in validating a positivity-preserving scheme.
\par
Let the computational domain $\Omega = [0, 1.1]^2$ and the simulation end time $T = 1$. 
We uniformly partition domain $\Omega$ by square cells with mesh resolution $\Delta x = 1.1/320$.
The initials are prescribed as piecewise constants: density $\rho^0 = 1$ and velocity $\vec{u}^0 = \vec{0}$, for all points in $\Omega$; the total energy $E^0$ equals to $10^{-12}$ everywhere except the cell at the lower left corner, where $0.244816/{\Delta x}^2$ is used.
When solving subproblem ($\mathrm{H}$), reflective boundary conditions are set for the left and bottom boundaries, while outflow conditions are set for the top and right boundaries. When solving subproblem ($\mathrm{P}$), Neumann-type boundary conditions are applied to all boundaries.
\par
The Figure~\ref{fig:sedov_blast_density} shows snapshots of density field at the simulation final time $T=1$. 
The postprocessing \eqref{postprocessing-local} with \eqref{definition-T} is used and necessary in all these tests. See Figure~\ref{fig:sedov_blast_DR}.
Our numerical algorithm preserves conservation and the shock location is correct.
\begin{figure}[ht!]
\begin{center}
\begin{tabularx}{0.95\linewidth}{@{}c@{~}c@{~}c@{~}c@{}}
\includegraphics[width=0.3\textwidth]{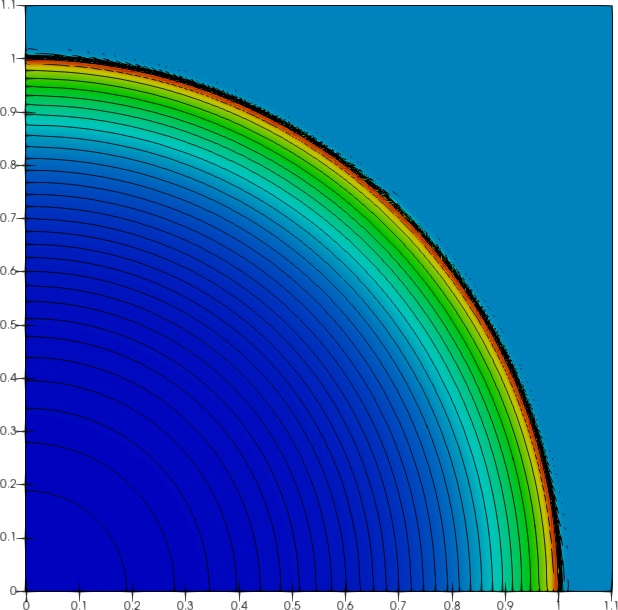} &
\includegraphics[width=0.3\textwidth]{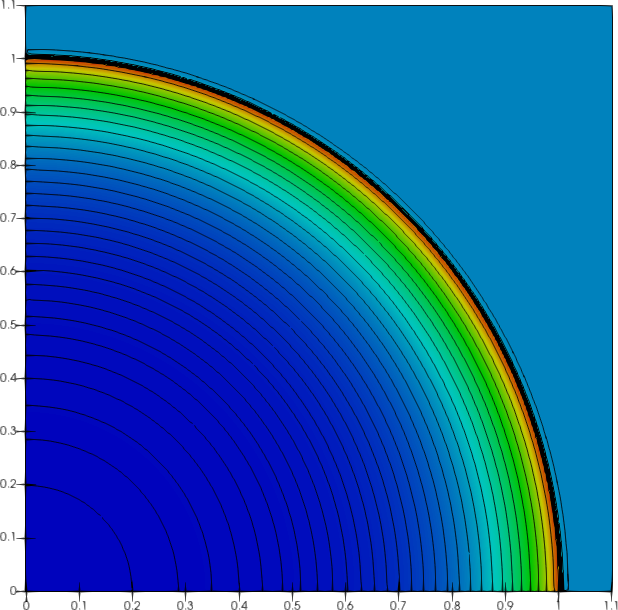} &
\includegraphics[width=0.3\textwidth]{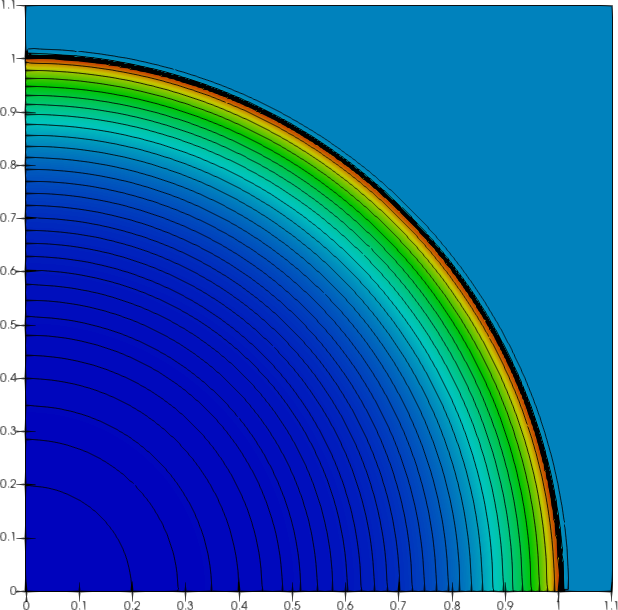} &
\includegraphics[width=0.045\textwidth]{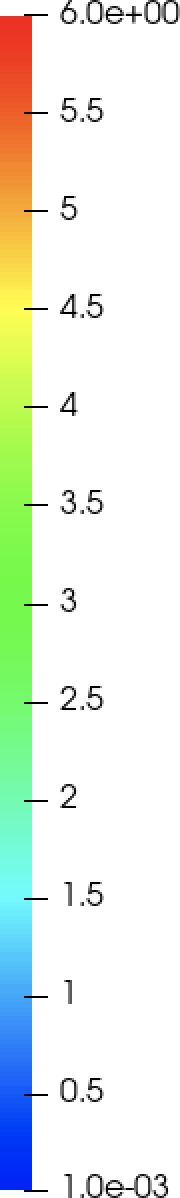}\\
$\IQ^1$ scheme & 
$\IQ^2$ scheme & 
$\IQ^3$ scheme & ~ \\
\includegraphics[width=0.3\textwidth]{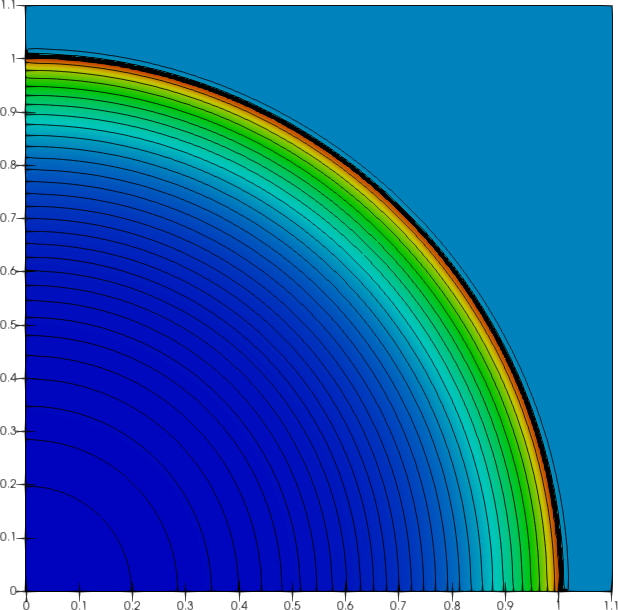} &
\includegraphics[width=0.3\textwidth]{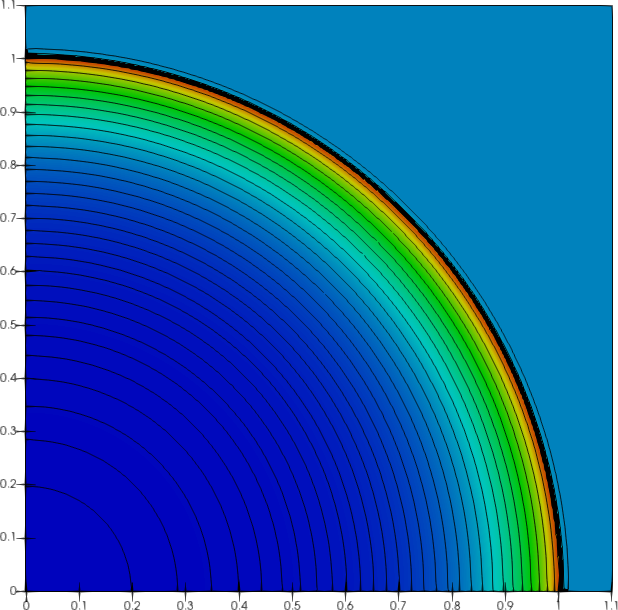} &
\includegraphics[width=0.3\textwidth]{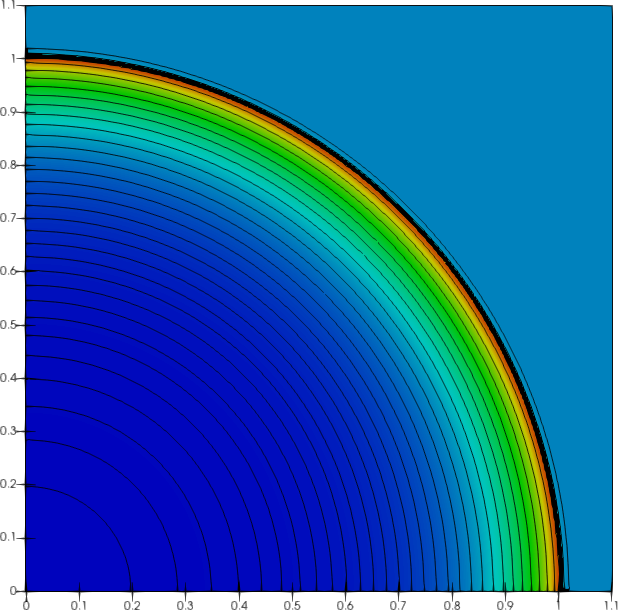} &
\includegraphics[width=0.045\textwidth]{{Figures/sedov/color_bar}.png}\\
$\IQ^4$ scheme &
$\IQ^5$ scheme & 
$\IQ^6$ scheme & ~ \\
\end{tabularx}
\caption{Sedov blast wave. The snapshots of density profile are taken at $T=1$. Plot of density: $50$ exponentially distributed contour lines of density from $0.001$ to $6$.}
\label{fig:sedov_blast_density}
\end{center}
\end{figure}
\begin{figure}[ht!]
\begin{center}
\begin{tabularx}{0.95\linewidth}{@{}c@{~}c@{~}c@{}}
\includegraphics[width=0.3\textwidth]{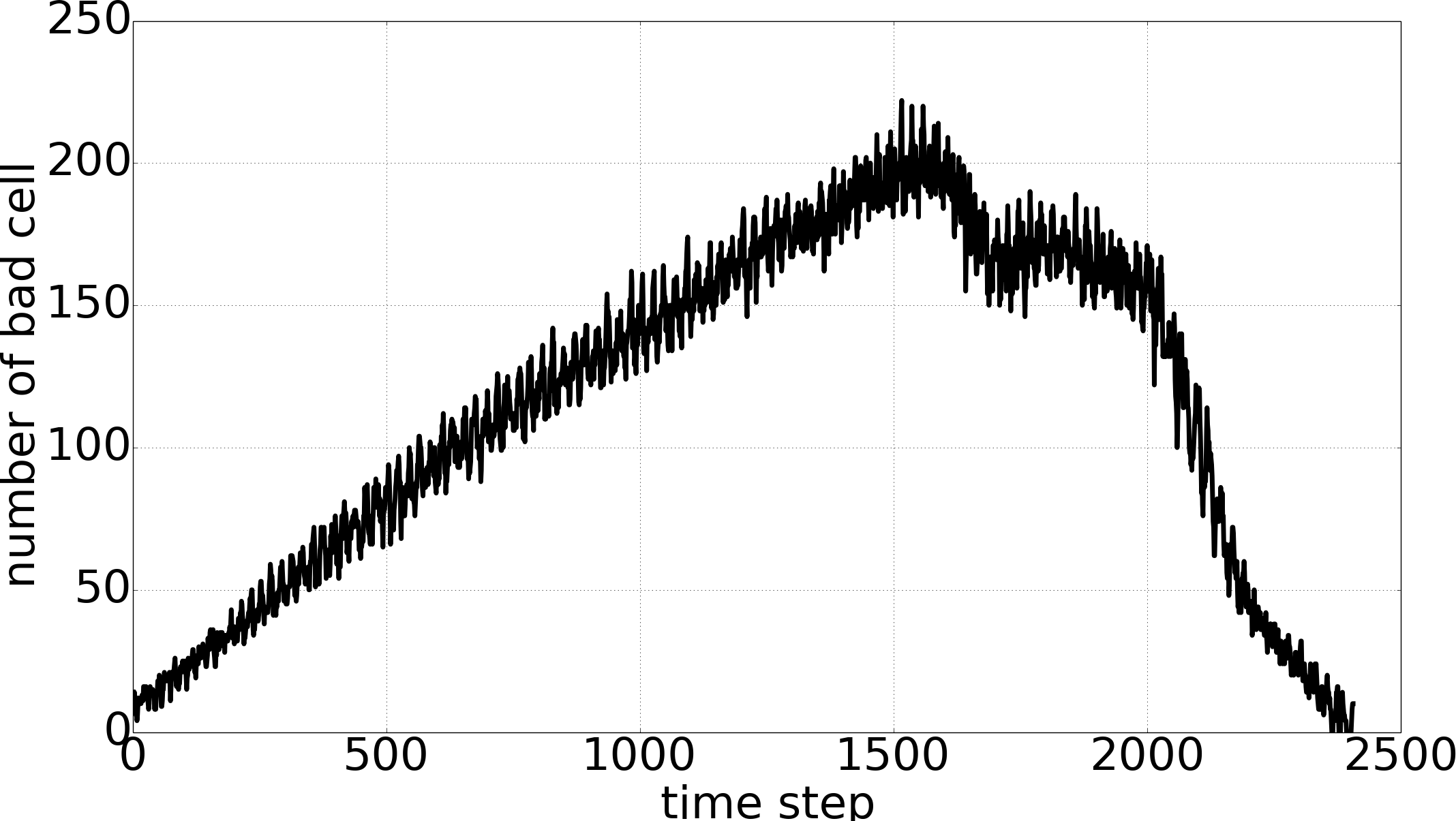} &
\includegraphics[width=0.3\textwidth]{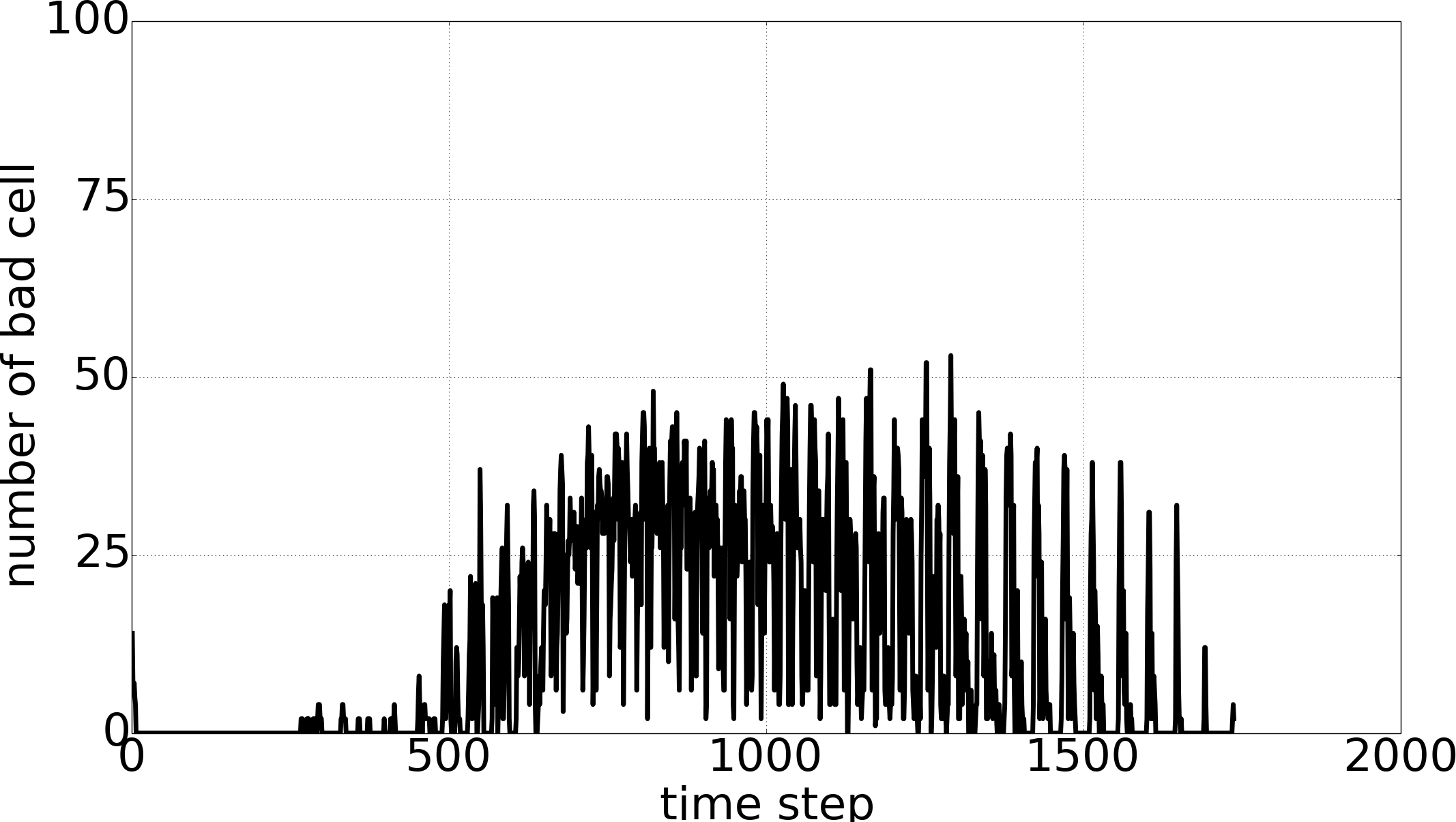} &
\includegraphics[width=0.3\textwidth]{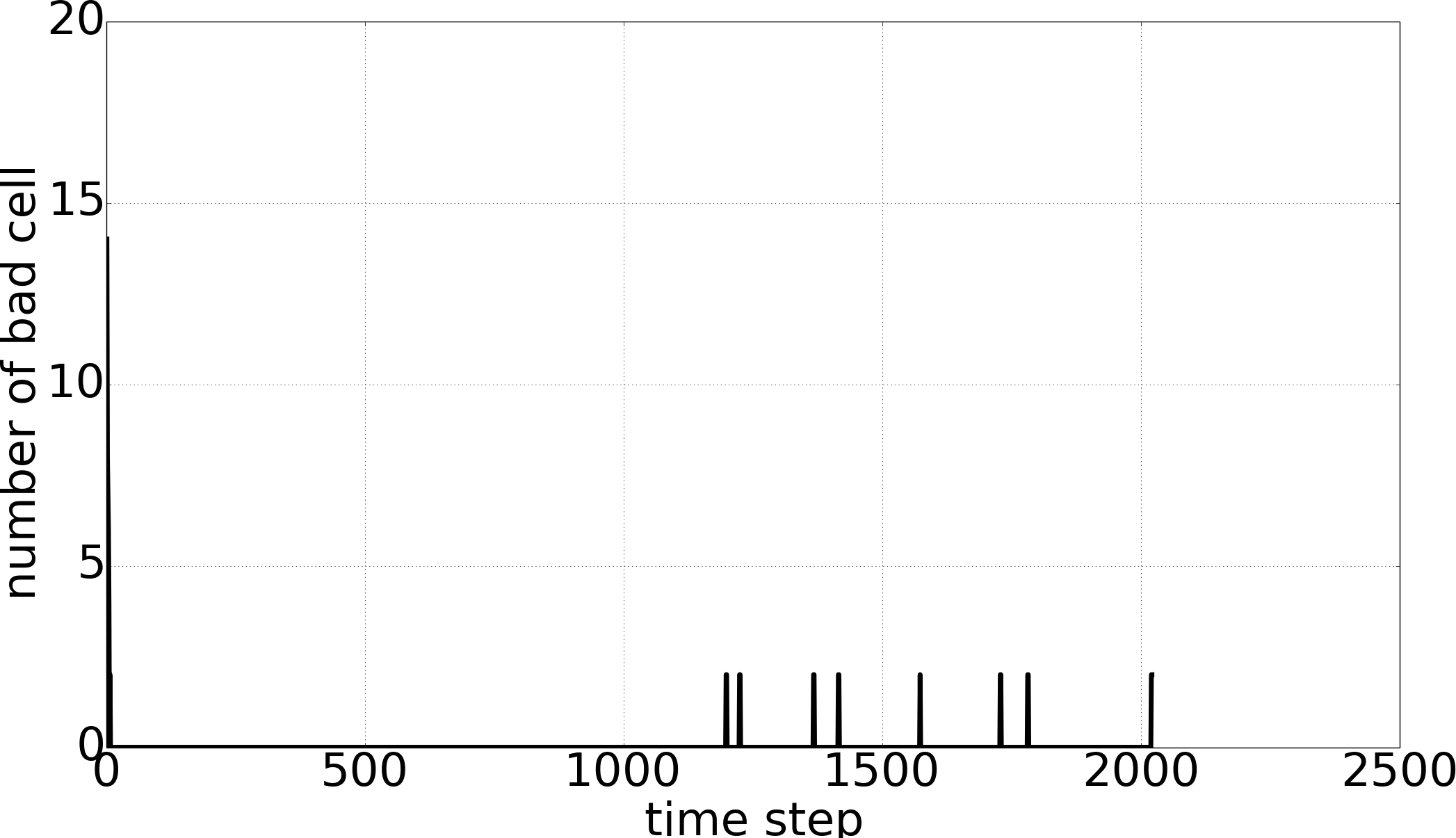} \\
\includegraphics[width=0.3\textwidth]{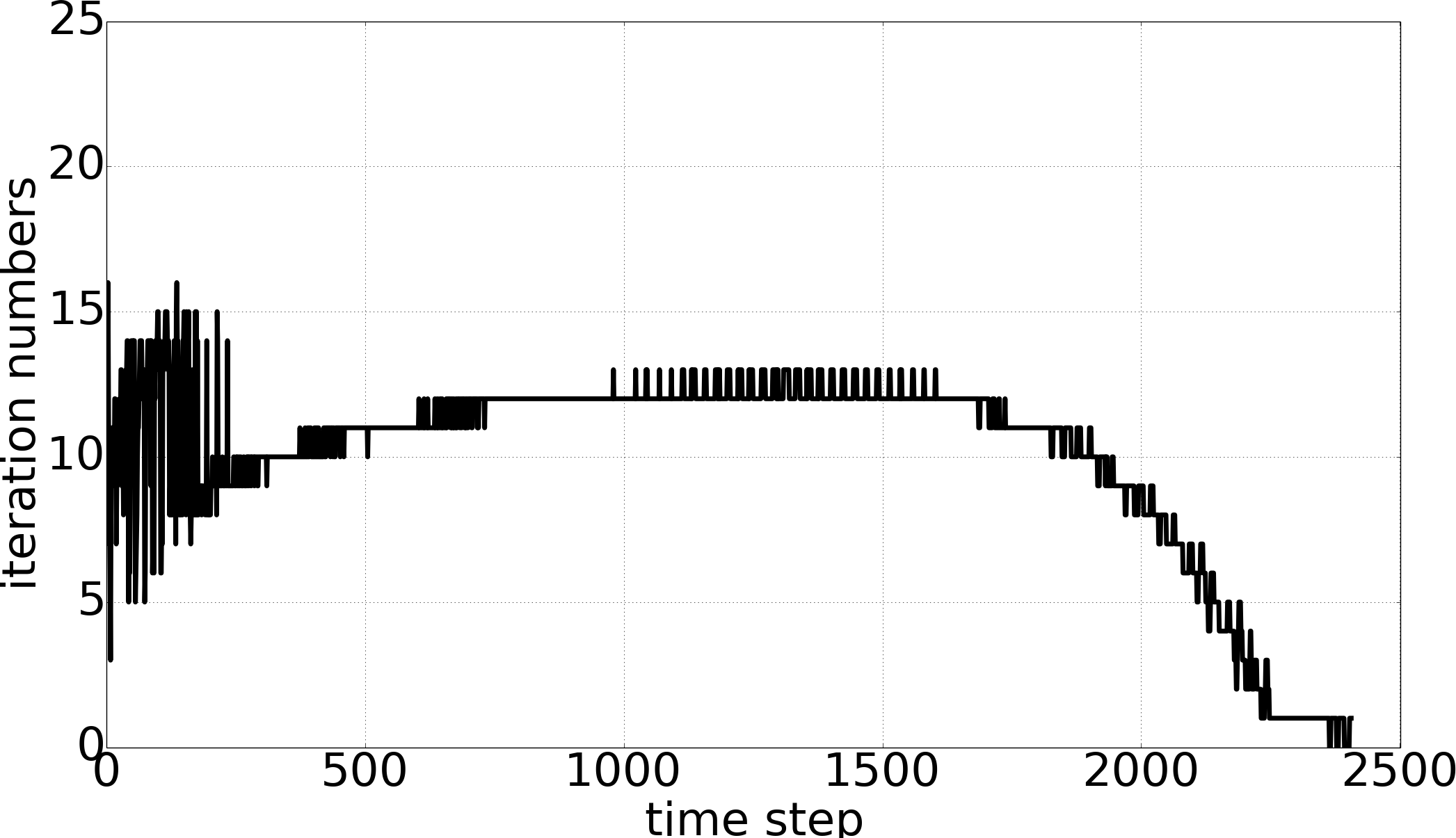} &
\includegraphics[width=0.3\textwidth]{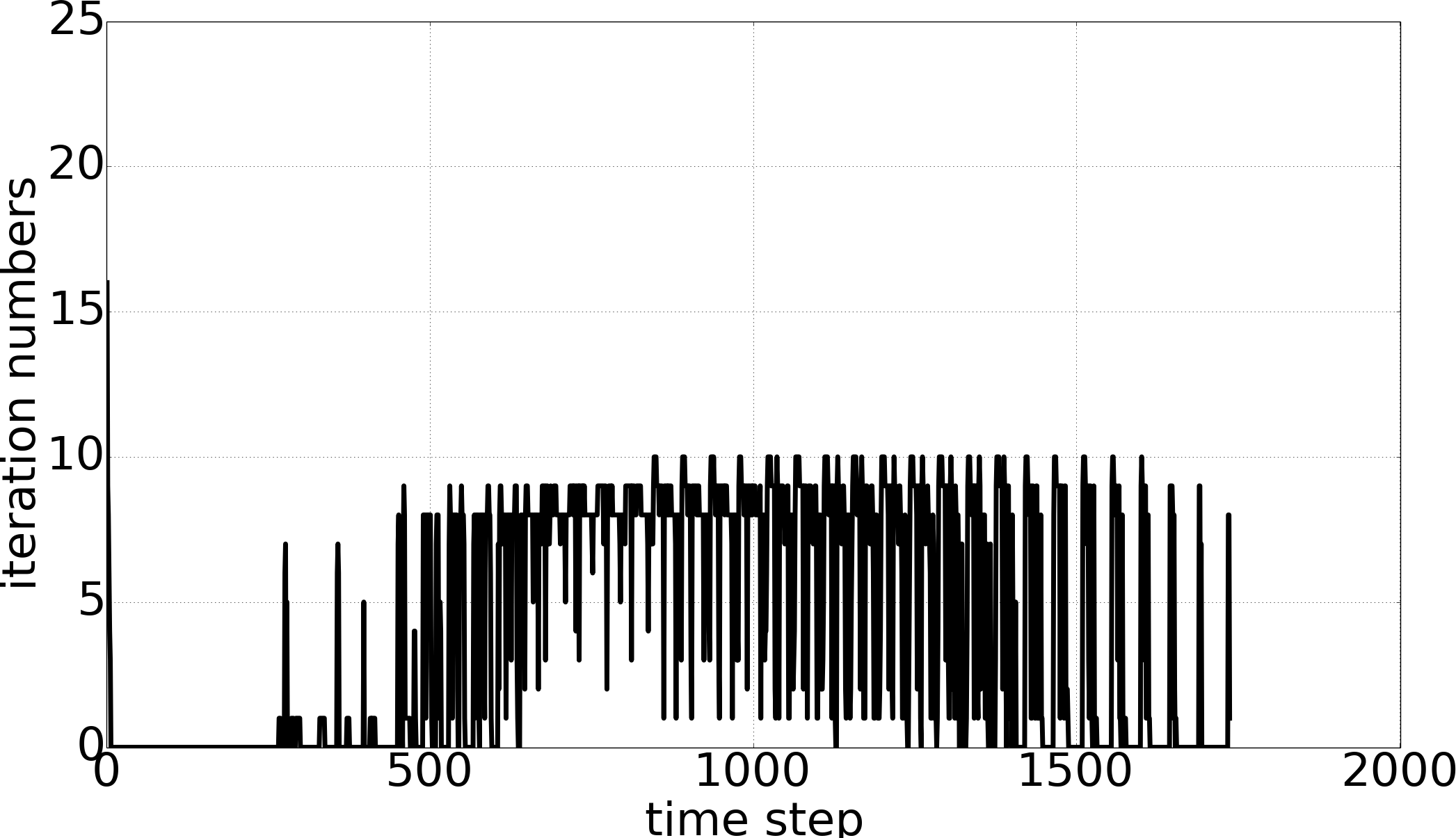} &
\includegraphics[width=0.3\textwidth]{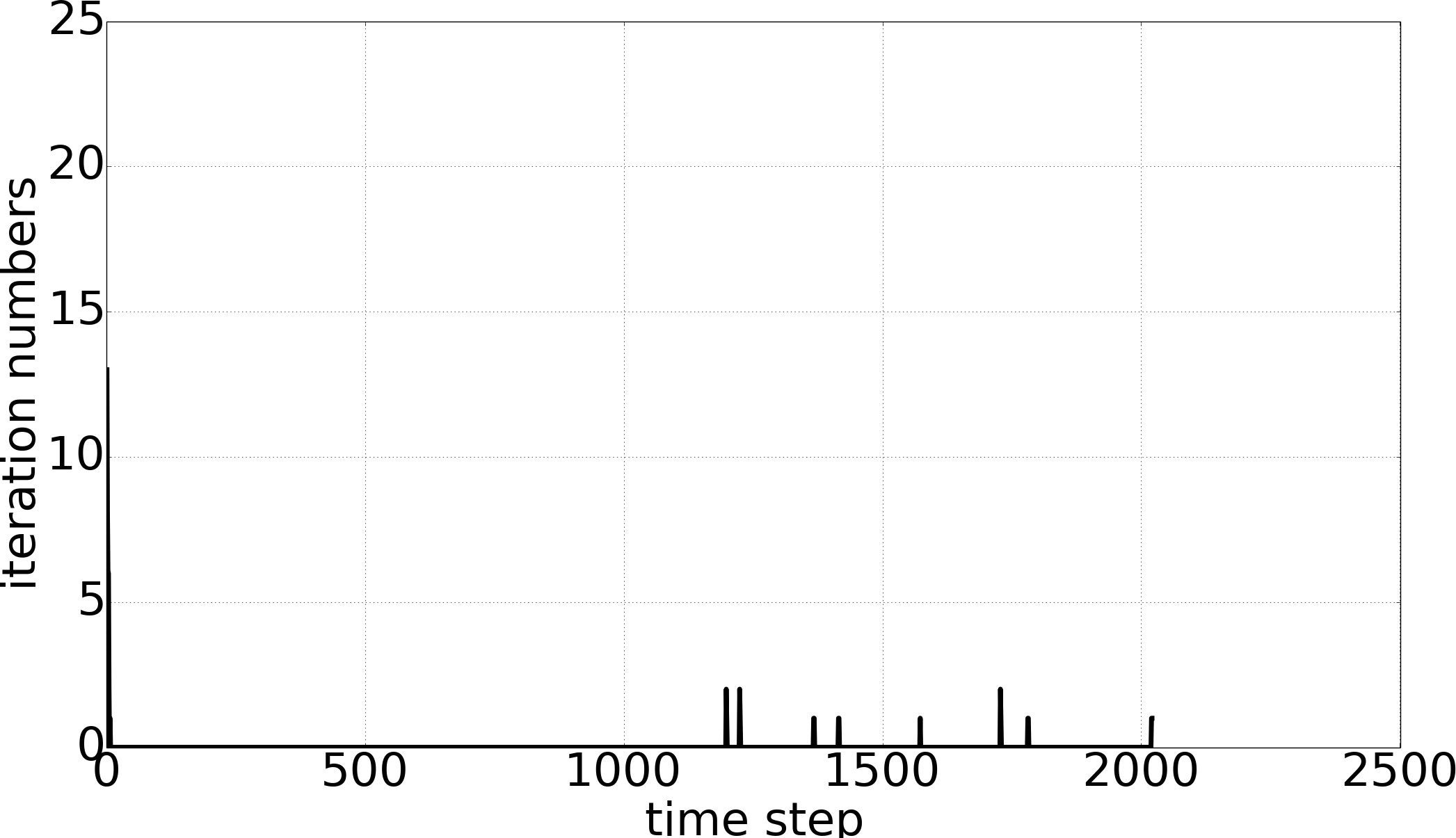} \\
\end{tabularx}
\caption{From left to right $\IQ^2$, $\IQ^4$, $\IQ^6$ DG schemes. Top: the number of bad cells after solving ($\mathrm{P}$) at each time step (the DG polynomial cell averages are not in the admissible set). Bottom: the number of Douglas--Rachford iterations need to reach round-off convergence for solving \eqref{total-energy-opt-2} with \eqref{definition-T}.}
\label{fig:sedov_blast_DR}
\end{center}
\end{figure}

\subsection{Shock diffraction}
In this test, we consider a right-moving high-speed shock, which is perpendicular to solid surface at initial and moves towards undisturbed air ahead. As the shock crosses the right corner, a region of low density and low pressure emerges, making this a challenging benchmark for conservation law.
\par
Let the computational domain $\Omega$ be the union of $[0,1]\times[6,11]$ and $[1,13]\times[0,11]$. We set the simulation end time $T = 2.3$.
The initial condition is a pure right-moving shock of Mach number $5.09$, initially located at $\{x=0.5, 6\leq y\leq 12\}$, moving into undisturbed air ahead of the shock with a density of $1.4$ and a pressure of $1$.
When solving subproblem ($\mathrm{H}$), the left boundary is inflow, while the right and bottom boundaries are outflow. The fluid--solid boundaries $\{y=6, 0\leq x\leq 1\}$ and $\{x=1, 0\leq y\leq 6\}$ are reflective. In addition, the flow values on the top boundary are set to accurately depict the motion of the Mach $5.09$ shock.
When solving subproblem ($\mathrm{P}$), Neumann-type boundary conditions are applied to the fluid--solid surfaces, while Dirichlet boundary conditions are applied to the remaining boundaries. The Dirichlet data on the left and top boundaries are determined by the inflow data and the exact motion of the Mach $5.09$ shock. Additionally, the Dirichlet data on the right and bottom boundaries remain unchanged from their initial values before the shock wave reaches the boundary.
\par
The Figure~\ref{fig:shock_def} displays snapshots of density field at the simulation final time $T=2.3$. The results are comparable to those in \cite{zhang2017positivity}.
\begin{figure}[ht!]
\begin{center}
\begin{tabularx}{\linewidth}{@{}c@{~}c@{~}c@{~}c@{}}
\includegraphics[width=0.31\textwidth]{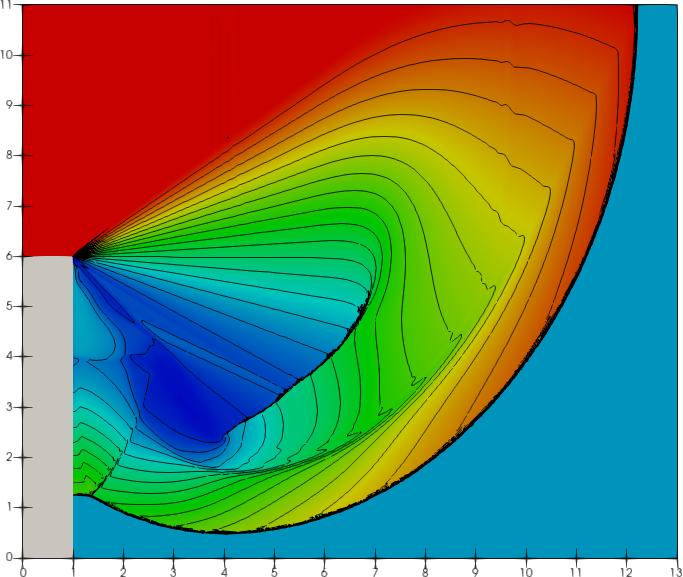}&
\includegraphics[width=0.31\textwidth]{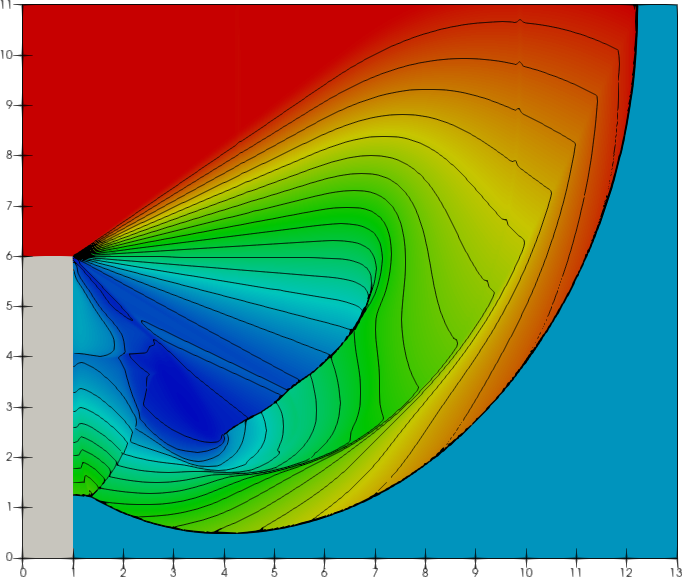}&
\includegraphics[width=0.31\textwidth]{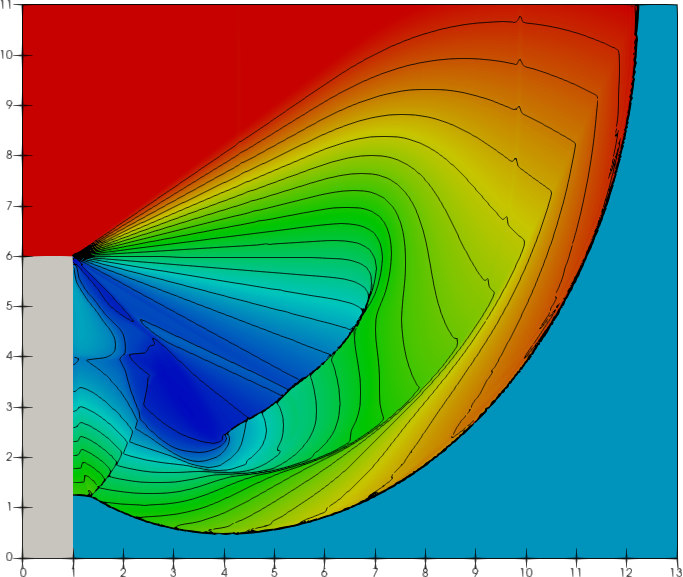}&
\includegraphics[width=0.0825\textwidth]{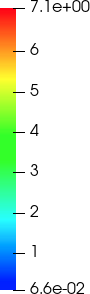} \\
$\IQ^1$ scheme ($\Delta{x}=1/96$) & 
$\IQ^2$ scheme ($\Delta{x}=1/96$) &
$\IQ^3$ scheme ($\Delta{x}=1/64$) & ~ \\
\includegraphics[width=0.31\textwidth]{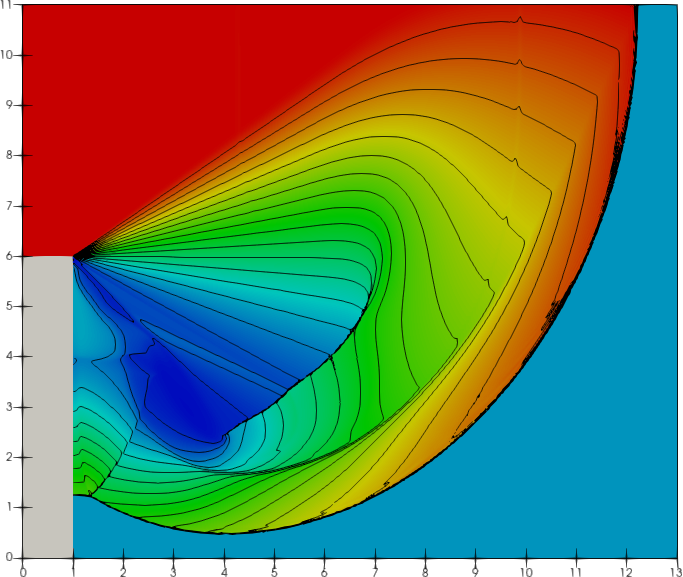}&
\includegraphics[width=0.31\textwidth]{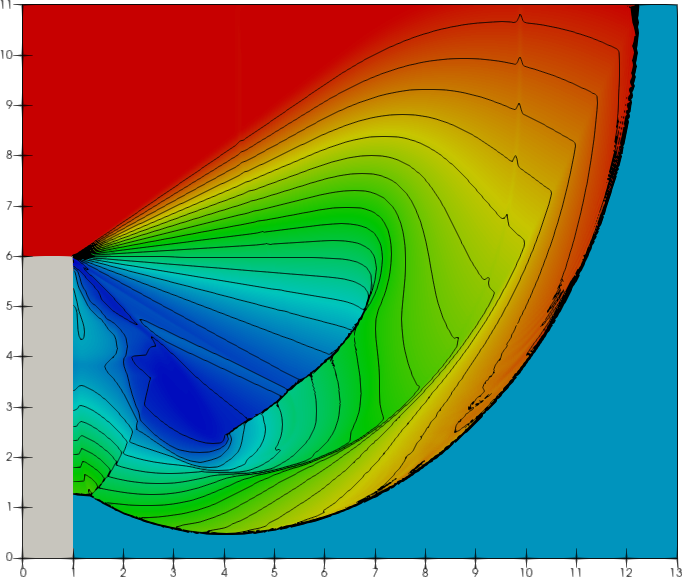}&
\includegraphics[width=0.31\textwidth]{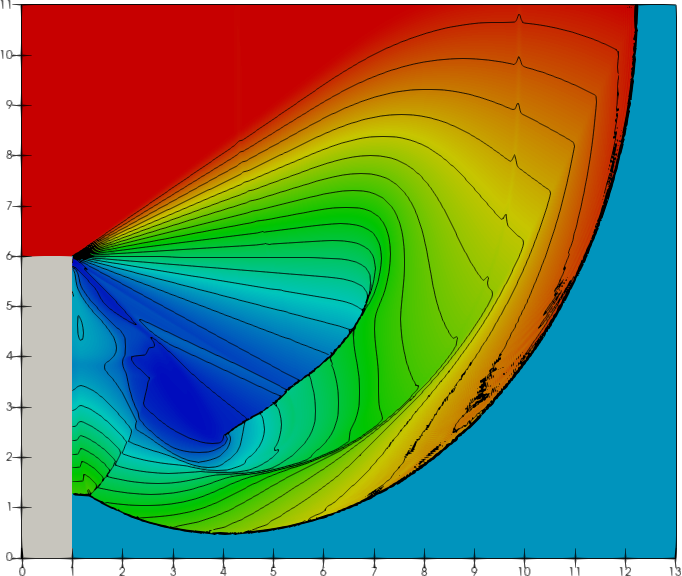}&
\includegraphics[width=0.0825\textwidth]{{Figures/shock_diffraction/color_bar}.png} \\
$\IQ^4$ scheme ($\Delta{x}=1/64$) &
$\IQ^5$ scheme ($\Delta{x}=1/48$) & 
$\IQ^6$ scheme ($\Delta{x}=1/48$) & ~ \\
\end{tabularx}
\caption{Shock diffraction. The snapshots of density profile are taken at $T = 2.3$. The gray colored region denotes solid. Plot of density: $20$ equally spaced contour lines from $0.066227$ to $7.0668$.}
\label{fig:shock_def}
\end{center}
\end{figure}

\subsection{Mach $10$ shock reflection and diffraction}
The high-speed shock reflection and diffraction test is a widely used benchmark \cite{fan2022positivity}. We consider a Mach 10 shock that moves to the right with a sixty-degree incident angle to the solid surface. As the shock across the sharp corner, areas of low density and low pressure appear. In the region of shock reflection, vortices are formed due to Kelvin--Helmholtz instabilities.
\par
Let the computational domain $\Omega$ be the union of $[0,4]\times[0,1]$ and $[1,4]\times[-1,0]$. We set the simulation end time $T = 0.2$.
The initial condition is a right-moving shock of Mach number $10$ positioned at $(\frac{1}{6},0)$ with a sixty-degree angle to the $x$-axis. The shock is moving into undisturbed air ahead of it, which has a density of $1.4$ and a pressure of $1$. In the post-shock region, the density is $8$, the velocity is $\transpose{[4.125\sqrt{3}, -4.125]}$, and the pressure is $116.5$.
\par
When solving subproblem ($\mathrm{H}$), the left boundary is inflow, while the right and bottom boundaries are outflow. Part of the fluid--solid boundaries $\{y=0, \frac{1}{6}\leq x\leq 1\}$ and $\{x=1, -1\leq y\leq 0\}$ are reflective, and the post-shock condition is imposed at $\{y=0, 0\leq x\leq \frac{1}{6}\}$.  On the boundary with post-shock condition, the density, velocity, and pressure are fixed in time with the initial values to make the reflected shock stick to the solid wall. In addition, the flow values on the top boundary are set to accurately depict the motion of the Mach $10$ shock.
When solving subproblem ($\mathrm{P}$), Neumann-type boundary conditions are applied to part of the fluid--solid surfaces associated with the reflective boundary in subproblem ($\mathrm{H}$), while Dirichlet boundary conditions are applied to the remaining boundaries. The Dirichlet data on the left and top boundaries are determined by the inflow data and the exact motion of the Mach $10$ shock. Additionally, the Dirichlet data on the right and bottom boundaries remain unchanged from their initial values before the shock wave reaches the boundary.
\par
From Figure~\ref{fig:shock_ref_def}, we see our scheme produces satisfactory non-oscillatory solutions with correct shock location and well-captured rollups. These test results are consistent with the observations for fully explicit high order accurate schemes in \cite{zhang2017positivity}.
\begin{figure}[ht!]
\begin{center}
\begin{tabularx}{\linewidth}{@{}c@{~}c@{~}c@{}}
\includegraphics[width=0.33\textwidth]{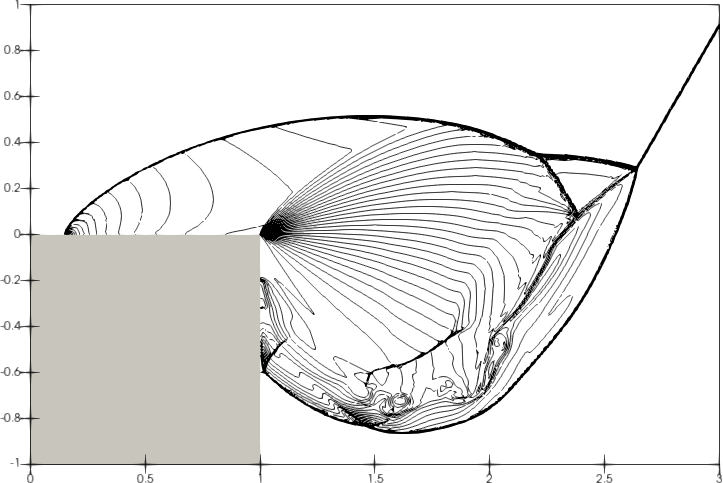}&
\includegraphics[width=0.33\textwidth]{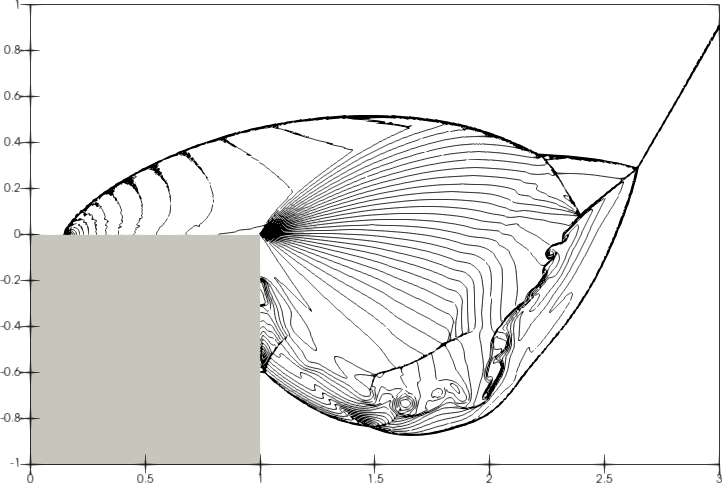}&
\includegraphics[width=0.33\textwidth]{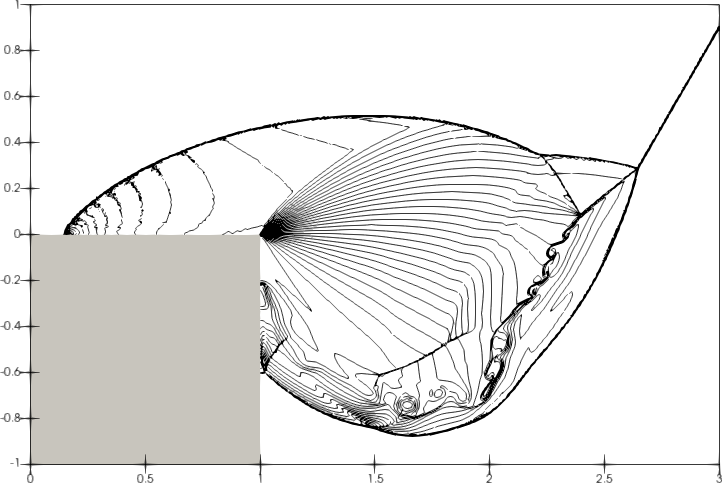}\\
$\IQ^1$ scheme ($\Delta{x}=1/600$) & 
$\IQ^2$ scheme ($\Delta{x}=1/400$) &
$\IQ^3$ scheme ($\Delta{x}=1/300$) \\
\includegraphics[width=0.33\textwidth]{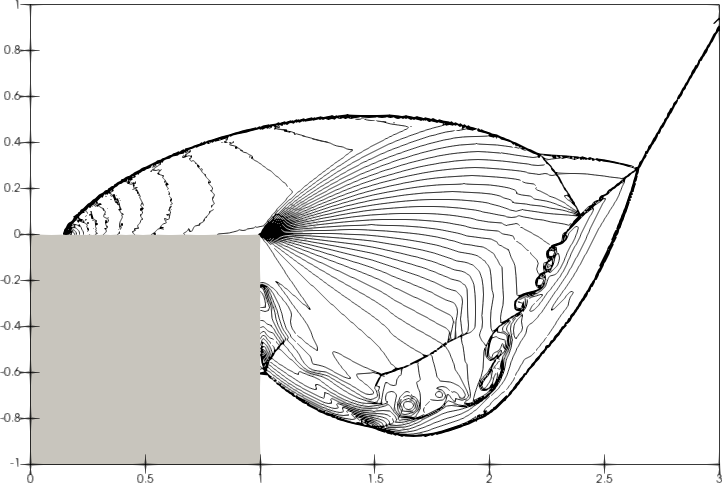}&
\includegraphics[width=0.33\textwidth]{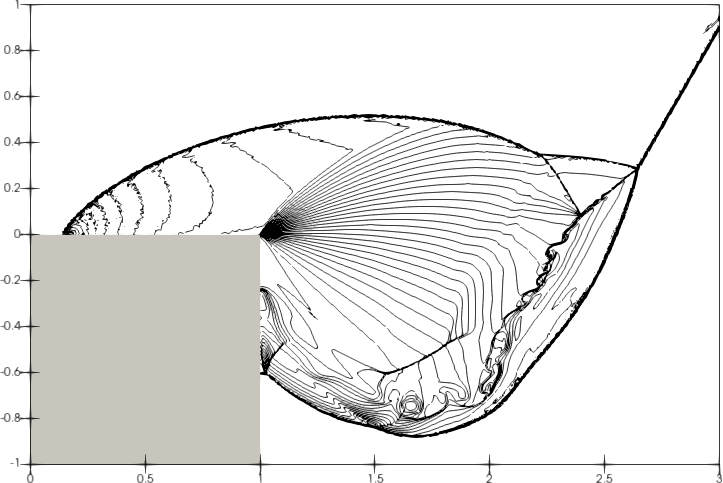}&
\includegraphics[width=0.33\textwidth]{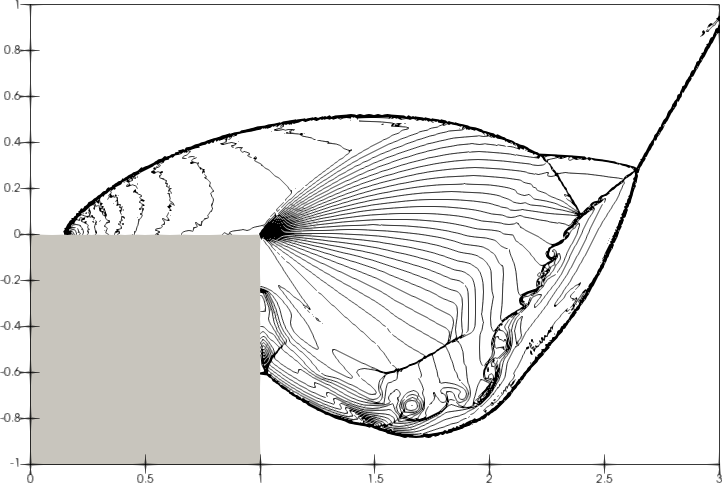}\\
$\IQ^4$ scheme ($\Delta{x}=1/240$) &
$\IQ^5$ scheme ($\Delta{x}=1/200$) & 
$\IQ^6$ scheme ($\Delta{x}=1/180$) \\
\end{tabularx}
\caption{Mach $10$ shock reflection and diffraction. The snapshots of density profile are taken at $T = 0.2$. The gray colored region denotes solid. Plot of density: $50$ equally space contour lines from $0$ to $25$. Only contour lines are plotted. We can observe that the scheme with higher order spatial accuracy indeed induces less artificial viscosity, despite that the temporal accuracy is at most second order.}
\label{fig:shock_ref_def}
\end{center}
\end{figure}

\subsection{High Mach number astrophysical jet}\label{sec:numercal_experiment:jet}
To replicate the gas flows and shock wave patterns observed in the Hubble Space Telescope images, one can utilize theoretical models within a gas dynamics simulator, see \cite{gardner2009numerical,ha2005numerical,tong2023class}. We consider the Mach $2000$ astrophysical jets without radiative cooling to demonstrate the robustness of our scheme.
\par
Let the computational domain $\Omega = [0,1]\times[-0.5,0.5]$. We set the simulation end time $T = 0.001$.
In this example, we use the ideal gas constant $\gamma = 5/3$. The initial density $\rho^0 = 0.5$, velocity $\vec{u}^0 = \vec{0}$, and pressure $p^0 = 10^{-6}$.
When solving subproblem ($\mathrm{H}$), the following inflow boundary conditions are set for the left boundary 
\begin{align*}
\transpose{[\rho, u_x, u_y, p]} = 
\begin{cases}
\transpose{[5,\, 800,\, 0,\, 0.4127]} & \text{if}~x=0~\text{and}~\abs{y}\leq 0.05,\\
\transpose{[0.5,\, 0,\, 0,\, 10^{-6}]} & \text{if}~x=0~\text{and}~\abs{y}> 0.05,
\end{cases}
\end{align*}
while the outflow boundary conditions are set for the top, right, and bottom boundaries. When solving subproblem ($\mathrm{P}$), Dirichlet boundary condition is applied to the left boundary, while Neumann-type boundary conditions are applied to the remaining boundaries.
The Dirichlet data on the left boundary are determined by the inflow data of the Mach $2000$ astrophysical jet.
\par
We take $\epsilon=10^{-8}$ in defining $G^\epsilon$ and the Zhang--Shu limiter in Section~\ref{sec:simplelimiter}.
The postprocessing of DG cell averages is necessary in these simulations. 
For the sake of robustness and efficiency in the postprocessing step,  we define the local region $T$ as the set of indices 
\begin{equation}
\label{definition-T-2}
T=\left\{i: \mathrm{either}\quad  \overline{\vec{U}_i^\mathrm{P}}\notin G^\epsilon \quad \mathrm{or}\quad \overline{E^\mathrm{P}_i}-\frac12\|\overline{\vec{m}^\mathrm{P}_i}\|/\overline{\rho^\mathrm{P}_i} \geq 2*10^{-6}\right \}.
\end{equation}

The Figure~\ref{fig:astrophysical_jet_density} shows snapshots of density field at the simulation final time $T=0.001$. 
See the performance of Douglas--Rachford splitting for solving \eqref{total-energy-opt-2} in Figure~\ref{fig:astrophysical_jet_DR}.
\begin{figure}[ht!]
\begin{center}
\begin{tabularx}{\linewidth}{@{}c@{~}c@{~}c@{~}c@{}}
\includegraphics[width=0.32\textwidth]{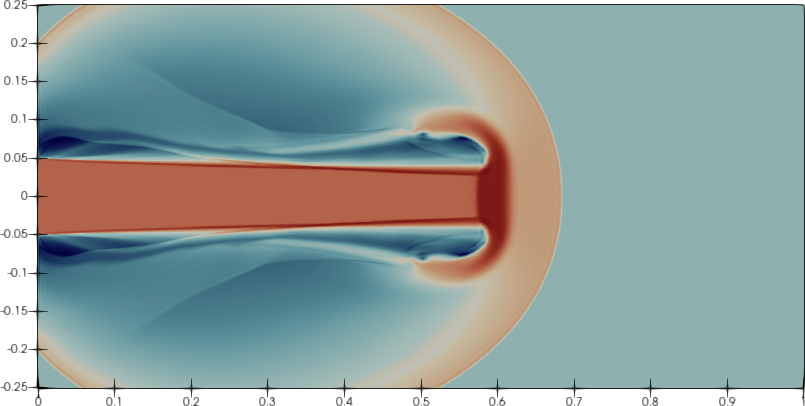} &
\includegraphics[width=0.32\textwidth]{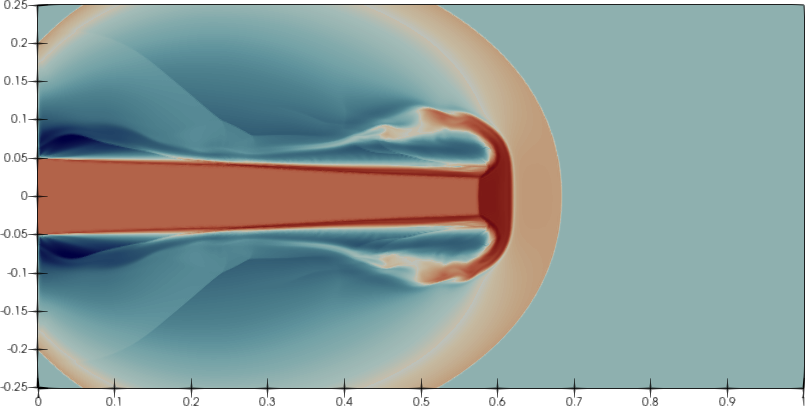} &
\includegraphics[width=0.32\textwidth]{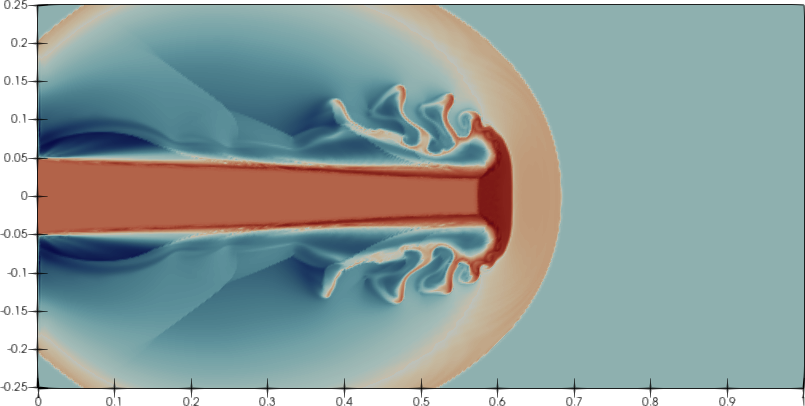} & 
\includegraphics[width=0.035\textwidth]{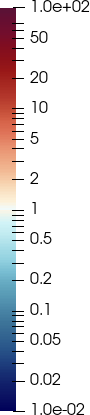} \\
$\IQ^1$ scheme ($\Delta{x}=1/640$) & 
$\IQ^2$ scheme ($\Delta{x}=1/640$) &
$\IQ^3$ scheme ($\Delta{x}=1/480$) \\
\includegraphics[width=0.32\textwidth]{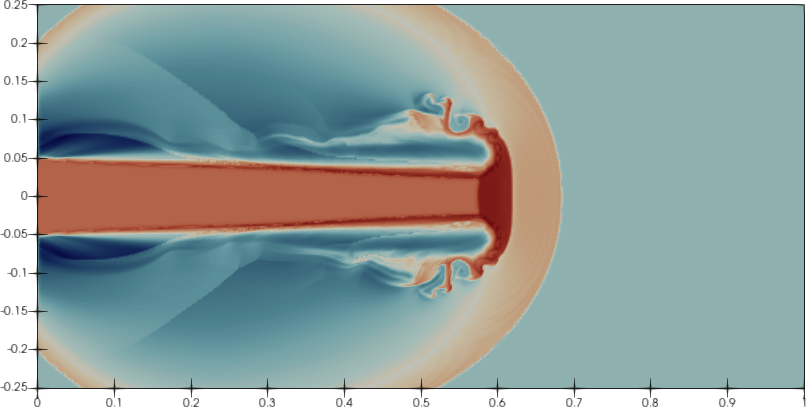} &
\includegraphics[width=0.32\textwidth]{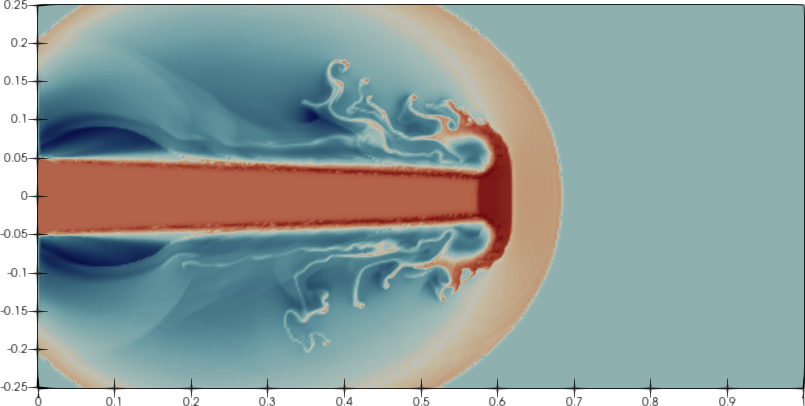} &
\includegraphics[width=0.32\textwidth]{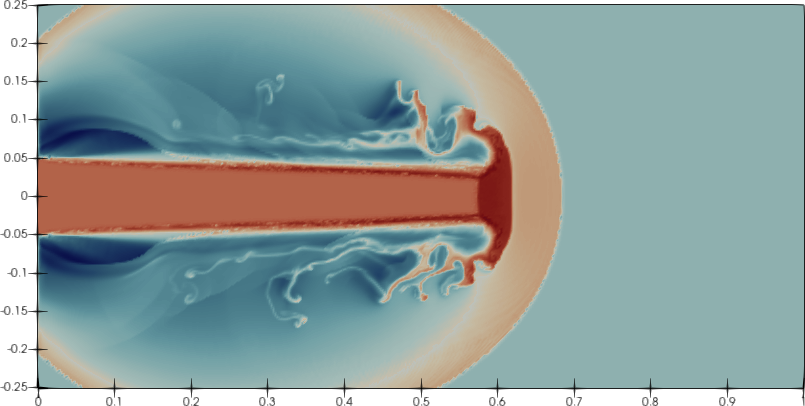} &
\includegraphics[width=0.035\textwidth]{Figures/astrophysical_jet/color_bar.png} \\
$\IQ^4$ scheme ($\Delta{x}=1/480$) &
$\IQ^5$ scheme ($\Delta{x}=1/400$) & 
$\IQ^6$ scheme ($\Delta{x}=1/400$) \\
\end{tabularx}
\caption{Astrophysical jets. The snapshots of the density filed at $T=0.001$. Scales are logarithmic. We can observe that the scheme with higher order spatial accuracy indeed induces less artificial viscosity, despite that the temporal accuracy is at most second order.}
\label{fig:astrophysical_jet_density}
\end{center}
\end{figure}
\begin{figure}[ht!]
\begin{center}
\begin{tabularx}{0.95\linewidth}{@{}c@{~}c@{~}c@{}}
\includegraphics[width=0.3\textwidth]{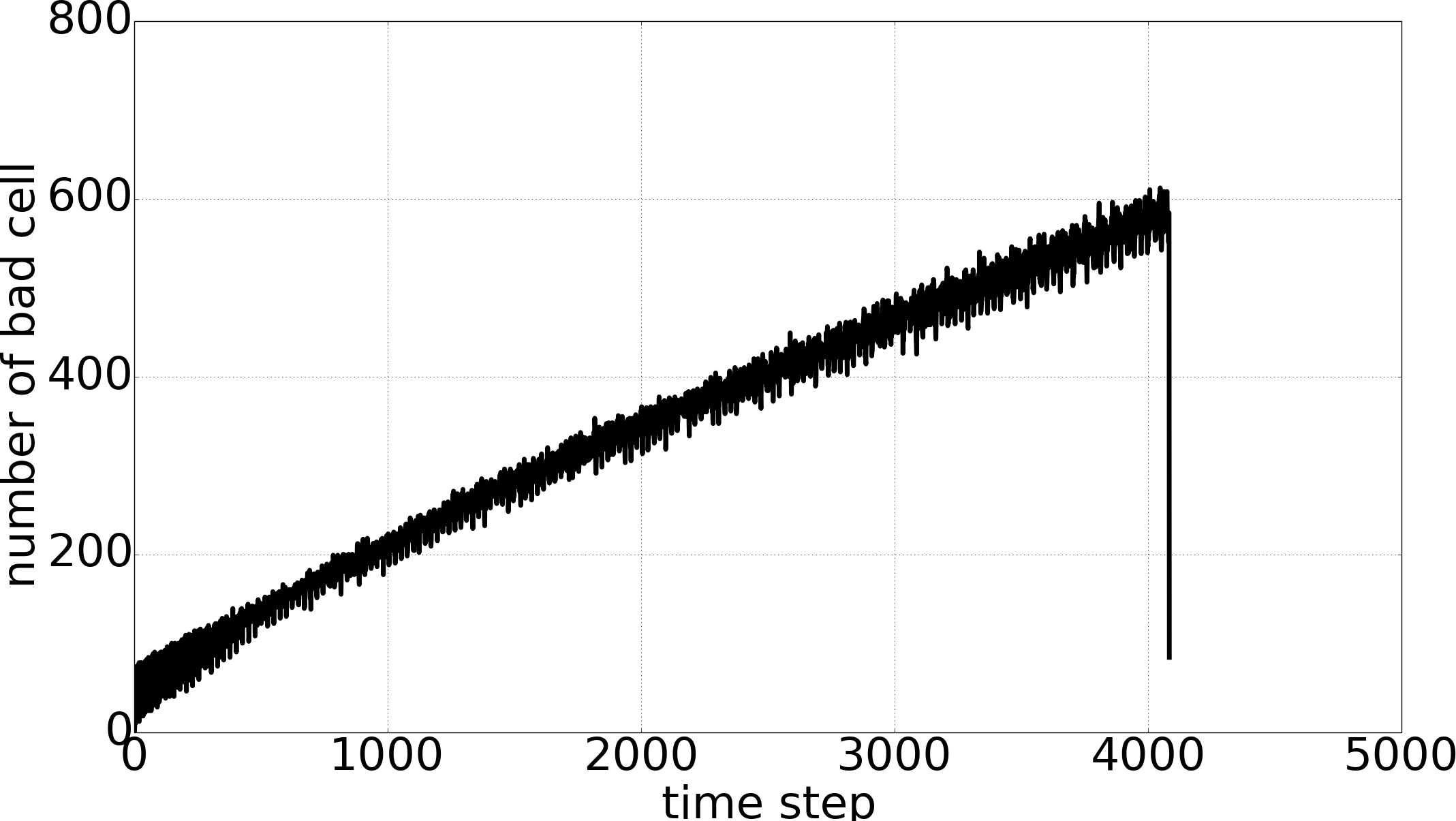} &
\includegraphics[width=0.3\textwidth]{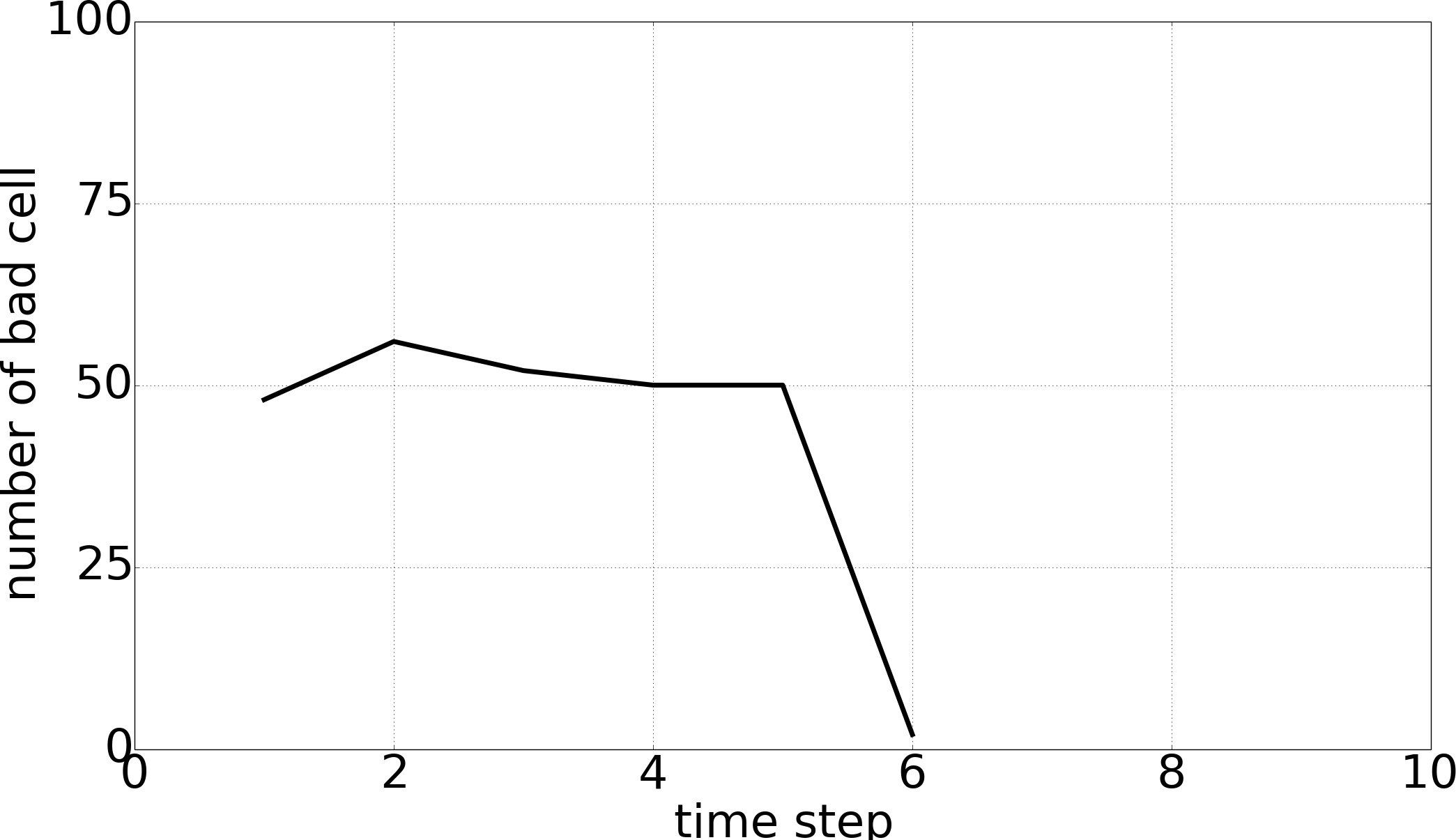} &
\includegraphics[width=0.3\textwidth]{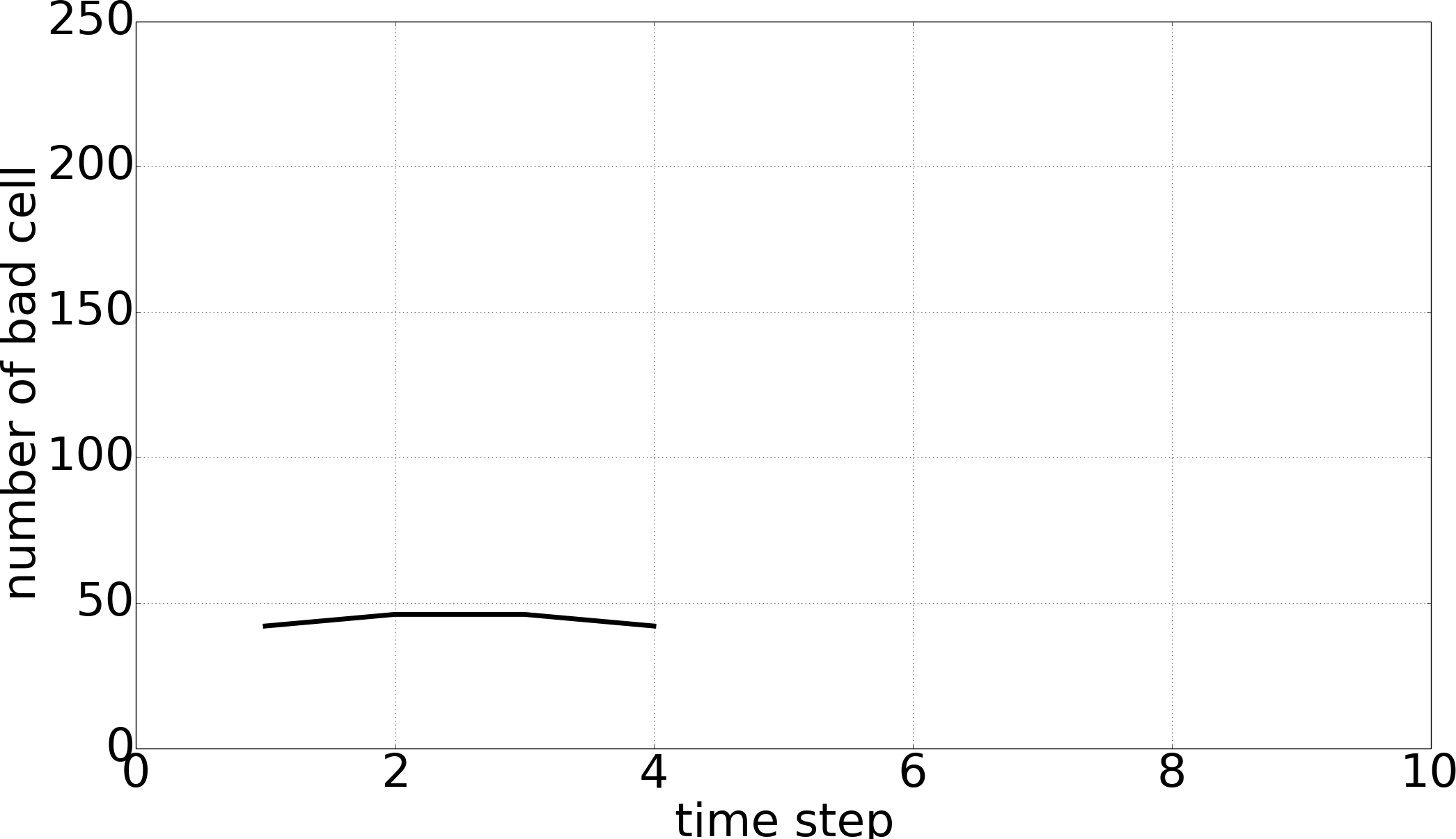} \\
\includegraphics[width=0.3\textwidth]{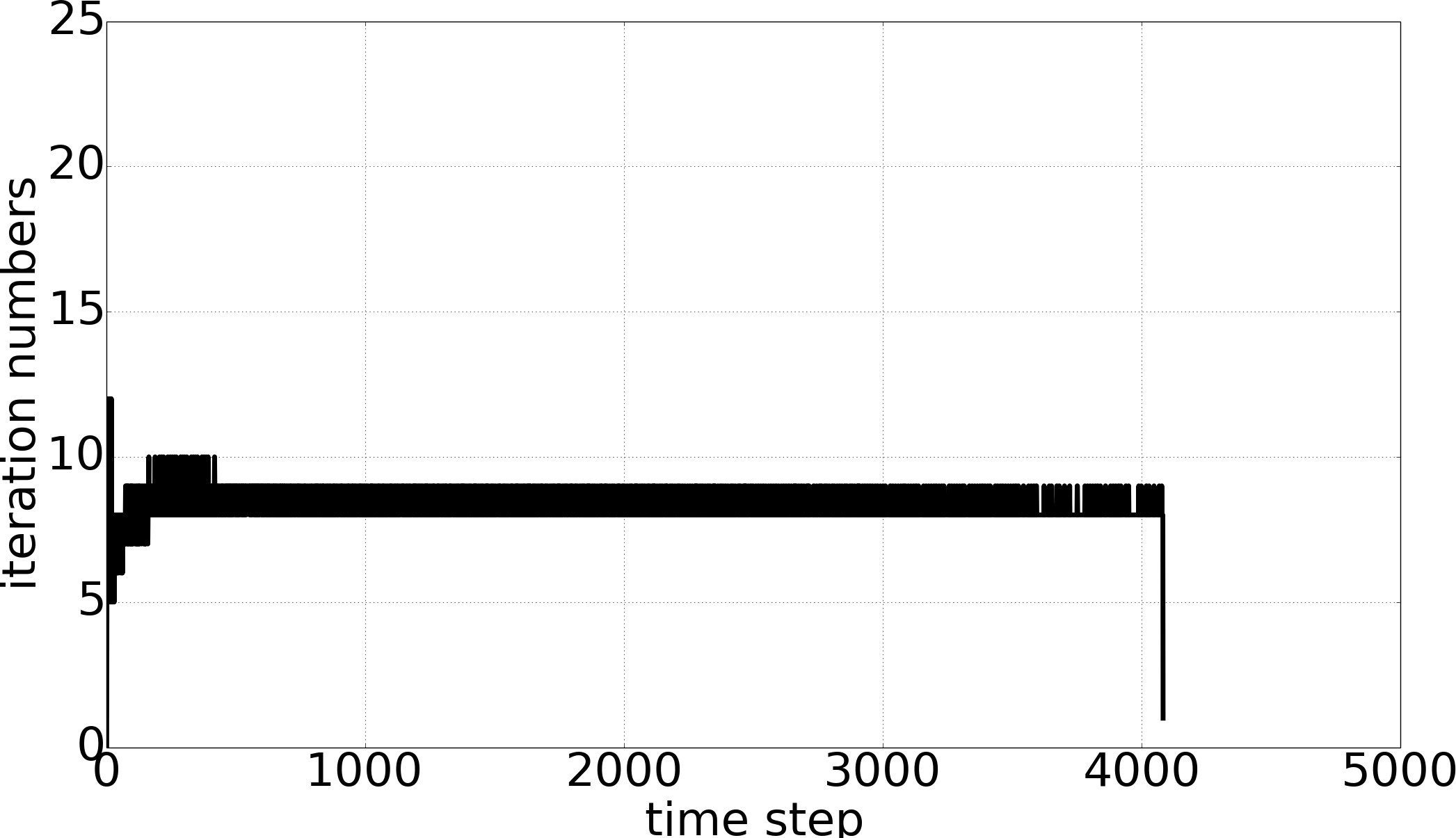} &
\includegraphics[width=0.3\textwidth]{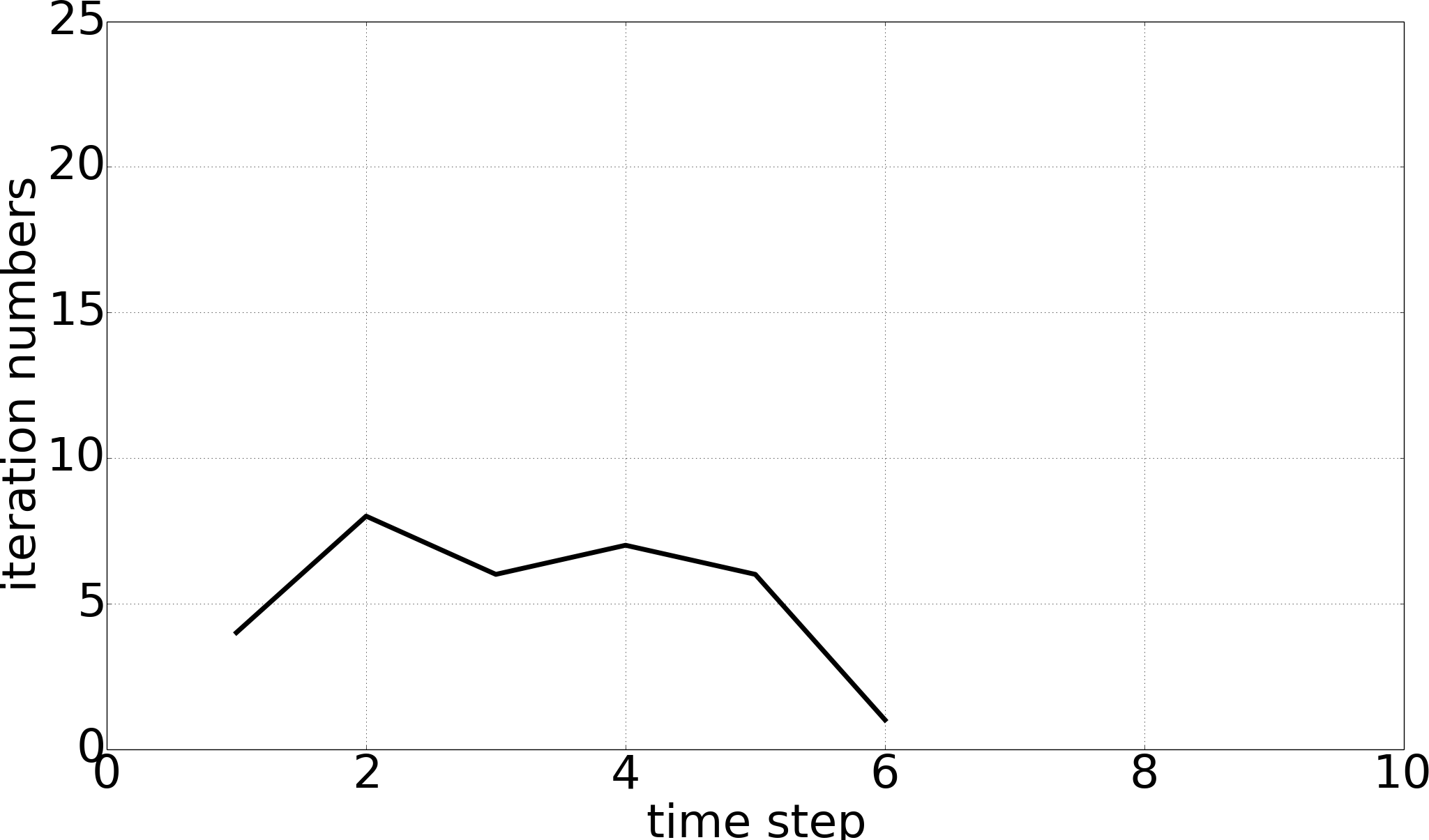} &
\includegraphics[width=0.3\textwidth]{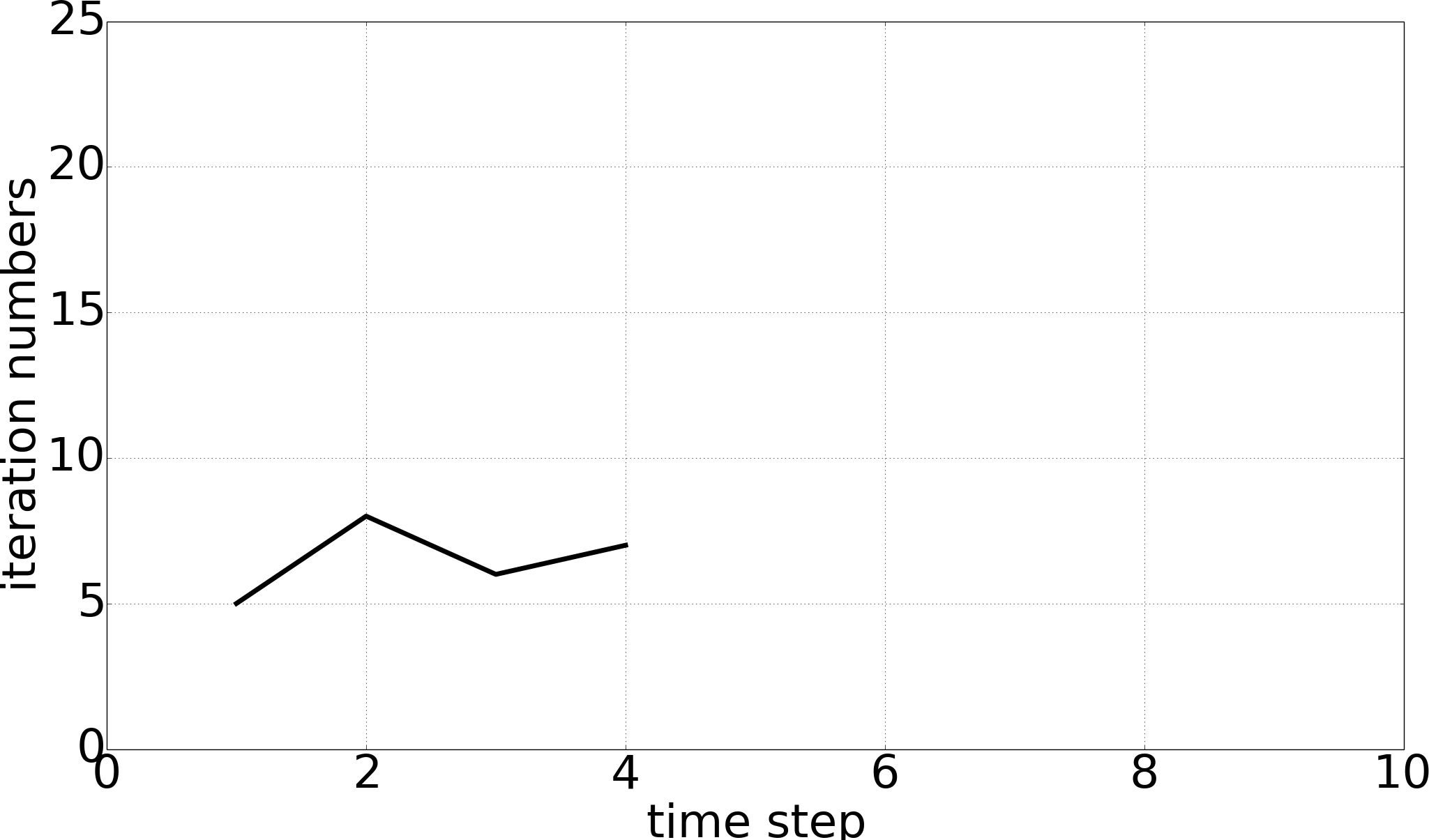} \\
\end{tabularx}
\caption{From left to right $\IQ^2$, $\IQ^4$, $\IQ^6$  DG schemes. Top: the number of bad cells after solving ($\mathrm{P}$) at each time step (the DG polynomial cell averages are not in the admissible set). Bottom: the number of Douglas--Rachford iterations need to reach round-off convergence for solving \eqref{total-energy-opt-2} with \eqref{definition-T-2}. }
\label{fig:astrophysical_jet_DR}
\end{center}
\end{figure}

\section{Concluding remarks}\label{sec-remark}
In this paper, we have constructed a semi-implicit DG scheme that is high order accurate in space, conservative, and positivity-preserving for solving the compressible NS equations. The time step constraint follows the standard hyperbolic CFL condition $\Delta t = \mathcal{O}{(\Delta x)}$. Our scheme is fully decoupled, requiring only the sequential solving of two linear systems at each time step to achieve second order accuracy in time. Conservation and positivity are ensured through a postprocessing of the cell averages of total energy variable. A high order accurate cell average limiter can be formulated as a constraint minimization, which can be efficiently computed by using the generalized Douglas--Rachford splitting method with nearly optimal parameters. Numerical tests suggest that such a simple and efficient postprocessing of the total energy variable indeed renders the semi-implicit high order DG method with Strang splitting much more robust.
Ongoing and future work consists of  extensions from the $\ell^2$-norm minimization postprocessing to the $\ell^1$-norm minimization, and also generalizations to  directly enforcing  the convex invariant domain.

\section*{Acknowledgments}
 Research is supported by NSF DMS-2208515.   

\appendix

\section{The   method of Lagrange multiplier}\label{sec:appendix}
Given a matrix $\vecc{A}=[1,1,\cdots,1]\in\IR^{1\times N}$ and a vector $\vec{w}\in\IR^N$. Define a constant $b = \vecc{A}\vec{w}$ and assume $\sum_{i=1}^N w_i >0$. Let us consider the following constrained minimization problem
\begin{align}\label{eq:opt_model}
\min_{\vec{x}\in\IR^N} \frac{1}{2}\norm{\vec{x} - \vec{w}}{2}^2
~~\text{subject~to}~~
\vecc{A} \vec{x} = b 
~~\text{and}~~
x_i \geq 0 ~~\text{for all}~~ i\in\{1,\cdots,N\}.
\end{align}
Consider the Lagrangian function with multipliers $\lambda_i$ and $\gamma$
\begin{align*}
L = \frac{1}{2}\norm{\vec{x} - \vec{w}}{2}^2 + \gamma \Big(\sum_{i=1}^{N} x_i - b\Big) + \sum_{i=1}^{N} (-\lambda_i x_i),
\end{align*}
and its Karush–Kuhn–Tucker (KKT) conditions, which are given as
\begin{subequations}\label{eq:KKT_condition}
\begin{align}
\frac{\partial L }{\partial x_i} = x_i - w_i + \gamma - \lambda_i &= 0,\label{eq:KKT_1}\\
-\lambda_i x_i &= 0,\label{eq:KKT_2}\\
\lambda_i &\geq 0,\label{eq:KKT_3}\\
-x_i &\leq 0,\label{eq:KKT_4}\\
\sum_{i=1}^{N} x_i &= b.\label{eq:KKT_5}
\end{align}
\end{subequations}
For the constrained minimization problem \eqref{eq:opt_model}, the KKT condition \eqref{eq:KKT_condition} is both sufficient and necessary. 
In the rest of this part, let us assume there exists at least one entry in $\vec{w}$ that is strictly less than $0$. Otherwise, the minimizer of the constraint optimization problem \eqref{eq:opt_model} is $\vec{w}$, which is trivial. 
\begin{lemma}\label{lem:lemma_gamma}
If there exists an entry in $\vec{w}$ less than $0$, then $\gamma\neq0$.
\end{lemma}
\begin{proof}
Assume $\gamma = 0$. Then \eqref{eq:KKT_1} becomes $x_i - w_i - \lambda_i = 0$, namely we have $\lambda_i = x_i - w_i$.
Summing over $i$ from $1$ to $N$, we get
\begin{align*}
\sum_{i=1}^N \lambda_i = \sum_{i=1}^N x_i - \sum_{i=1}^N w_i = b - b = 0.
\end{align*}
Notice \eqref{eq:KKT_3} gives $\lambda_i \geq 0$, we have $\lambda_i = 0$ for all $i$. Thus $x_i = w_i$ for all $i$, which  contradicts the existence of a negative entry in $\vec{w}$.
\end{proof}
\noindent
Let $B = \{j: x_j = 0\}$ denote the set of all indexes, as represented in the minimizer $\vec{x}$ of \eqref{eq:opt_model}, touching the boundary of the feasible region. Let $\tthash B$ be the number of elements in set $B$.
The next lemma shows that if an entry of the vector $\vec{w}$ is less than $0$, then the minimizer plugs that entry back to the boundary of the feasible region. 
\begin{lemma}\label{lem:neg_wi}
Assume there exists at least one entry in vector $\vec{w}$ that is strictly less than $0$.  Then for any index $i$ so that $w_i \leq 0$, we have $i\in B$ and hence $x_i = 0$.
\end{lemma}
\begin{proof}
From \eqref{eq:KKT_2}, we only need to show $\lambda_i > 0$.
By \eqref{eq:KKT_1}, we have $x_i - w_i + \gamma = \lambda_i$. Summing over $i$ from $1$ to $N$, we get
\begin{align*}
\underbrace{\sum_{i=1}^N x_i - \sum_{i=1}^N w_i}_{=~ b-b ~=~ 0} + N\gamma = \sum_{i=1}^N \lambda_i
\quad\Rightarrow\quad
\gamma = \frac{1}{N}\sum_{i=1}^N \lambda_i.
\end{align*}
Thus, by \eqref{eq:KKT_3}, we know $\gamma \geq 0$. Furthermore, by Lemma~\ref{lem:lemma_gamma}, we get $\gamma > 0$. Notice \eqref{eq:KKT_4} gives $x_i \geq 0$. Therefore, under the condition $w_i \leq 0$, the \eqref{eq:KKT_1} implies $\lambda_i = x_i - w_i + \gamma > 0$.
\end{proof}
\begin{lemma}\label{lem:opt_exact_sol_structure}
The solution of the constrained minimization problem \eqref{eq:opt_model} satisfies:
\begin{itemize} 
\item If $x_i = 0$, then we have
\begin{align}\label{eq:B_lambda}
\lambda_i - \frac{1}{N}\sum_{j\in B} \lambda_j = - w_i, \quad \forall i \in B.
\end{align}
\item If $x_i > 0$, then we have
\begin{align}
x_i = w_i + \frac{1}{N - \tthash B}\sum_{j\in B} w_j.
\end{align}
\end{itemize} 
\end{lemma}
\begin{proof}
From \eqref{eq:KKT_1}, we have $\gamma = \lambda_i + w_i - x_i$. Summing over $i$ from $1$ to $N$, we get
\begin{align*}
N\gamma = \sum_{i=1}^N \lambda_i + \underbrace{\sum_{i=1}^N w_i - \sum_{i=1}^N x_i}_{=~ b-b ~=~ 0}
= \sum_{i\in B} \lambda_i + \sum_{i\notin B} \lambda_i.
\end{align*}
Recall that the set $B = \{j:x_j = 0\}$. By \eqref{eq:KKT_4}, $i\notin B$ gives $x_i>0$. By \eqref{eq:KKT_2}, we have $\lambda_i = 0$ for all $i\notin B$. Thus, we get
\begin{align}\label{eq:compute_gamma}
\gamma = \frac{1}{N}\sum_{i\in B} \lambda_i.
\end{align}
If $x_i = 0$, then $i\in B$ and \eqref{eq:KKT_1} becomes $\lambda_i - \gamma = - w_i$, so replacing $\gamma$ with \eqref{eq:compute_gamma}, we obtain \eqref{eq:B_lambda}.
Summing over $i\in B$ of \eqref{eq:B_lambda}, we have
\begin{align*}
\sum_{i\in B} \lambda_i - \frac{\tthash B}{N}\sum_{j\in B} \lambda_j = -\sum_{i\in B} w_i
\quad\Rightarrow\quad 
\sum_{j\in B} \lambda_j = -\frac{N}{N - \tthash B}\sum_{j\in B} w_j.
\end{align*}
If $x_i > 0$, then by \eqref{eq:KKT_2} we have $\lambda_i = 0$. Again, \eqref{eq:KKT_1} and \eqref{eq:compute_gamma} gives 
\begin{align*}
x_i = w_i - \gamma 
= w_i - \frac{1}{N}\sum_{j\in B} \lambda_j 
= w_i + \frac{1}{N - \tthash B}\sum_{j\in B} w_j.
\end{align*}
Therefore, we conclude the proof.
\end{proof}
\begin{lemma}\label{lem:compare_w}
If $w_{i_1} \geq w_{i_2} > 0$, then $x_{i_1} = 0$ implies $x_{i_2} = 0$, namely $i_1\in B$ implies $i_2\in B$. 
\end{lemma}
\begin{proof}
Let us first deal with the case $w_{i_1} > w_{i_2}$.
If the vector $\vec{x}$ is a solution of the minimization problem \eqref{eq:opt_model} with $x_{i_1} = 0$ and $x_{i_2} > 0$, then we will construct a solution vector $\tilde{\vec{x}}$ such that $\tilde{x}_i = x_i$ for all $i\notin\{i_1,i_2\}$ and $\tilde{x}_{i_1} = x_{i_2}$ and $\tilde{x}_{i_2} = 0$.
\begin{itemize}
\item \emph{Check constraint}: since $\tilde{x}_i = x_i$ for all $i\notin\{i_1,i_2\}$, we only need to check $\tilde{x}_{i_1} + \tilde{x}_{i_2} = x_{i_1} + x_{i_2}$. This holds since $\tilde{x}_{i_1} + \tilde{x}_{i_2} = x_{i_2} + 0$ and $x_{i_1} + x_{i_2} = 0 + x_{i_2}$.
\item \emph{Compare $2$-norm}: we have $(w_{i_1} - x_{i_1})^2 + (w_{i_2} - x_{i_2})^2 > (w_{i_1} - \tilde{x}_{i_1})^2 + (w_{i_2} - \tilde{x}_{i_2})^2$, which can be easily verified as follows
\begin{align*}
&& (w_{i_1} - x_{i_1})^2 + (w_{i_2} - x_{i_2})^2 &> (w_{i_1} - \tilde{x}_{i_1})^2 + (w_{i_2} - \tilde{x}_{i_2})^2\\
\Leftrightarrow&&  
w_{i_1}^2 + (w_{i_2} - x_{i_2})^2 &> (w_{i_1} - x_{i_2})^2 + w_{i_2}^2\\ 
\Leftrightarrow&&
w_{i_1}^2 + w_{i_2}^2 - 2w_{i_2}x_{i_2} + x_{i_2}^2 &> w_{i_1}^2 - 2w_{i_1}x_{i_2} + x_{i_2}^2 + w_{i_2}^2 \\ 
\Leftrightarrow&&
(w_{i_1} - w_{i_2})x_{i_2} &> 0,
\end{align*}
which holds when $w_{i_1} > w_{i_2}$ and $x_{i_2} > 0$.
\end{itemize}
Hence we have constructed a vector $\tilde{\vec{x}}$ that satisfies the constraint but has smaller objective value, which contradicts that $\vec{x}$ is the unique minimizer of \eqref{eq:opt_model}.
\par
In case of $w_{i_1} = w_{i_2}$, we use contradiction argument to show the vector $\vec{x}$ with $x_{i_1} = 0$ and $x_{i_2} > 0$ is not a solution of the minimization problem \eqref{eq:opt_model}.
We construct a vector $\tilde{\vec{x}}$ such that $\tilde{x}_i = x_i$ for all $i\notin\{i_1,i_2\}$ and $\tilde{x}_{i_1} = \frac{1}{2}x_{i_2}$ and $\tilde{x}_{i_2} = \frac{1}{2}x_{i_2}$.
\begin{itemize}
\item \emph{Check constraint}: since $\tilde{x}_i = x_i$ for all $i\notin\{i_1,i_2\}$, we only need to check $\tilde{x}_{i_1} + \tilde{x}_{i_2} = x_{i_1} + x_{i_2}$. This holds since $\tilde{x}_{i_1} + \tilde{x}_{i_2} = x_{i_2}$ and $x_{i_1} + x_{i_2} = x_{i_2}$.
\item \emph{Compare $2$-norm}: we have $(w_{i_1} - x_{i_1})^2 + (w_{i_2} - x_{i_2})^2 > (w_{i_1} - \tilde{x}_{i_1})^2 + (w_{i_2} - \tilde{x}_{i_2})^2$, which can be easily verified as follows
\begin{align*}
&&(w_{i_1} - x_{i_1})^2 + (w_{i_2} - x_{i_2})^2 &> (w_{i_1} - \tilde{x}_{i_1})^2 + (w_{i_2} - \tilde{x}_{i_2})^2\\
\Leftrightarrow&&
w_{i_1}^2 + (w_{i_2} - x_{i_2})^2 &> (w_{i_1} - \frac{1}{2}x_{i_2})^2 + (w_{i_2} - \frac{1}{2}x_{i_2})^2\\ 
\Leftrightarrow&&
w_{i_1}^2 - (w_{i_1} - \frac{1}{2}x_{i_2})^2 &> (w_{i_2} - \frac{1}{2}x_{i_2})^2 - (w_{i_2} - x_{i_2})^2\\ 
\Leftrightarrow&&
x_{i_2} (2w_{i_1} - \frac{1}{2}x_{i_2}) &> x_{i_2} (2w_{i_2} - \frac{3}{2}x_{i_2})\\ 
\quad\Leftrightarrow&&
w_{i_1} - w_{i_2} &> -\frac{1}{2} x_{i_2} 
\end{align*}
This hold when $w_{i_1} = w_{i_2}$ and $x_{i_2} > 0$.
\end{itemize}
The proof is now concluded.
\end{proof}
\par
The Lemma~\ref{lem:neg_wi} indicates the following: if the $i$-th entry of the vector $\vec{w}$ is non-positive, $w_i\leq0$, then we need to set $x_i = 0$.
Lemma~\ref{lem:opt_exact_sol_structure} gives the structure of the exact solution to the minimization problem \eqref{eq:opt_model}. 
Lemma~\ref{lem:compare_w} helps us to construct the following algorithm to find the set $B$ and obtain the solution to \eqref{eq:opt_model}. 
\begin{itemize}
\item Step~1. If $w_i\leq0$, then set $x_i = 0$ and push $i$ in set $B$.
\item Step~2. \textbf{Sort all entries} $w_i>0$ in $\vec{w}$ in ascending order.
\item Step~3. Compute the ``total out-of-bound mass'' of set $B$ by the following formula:
\begin{align*}
-\sum_{j\in B} w_j.
\end{align*}
\item Step~4. Check whether the smallest $w_{i_s}$, where $i_s\notin B$, satisfies
\begin{align}\label{eq:check_smallest_entry}
w_{i_s} + \frac{1}{N - \tthash B}\sum_{j\in B} w_j > 0.
\end{align}
If \eqref{eq:check_smallest_entry} holds, then uniformly allocate the ``total out-of-bound mass'' of set $B$ to all other ``good entries'' by formula
\begin{align*}
x_i = w_{i} + \frac{1}{N - \tthash B}\sum_{j\in B} w_j \quad\text{for all}~i\notin B.
\end{align*}
Otherwise, push $i_s$ into set $B$ and go to Step~3. Note: if there are multiple entries with the same smallest value, then push all of them into set $B$.
\end{itemize}
The complexity of sorting algorithm $\mathtt{std::sort()}$ in C$++$ is $\mathcal{O}(N\log(N))$ on the best and average case scenarios. Additionally, sophisticated coding skills are required for implementing sorting algorithms on a distributed memory system.
In comparison, the complexity of the DR algorithm is $\mathcal{O}(N)$ and DR is very amenable to parallelization.

\paragraph{\bf Comparison of optimization algorithms}
We create synthetic data to let $\vec w$ in \eqref{eq:opt_model} be defined as point values of the following function on a uniform grid of size $1000^2$ on the domain $[0,1]^2$:
\[ f(x,y) = \begin{cases}
-0.5, & -\frac{\delta}{4}+0.25\leq x\leq \frac{\delta}{4}+0.25 \\
-0.5, & -\frac{\delta}{4}+0.75\leq x\leq \frac{\delta}{4}+0.75 \\
    \cos^8{(2\pi x)} + 10^{-13}, &\mbox{otherwise}
\end{cases},\]
 where $\delta>0$ is a parameter. A different value of $\delta$ gives a different ratio of negative point values, and we consider values of $\delta$ such that the ratio of negative point values is $1\%, 2\%, 5\%, 10\%$ and $20\%$.

\par
We then solve \eqref{eq:opt_model} with $b = \vecc{A}\vec{w}$ by both the   method of Lagrange multiplier and the Douglas--Rachford method. 
For each optimization method, we solve \eqref{eq:opt_model} to machine precision $100$ times and compare the average CPU time for solving it once on a single   Intel Xeon CPU E5-2660 v3 $2.60$GHz. 
The Table~\ref{tab:compare_opt_algorithm} shows the computational time of finding the minimizer up to machine precision.
\par
The time cost of the DR algorithm increases as the ratio of negative points increases, which  is however still faster than the  Lagrange multiplier approach for large data set, due to the $\mathcal{O}(N\log(N))$ sorting operation. 
Notice that as the number of negative points increases, the data set requiring sorting becomes smaller, resulting in a decrease in the time cost of the   Lagrange multiplier method. However, in large-scale simulations with a good base scheme such as a proper DG scheme in this paper, the percentage of negative points is typically small. 
Such a comparison suggests that the DR algorithm is a preferable option from the efficiency perspective.
\begin{table}[ht!]
\centering
\begin{tabularx}{0.75\linewidth}{@{~}C@{~}|C@{~}|C@{~}|C@{~}|C@{~}|C@{~}}
\toprule
   bad cells \% & $1\%$ & $2\%$ & $5\%$ & $10\%$ & $20\%$ \\
\midrule
LM & $1.426\,\mathrm{s}$ & $1.500\,\mathrm{s}$ & $1.509\,\mathrm{s}$ & $1.418\,\mathrm{s}$ & $1.130\,\mathrm{s}$\\
DR & $0.378\,\mathrm{s}$ & $0.467\,\mathrm{s}$ & $0.565\,\mathrm{s}$ & $0.656\,\mathrm{s}$ & $0.846\,\mathrm{s}$\\
\bottomrule
\end{tabularx}
\caption{The CPU time for applying the method of Lagrange multiplier and the DR algorithm to solve the minimization \eqref{eq:opt_model} for a problem of size $10^6$ for problems with different ratios of negative points (bad cells). The time unit is second. The ``LM'' refers to the method of Lagrange multiplier and the ``DR'' refers to the Douglas--Rachford splitting algorithm.}
\label{tab:compare_opt_algorithm}
\end{table}

\bibliographystyle{elsarticle-num}
\bibliography{bibliography} 

\end{document}